\documentclass[graybox]{svmult}

\usepackage{type1cm}        % activate if the above 3 fonts are
\usepackage{makeidx}         % allows index generation
\usepackage{graphicx}        % standard LaTeX graphics tool
\usepackage{multicol}        % used for the two-column index
\usepackage[bottom]{footmisc}% places footnotes at page bottom

\usepackage{newtxtext}       % 
\usepackage[varvw]{newtxmath}       % selects Times Roman as basic font

\usepackage[margin=1in]{geometry}
\usepackage[useregional]{datetime2}
\usepackage[english]{babel}
\usepackage[utf8]{inputenc}
\usepackage[T1]{fontenc}
\usepackage{xcolor, graphicx,changepage, float, array, amsmath, mathrsfs, amsfonts, amsthm, tikz-cd, fancyhdr, enumerate, bbm, multicol,cite,refcount}
\usetikzlibrary{positioning}

\usepackage[pagebackref,colorlinks=true,urlcolor=blue,linkcolor=blue,citecolor=blue]{hyperref}
\usepackage[capitalize]{cleveref}

\makeatletter
\newcommand*{\rom}[1]{\expandafter\@slowromancap\romannumeral #1@}
\makeatother

\makeatletter
\renewcommand*\env@matrix[1][*\c@MaxMatrixCols c]{%
	\hskip -\arraycolsep
	\let\@ifnextchar\new@ifnextchar
	\array{#1}}
\makeatother

\theoremstyle{definition}

\usepackage{setspace}
\newcommand{\N}{\mathbb{N}}
\newcommand{\Z}{\mathbb{Z}}

\newcommand{\C}{\mathbb{C}}
\renewcommand{\d}{~\mathrm{d}}

\newcommand{\XX}{\mathcal{X}}
\newcommand{\YY}{\mathcal{Y}}
\newcommand{\ZZ}{\mathcal{Z}}

\newcommand{\plimgG}[3]{\mathop {#1\text{-lim}}_{#2\in #3}\,}

\definecolor{ggreen}{RGB}{0,200,0}
\definecolor{rred}{RGB}{150,0,70}
\definecolor{yyellow}{RGB}{250,210,0}
\newcommand{\edit}[3]{\color{#1}{#3}\color{black}\marginpar{\textcolor{#1}{[[#2]]}}}

\renewcommand{\sf}[1]{\edit{yyellow}{SF}{[[#1]]}}

\usepackage[normalem]{ulem}

\renewcommand{\emph}[1]{\textbf{#1}}

\usepackage[switch,mathlines]{lineno}
\makeindex             % used for the subject index
\title*{Multipliers and Disjointness from Mixing
}
\author{Sohail Farhangi,\\ Joel Moreira and\\ Rigoberto Zelada}
\institute{Sohail Farhangi \at Beijing Institute of Mathematical Sciences and Applications, Beijing, P.R.C. 101408, \email{sohail.farhangi@gmail.com}
\and Joel Moreira \at University of Warwick \email{joel.moreira@warwick.ac.uk} \and Rigoberto Zelada \at University of Warwick \email{Rigoberto.Zelada@warwick.ac.uk}}
\begin{document}

\maketitle

\abstract*{In 2005, Parreau proved that if a measure-preserving system is not strongly mixing then it contains a nontrivial factor that is disjoint from every strongly mixing system. 
Taking this construction as the starting point, we develop the complementary notions of $\mathcal U$-generated and $\mathcal U$-mixing systems, for a set $\mathcal U$ of ultrafilters, and use them to recover several classical results in ergodic theory as special cases of a unified framework.\\
We prove that a system is $\mathcal U$-mixing if and only if it is disjoint from all $\mathcal U$-generated systems.
In fact, we show that if $\mathcal Y$ is a $\mathcal U$-generated system and $\mathcal Z$ is disjoint from every $\mathcal U$-mixing system, then any joining of $\mathcal Y$ and $\mathcal Z$ remains disjoint from all $\mathcal U$-mixing systems.
We also show that every partially rigid system is a finite extension of some $\mathcal{U}$-generated system.
}

\abstract{In 2005, Parreau proved that if a measure-preserving system is not strongly mixing then it contains a nontrivial factor that is disjoint from every strongly mixing system. 
Taking this construction as the starting point, we develop the complementary notions of $\mathcal U$-generated and $\mathcal U$-mixing systems, for a set $\mathcal U$ of ultrafilters, and use them to recover several classical results in ergodic theory as special cases of a unified framework.\\
We prove that a system is $\mathcal U$-mixing if and only if it is disjoint from all $\mathcal U$-generated systems.
In fact, we show that if $\mathcal Y$ is a $\mathcal U$-generated system and $\mathcal Z$ is disjoint from every $\mathcal U$-mixing system, then any joining of $\mathcal Y$ and $\mathcal Z$ remains disjoint from all $\mathcal U$-mixing systems.
We also show that every partially rigid system is a finite extension of some $\mathcal{U}$-generated system.
}

\section{Introduction}
%The study of disjointness of measure-preserving systems began with the seminal work of Furstenberg \cite{DisjointnessInErgodicTheory} in 1967.
Disjointness of measure-preserving systems, introduced in the seminal work of Furstenberg \cite{DisjointnessInErgodicTheory} in 1967, can be used to characterize many classical dualities in ergodic theory.
%Many classical dichotomies in ergodic theory can be understood from the perspective of disjointness.
In particular, we have

\begin{theorem}[{cf. \cite[Remark 10]{ErgodicGaussianSelfJoinings} or \cite[Page 145]{ErgodicTheoryViaJoinings}}]\label{thm:ClassicalDichotomies}
    Let $\mathcal{X} := (X,\mathscr{B},\mu,T)$ be a measure-preserving system.
    \begin{enumerate}[(i)]
        \item $\mathcal{X}$ is ergodic if and only if it is disjoint from every identity system.

        \item $\mathcal{X}$ is weakly mixing if and only if it is disjoint from every rotation on a compact abelian group.

        \item $\mathcal{X}$ is mildly mixing if and only if it is disjoint from every rigid system.

        \item $\mathcal{X}$ has completely positive entropy if and only if it is disjoint from every zero entropy system.
    \end{enumerate}
\end{theorem}

We recall that if the systems $\mathcal{X}$ and $\mathcal{Y}$ have a nontrivial common factor $\mathcal{Z}$, then the relatively independent joining of $\mathcal{X}$ and $\mathcal{Y}$ over $\mathcal{Z}$ is a nontrivial joining, so $\mathcal{X}$ and $\mathcal{Y}$ are not disjoint.
However, it is possible for systems $\mathcal{X}$ and $\mathcal{Y}$ not to have a common factor, but still have a nontrivial joining, i.e., to not be disjoint (cf. \cite[Example 7.15]{ErgodicTheoryViaJoinings}).

A class $\mathcal{C}$ of measure-preserving systems is a \textbf{characteristic class} if it is closed under factors and countable joinings.
The classes $\mathcal{I}$ of invariant systems, $\mathcal{K}$ of Kronecker systems, and $\mathcal{Z}_e$ of zero entropy systems are all characteristic classes.
%A treatment of the basic properties of characteristic classes is given in \cite{ArithmeticFunctionsOrthogonalToDeterminism}.
%It turns out that a system $\mathcal{X}$ is disjoint from every system in a characteristic class $\mathcal{C}$ if and only if no nontrivial factor of $\mathcal{X}$ belongs to  $\mathcal{C}$ (cf. \cite[Proposition 2.2]{ArithmeticFunctionsOrthogonalToDeterminism}).
%A treatment of the  is given in  
In \cite{ArithmeticFunctionsOrthogonalToDeterminism} some basic properties of characteristic classes are studied, and in particular it is proved that any nontrivial characteristic class $\mathcal{C}$ satisfies $\mathcal{I} \subseteq \mathcal{C} \subseteq \mathcal{Z}_e$.
It follows from \cite[Proposition 2.2]{ArithmeticFunctionsOrthogonalToDeterminism} that a system $\mathcal{X}$ is disjoint from every system in a characteristic class $\mathcal{C}$ if and only if no nontrivial factor of $\mathcal{X}$ belongs to  $\mathcal{C}$.
On the other hand, while the class $\mathcal{R}$ of rigid systems is not a characteristic class, it is still the case that a system $\mathcal{X}$ is mildly mixing if and only if it has no rigid factors.\footnote{Given a sequence $(n_k)_{k = 1}^\infty$ in $\mathbb N$, one has that the class $\mathcal{R}(n_k)_{k = 1}^\infty$ of systems that are rigid along $(n_k)_{k = 1}^\infty$ is a characteristic class.}

Another important class of systems in ergodic theory is the class $\mathcal{S}$ of strongly mixing systems.
The question of whether strongly mixing systems can be characterized as those disjoint from a certain class, in the spirit of \cref{thm:ClassicalDichotomies}, remained open until the 2000s. 
A positive answer was essentially found by Parreau \cite{ParreauFactor} (see also \cite[Theorem 11]{Lemanczyk2009Spectral}), who proved that any system that is not strongly mixing contains a nontrivial factor disjoint from all mixing systems.
%For example, it is still a major open problem to determine whether or not every strongly mixing system is also mixing of all orders \cite{Rohlin1949}.
%Later, Lemanczyk observed that several consequences can be derived from this result \jpm{cite his notes?}
In the present paper we explicitly describe and study a class of systems, which we denote by $\mathcal{G}$ (cf. \cref{def:ParreauFactor}), which, in view of Parreau's theorem, allows one to characterize strongly mixing systems in terms of disjointness as in \cref{thm:ClassicalDichotomies}.
The following is an immediate consequence of \cite{ParreauFactor}.
\begin{theorem}\label{thm:ParreauFactorForZ}
A system $\mathcal{X}$ is strongly mixing if and only if it is disjoint from every system from $\mathcal{G}$.
\end{theorem}
We will see in Lemma \ref{lem:ParreauFactorsAreClosedUnderJoinings} that the class $\mathcal{G}$ is closed under countable joinings, but Examples \ref{ex:SecondExample}-\ref{ex:FourthExample} will show that it need not be closed under passage to factors, so $\mathcal{G}$ is not a characteristic class.
Nonetheless, it is still the case that $\mathcal{X}$ is strongly mixing if and only if it has no factor in $\mathcal{G}$.
Indeed, we will see in Corollary \ref{cor:ClosedUnderSubfactors} below that 
 $\mathcal G$ is ``closed under subfactors'', meaning that for any system $\mathcal{X} \in \mathcal{G}$, and any nontrivial factor $\mathcal{Y}$ of $\mathcal{X}$, there is a nontrivial factor $\mathcal{Z}$ of $\mathcal{Y}$ for which $\mathcal{Z} \in \mathcal{G}$.
 As discussed in Remark \ref{rem:QuasiCharacteristicClass}, it follows that the class $\overline{\mathcal{G}}$ of systems that are a factor of some member of $\mathcal{G}$ is a characteristic class.

The classical notion of disjointness is often called Furstenberg disjointness in order to distinguish it from the stronger notion of \textbf{spectral disjointness} (defined in Section \ref{ssec:Spectrum}).
It is worth noting that parts (i)-(iii) of Theorem \ref{thm:ClassicalDichotomies} are still true if the word ``disjoint'' is replaced by ``spectrally disjoint''.
In particular, the mean ergodic theorem can be reinterpreted as the spectral disjointness of identity systems $\mathcal{I}$ and ergodic systems $\mathcal{E}$, the Jacobs-de Leeuw-Glicksberg decomposition can be reinterpreted as the spectral disjointness of rotations on compact abelian groups $\mathcal{K}$ and weakly mixing systems $\mathcal{W}$, 
and the spectral disjointness of rigid systems $\mathcal{R}$ and mildly mixing systems $\mathcal{M}$ can be deduced from 
\cite[Lemma 1]{Walters1972} (see also \cite[Theorem 0.4]{PolynomialActionsOfUnitaryOperators}).
All systems $\mathcal{X}$ of completely positive entropy have countable Lebesgue spectrum \cite{KMixingImpliesLebesgueSpectrumAmenable}, and many systems of zero entropy also have countable Lebesgue spectrum, so the dichotomy present in Theorem \ref{thm:ClassicalDichotomies} (iv) cannot be expressed through spectral disjointness.
We will see in Examples \ref{ex:FirstExample}-\ref{ex:FourthExample} that systems in $\mathcal{G}$ need not be spectrally disjoint from all strongly mixing systems. 
Nonetheless, there is a strong relation between strong mixing, $\mathcal{G}$, and spectral disjointness, described in Theorem \ref{thm:Disjointness}. 
%\sf{Now we have to bring back the vague sentence to show why this special case is worth distinguishing from Lemanczyks arguably more general result.}
%\jpm{What is Lemanczyk's result here?}
%Nonetheless, one can generate the members of $\mathcal{G}$ via the notion of spectral disjointness.%as those systems $\mathcal{X}$ that are "generated" by the vectors in $L^\infty(X,\mu)$ that are orthogonal to all strongly mixing vectors.

Another important topic in the theory of disjointness is that of multipliers. 
Given a class $\mathscr{D}$ of systems, we let $\mathscr{D}^\perp$ denote the class of systems that are disjoint from all systems in $\mathscr{D}$.
We let $\mathscr{M}(\mathscr{D}^\perp)$ denote the class of systems $\mathcal{X}$, such that for any $\mathcal{Y} \in \mathscr{D}^\perp$, every joining of $\mathcal{X}$ and $\mathcal{Y}$ is again in $\mathscr{D}^\perp$.\footnote{We mention that in many other papers in the literature, the members of $\mathscr{D}^\perp$ and $\mathscr{M}(\mathscr{D}^\perp)$ are required to be ergodic.}
For any class of systems $\mathscr{D}$, we have that $\mathscr{M}(\mathscr{D}^\perp)$ is a characteristic class \cite[Section 2.1]{ArithmeticFunctionsOrthogonalToDeterminism}.

Letting $\mathcal{D}$ denote the class of measurably distal systems, Furstenberg \cite{DisjointnessInErgodicTheory} had shown that $\mathcal{D}\cap\mathcal{E} \subseteq \mathcal{W}^\perp$ and $\mathcal{W} = \mathcal{K}^\perp$.
However, we are far from a complete understanding of $\mathcal{W}^\perp$. 
Glasner and Weiss \cite{MultWPerpBiggerThanDist} showed that $\mathcal{D}\cap\mathcal{E} \subsetneq \mathcal{W}^\perp$. 
Later, Glasner \cite{WPerpAndItsMultipliers} showed that $\mathcal{D}\cap\mathcal{E} \subsetneq \mathscr{M}(\mathcal{W}^\perp)$, then Lema\'nczyk and Parreau \cite{RokhlinExtensionsLiftingDisjointness} showed that $\mathscr{M}(\mathcal{W}^\perp) \subsetneq \mathcal{W}^\perp$.
Recently Berk, G\'orska, and de la Rue \cite{ErgodicPerp} showed that $\mathcal{I} = \mathscr{M}(\mathcal{E}^\perp) \subsetneq \mathcal{E}^\perp$, and a complete characterization of $\mathcal{E}^\perp$ was obtained by G\'orska, Lema\'nczyk, and de la Rue \cite{UniquelyErgodicPerp}.

In this paper we establish a general result, \cref{thm:ParreauFactorsAreMultipliers} below, which contains as special cases the facts that $\mathcal{G} \subseteq \mathscr{M}(\mathcal{S}^\perp)$ and $\mathcal{R} \subseteq\mathscr{M}(\mathcal{M}^\perp)$ (which were originally obtained by Lema\'nczyk), as well as the classical fact that $\mathcal{K} \subseteq \mathscr{M}(\mathcal{W}^\perp)$.
%In the present paper, we show in Theorem \ref{thm:ParreauFactorsAreMultipliers} that $\mathcal{G} \subseteq \mathscr{M}(\mathcal{S}^\perp)$.
%Theorem \ref{thm:ParreauFactorsAreMultipliers} is general enough that it also shows that $\mathcal{R} \subseteq\mathscr{M}(\mathcal{M}^\perp)$, where $\mathcal{M}$ is the class of mildly mixing systems, and it gives a new proof of the classical fact that $\mathcal{K} \subseteq \mathscr{M}(\mathcal{W}^\perp)$.
% \rsout{We also show in Theorem \ref{SpectralDisjointnessImpliesMultiplier} that in many cases of interest, spectral disjointness implies the multiplier property, which gives a second proof of the fact that $\mathcal{R} \subseteq \mathscr{M}(\mathcal{M}^\perp)$.
% Theorem \ref{SpectralDisjointnessImpliesMultiplier} also allows us to see that the class of systems of singular spectrum is contained in $\mathscr{M}(\mathcal{A}^\perp)$, where $\mathcal{A}$ is the class of systems with absolutely continuous spectrum.
% We also show that any partially rigid system is a finite extension of some member of $\mathcal{G}$ (see Theorem \ref{thm:PartiallyRigidFiniteFibers}), from which we deduce that partially rigid systems, such as substitution systems and interval exchange transformations (IET), are also members of $\mathscr{M}(\mathcal{S}^\perp)$.}
%%%%%%%%%%%%%%%%%%%%%%%%%%%%%%%%%%%%%%%%
\subsection{A unified approach to many ergodic-theoretical dichotomies}
%\jpm{Add in each theorem/lemma that was known before the relevant reference}
A goal of this paper is to, using the language of limits along ultrafilters and properties of normal operators, present a unified approach to several structural results about measure-preserving actions of countably infinite abelian groups.
In particular, we prove a variant of \cref{thm:ClassicalDichotomies} dealing with measure-preserving actions of countably infinite abelian groups, which has as corollaries
\cref{thm:ClassicalDichotomies}(ii)-(iii) and Theorem \ref{thm:ParreauFactorForZ}. 
Before stating this variant (see Theorem \ref{thm:Disjointness} below), we need to introduce various definitions and notation.\\

Let $(G,+)$ be a countably infinite abelian group which we tacitly assume to be endowed with the discrete topology. 
Recall that one can view $\beta G$, the Stone-{\v C}ech compactification of $G$, as  a compact Hausdorff space having $G$ as a dense subset. 
We set $G^*:=\beta G\setminus G$. 
Consider a Hilbert space $\mathcal H$ and a unitary representation $(U^g)_{g\in G}$ of $G$ on $\mathcal H$. 
By the universal property of $\beta G$, for every $p\in\beta G$, there exists a unique bounded operator $U^p:\mathcal H\rightarrow\mathcal H$ with the property that for any $f,f'\in\mathcal H$, the unique continuous extension of the map $g\mapsto \langle f,U^g f'\rangle$ from  $G$ to $\beta G$ is given by $p\mapsto \langle f, U^pf'\rangle$. 
We remark in passing that, while the use of convergence along ultrafilters in this paper could be replaced by other convergence schemes (see \cite[Section 9.2]{FBook}, \cite[Section 6]{AlmostMixingOfAllOrders}, for example), the ultrafilter formalism provides a canonical way to avoid cumbersome notation and connect the algebraic properties of $G$ with those of the semigroup formed by the weak accumulation points of the sequence of operators $(U^g)_{g\in G}$ (see Subsection \ref{sssec:UltrafiltersAndUnitaryReps}).
See Section \ref{ssec:Ultrafilters} for a more detailed discussion of ultrafilters and the algebra of $\beta G$.

Given a measure-preserving system $\mathcal X=(X,\mathscr{B},\mu,(T^g)_{g\in G})$ and a nonempty family $\mathcal{U} \subseteq G^*$, we will let $\mathscr{B}_\mathcal U$ denote the smallest sub-$\sigma$-algebra of $\mathscr{B}$ with the property that for every $p\in\mathcal U$ and every $f\in L^2(X,\mu)$, $T^pf$ is $\mathscr{B}_\mathcal U$-measurable. 
One can show that the quadruple $\mathcal X^\mathcal U=(X,\mathscr{B}_\mathcal U,\mu,(T^g)_{g\in G})$ is an invertible measure-preserving system which we call the $\mathcal U$-generated factor of $\mathcal X$ (see Lemma \ref{lem:StabilityOfParreauFactors} below).
 We note that $\mathcal{U}$-generated factors are also related to spectral disjointness.
For any system $\mathcal{X}$ and any $\mathcal{U} \subseteq G^*$, we have the Hilbert space decomposition 

  $$L^2(X,\mu) = \underbrace{\overline{\text{Span}_{\mathbb{C}}\left(\bigcup_{p \in \mathcal{U}}\text{Im}(T^p)\right)}}_{\mathcal{U}-\text{generated}}\oplus\underbrace{\bigcap_{p \in \mathcal{U}}\text{Ker}(T^p)}_{\mathcal{U}-\text{mixing}}.$$
  It turns out that the $\mathcal{U}$-generated factor of $\mathcal{X}$ corresponds to the smallest $\sigma$-algebra with respect to which every function in the $\mathcal{U}$-generated part of $L^2(X,\mu)$ is measurable.

We are now in position to formulate the variant of \cref{thm:ClassicalDichotomies} mentioned above.
\begin{theorem}\label{thm:Disjointness}
    Let $(G,+)$ be a countable abelian group, let $(X,\mathscr{B},\mu,(T^g)_{g\in G})$ be a measure-preserving system, and let $\mathcal U\subseteq G^*$ be non-empty.   The following statements are equivalent:
    \begin{enumerate}[(i)]
    \item $(L_0^2(X,\mu),(T^g)_{g \in G})$ is spectrally disjoint from every pair $(\mathcal H,(U^g)_{g\in G})$
    satisfying 
    $$
\mathcal H=\overline{\text{Span}_{\mathbb C}\left(\bigcup_{p\in\mathcal U} \text{Im}(U^p)\right)}.
    $$
    Here $L_0^2(X,\mu):=\{f\in L^2(X,\mu)\,|\,\int_Xf\text{d}\mu=0\}$.
     \item $(X,\mathscr{B},\mu,(T^g)_{g\in G})$ is Furstenberg disjoint from every $\mathcal U$-generated system. 
      \item $(X,\mathscr{B},\mu,(T^g)_{g\in G})$ is $\mathcal U$-mixing, meaning that for every $A,B\in\mathscr{B}$ and every $p\in\mathcal U$, $\plimgG{p}{g}{G}\mu(A\cap T^{-g}B)=\mu(A)\mu(B)$ (or, equivalently, $\int_X \mathbbm 1_A T^p\mathbbm 1_B\text{d}\mu=\mu(A)\mu(B)$).%\rz{$\int_X \mathbbm 1_A T^p\mathbbm 1_B\text{d}\mu=\mu(A)\mu(B)$? We have not yet defined $\plimgG{p}{g}{G}$.}
    \end{enumerate}
\end{theorem}

\begin{remark}
It is worth pointing out that the implication  (i)$\implies$(ii) does not admit the following variant: If  $(L_0^2(X,\mu),(T^g)_{g \in G})$ is spectrally disjoint from every pair $(\mathcal H,(U^g)_{g\in G})$ having singular spectrum, then $(X,\mathscr{B},\mu,(T^g)_{g\in G})$ is Furstenberg disjoint from every system whose $\sigma$-algebra is generated by its vectors with singular spectrum.
To see this, consider a probability measure $\sigma$ on $[0,1]$ that is singular to Lebesgue, but for which $\sigma*\sigma$ is in the same measure class as Lebesgue.
If $\mathcal{X}$ is the Gaussian system constructed from $\sigma$, then the even factor $\mathcal{X}_e$ of $\mathcal{X}$ has countable Lebesgue spectrum, and $\mathcal{X}$ is a 2-point extension of $\mathcal{X}_e$.
Furthermore, there exists $f \in L^2(X,\mu)$ whose spectral measure is singular, and for which the smallest $T$-invariant $\sigma$-algebra with respect to which $f$ is measurable is all of $\mathscr{B}$.
We see that $\mathcal{X}_e$ is spectrally disjoint from every system with singular spectrum, but it is not disjoint from $\mathcal{X}$ even though $\mathscr{B}$ is generated by a single function with singular spectrum.
\end{remark}

\begin{corollary}\label{cor:ClosedUnderSubfactors}
    For any nontrivial $\mathcal U$-generated system $\mathcal X$ and any nontrivial factor $\mathcal Y$ of $\mathcal X$, one has that $\mathcal Z:=\mathcal Y^\mathcal U$ (which is itself a $\mathcal U$-generated system) is  a nontrivial factor of $\mathcal Y$.
\end{corollary}
\begin{proof}[Proof of Corollary \ref{cor:ClosedUnderSubfactors}]
Because $\mathcal Y$ is not Furstenberg disjoint from $\mathcal X$, the contrapositive of (iii)$\implies$(ii) in Theorem \ref{thm:Disjointness} implies that $\mathcal Y$ is not $\mathcal U$-mixing and so,  $\mathcal Y^\mathcal U$ is nontrivial.
\end{proof}
As we discuss in Section \ref{sec:SpectralMixing}, one can pick $\mathcal U\subseteq G^*$ so that the notion of $\mathcal U$-mixing coincides with any of the classical notions of weak mixing, mild mixing, or strong mixing. 
In particular, by choosing $\mathcal U$ appropriately, one can employ Theorem \ref{thm:Disjointness} to obtain the corresponding extensions of items (ii), (iii) in Theorem \ref{thm:ClassicalDichotomies} and Theorem \ref{thm:ParreauFactorForZ} to arbitrary countably infinite abelian groups.

  We remark that Theorem \ref{thm:Disjointness} is an immediate consequence of various results which we obtain throughout this text. 
  That (ii)$\implies$(iii) holds can be deduced from  Remark \ref{rem:UMixingIFFNoUParreauFactor}.
  That (iii)$\implies$(ii) is a consequence of Theorem \ref{thm:ParreauFactorsAreMultipliers}.
  The equivalence of (i) and (iii) follows from Corollary \ref{cor:HilbertianUmixingChar}.
  Moreover,  we provide two different approaches to the proof of  Theorem \ref{thm:ParreauFactorsAreMultipliers}. 
  The first of these approaches is based on Parreau's original proof of  Theorem \ref{thm:ParreauFactorForZ} and it relies on the concept of relatively independent joinings.
  Our second approach still employs some of the main ideas presented in \cite{ParreauFactor} but avoids the use of relatively independent joinings by employing a variant of the van der Corput trick introduced in \cite{StronglyMixingPET}.
  This second approach may provide a useful approach to extend Theorem \ref{thm:ParreauFactorsAreMultipliers} to the context of noncommutative groups.

The structure of the paper is as follows.
In Section \ref{sec:Preliminaries} we collect relevant facts about measure-preserving systems, joinings and disjointness, spectrum and spectral disjointness, the Image-Kernel decomposition of a normal operator on a Hilbert space, the algebra of the Stone-\v{C}ech compactification, and ultrafilter limits of operators.
In Section \ref{sec:SpectralMixing} we define $\mathcal{U}$-mixing for a family of ultrafilters $\mathcal{U}$, discuss how this recovers the classical notions of weak, mild, and strong mixing, and classify which systems are spectrally disjoint from all $\mathcal{U}$-mixing systems.
In Section \ref{sec:ParreauFactors} we define $\mathcal U$-generated systems and $\mathcal U$-generated factors, prove some of their basic properties, and provide many examples.
In Section \ref{sec:PartialRigidity} we study partially rigid systems and show that they are always a finite extension of their $\mathcal U$-generated factor.
In Section \ref{sec:MultipliersMainResults} we study the multiplier property and prove Corollaries \ref{cor:SingularSpectrumMultipliers} and \ref{cor:MainMultiplierResult}, which are two of the main results of this article.
In Section \ref{sec:vdC} we give an alternative proof of Theorem \ref{thm:ParreauFactorsAreMultipliers} using the van der Corput difference trick.
In Section \ref{sec:Questions} we state some open questions and conjectures.\\

\subsection{Previous work}

After an initial version of the present paper was shared with various experts in the field, it was brought to our attention that Parreau's work \cite{ParreauFactor} and many of its generalizations and consequences were announced and promoted in various talks in the mid 2000s (see \cite{Lemanczyk_Parreau}, for example). 
%The previous works building upon the work of Parreau include results due to   Lema\'nczyk \cite{Lemanczyk_IET,Lemanczyk_Parreau} and Tehungue \cite{Tehungue2024Parreau}. 
Many of the results proved in this paper were already known at that time and include results originally proved by Lema\'nczyk \cite{Lemanczyk_IET,Lemanczyk_Parreau} and Tehungue \cite{Tehungue2024Parreau}. 
For instance, they showed that the construction of Parreau yields systems in $\mathscr{M}(\mathcal{S}^\perp)$, and Theorem \ref{SpectralDisjointnessImpliesMultiplier} appears in \cite{Lemanczyk_IET,Lemanczyk_Parreau}.

Lema\'nczyk defines the \textbf{Parreau factor} of a system $\mathcal{X}$ to be the largest factor of $\mathcal{X}$ contained in $\mathscr{M}(\mathcal{S}^\perp)$, and in \cite{Lemanczyk_IET} shows that an ergodic IET is its own Parreau factor, hence a member of $\mathscr{M}(\mathcal{S}^\perp)$. 
The proof given there easily generalizes to any ergodic partially rigid system (cf. \cref{cor:MainMultiplierResult} part (iii)).
Furthermore, in \cite{Lemanczyk_IET,Lemanczyk_Parreau} it is shown that Parreau's construction can be generalized to recover the duality between identity and ergodic systems, rotations on compact abelian groups and weakly mixing systems, as well as rigid and mildly mixing systems.
Furthermore, it yields a class of systems that are related to singular spectrum systems and are members of $\mathscr{M}(\mathcal{A}^\perp)$, where $\mathcal{A}$ denotes the class of systems with absolutely continuous spectrum (see also \cite[Page 17]{Lemanczyk2009Spectral}).
The Parreau factor always contains the $\mathcal U$-generated factor of a given system (in view of \cref{cor:MainMultiplierResult}), but in some cases it is strictly larger: indeed, any nilsystem is its own Parreau factor, but its $\mathcal U$-generated factor is the Kronecker factor (cf. \cref{ex:Nilsystems}).

While many of the results in this paper were known before, we hope to bring renewed attention to this interesting topic and provide an accessible and detailed presentation.\\

\textbf{Acknowledgments:} %Firstly, we would like to thank Fran\c{c}ois Parreau for writing \cite{ParreauFactor} when asked for a proof of \cite[Theorem 11]{Lemanczyk2009Spectral}, which is the basis for the current paper.
We would like to thank Terry Adams, Tanja Eisner, Thierry de la Rue, Valery Ryzhikov, and Joanna Ku\l aga-Przymus for helpful discussions about mixing.
We would like to thank Jon Chaika for helpful conversations about interval exchange transformations and Mariusz Lema\'nczyk for bringing to our attention and making publicly available the notes \cite{Lemanczyk_IET,Lemanczyk_Parreau}.
\\

\textbf{AI usage:}
AI tools were used to review an earlier version of the paper for correctness. They found inaccuracies and gaps in a few proofs and helped fix them, namely \cref{lem:DisjointnessAndErgodicDecomposition}, \cref{lem:ProductIsExtreme} and \cref{thm:PartiallyRigidMixingDisjointness}. 
%We would also like to thank ChatGPT 5.6 pro for helping find and correct some mistakes in a previous edition of this work.
\\

\textbf{Funding acknowledgments:} JM and RZ are supported by EPSRC through Joel Moreira's Frontier Research Guarantee grant, ref. EP/Y014030/1. Additionally, SF was supported by this same grant during his visit to the University of Warwick during July of 2025, when a significant portion of this research was carried out.

%\textbf{AI Statement:} We would like to thank various AI models for their contributions to this paper.
%Firstly, Claude was used during the Perspectives in Ergodic Theory Conference at IMPAN during June of 2025 to obtain an English translation of \cite{ParreauFactor}. 
%We would like to thank ChatGPT for help finding many of the references that we have cited.
%We would also like to thank ChatGPT and DeepSeek for many helpful conversations about the mathematical content of this paper that led to improvements of some of the results.
%We acknowledge that AI frequently hallucinates and makes mistakes, so we have checked all of the given information for correctness before using it, and we take full responsibility for the contents of this paper.
%%%%%%%%%%%%%%%%%%%%%%%%%%%%%%%%%%%%%%%%%%%%%%%%%%%%%%%%%%%%%%%%%%%%%%%%%%%%%%%%%%%%%%%%%%%%%%%%%%%%%%%%%%%%%%
\section{Preliminaries}\label{sec:Preliminaries}
\subsection{Measure-Preserving Systems and the Ergodic Decomposition}\label{ssec:MPS&ErgodicDecomposition}
Throughout this article, we fix a countably infinite abelian group $G$.
A \textbf{measure-preserving $G$-system (m.p.s.)} is a tuple $(X,\mathscr{B},\mu,(T^g)_{g \in G})$ in which $(X,\mathscr{B},\mu)$ is a standard probability space and $T$ is a measure-preserving action of $G$.
We will always use $\mathcal{X}, \mathcal{Y}$, and $\mathcal{Z}$ to denote the $G$-systems $(X,\mathscr{B},\mu,(T^g)_{g \in G}), (Y,\mathscr{A},\nu,(S^g)_{g \in G}),$ and $(Z,\mathscr{C},\rho,(R^g)_{g \in G})$ respectively. 
%The $G$-system $\mathcal{X}$ is \textbf{ergodic} if the only $A \in \mathscr{B}$ for which $\mu(A\triangle T^gA) = 0$ for all $g \in G$ satisfy $\mu(A) \in \{0,1\}$.
%The $G$-system $\mathcal{X}$ is \textbf{weakly mixing} if the product system $\mathcal{X}\times\mathcal{X}$ is ergodic. 
%We will see more of the ergodic hierarchy of mixing properties in Section \ref{ssec:MixingAndQuasiInvariance}.

The $G$-system $\mathcal{Y}$ is a \textbf{factor} of $\mathcal{X}$ if there exists a measurable map $\pi:X\rightarrow Y$ satisfying $\pi_*\mu = \nu$ and $\pi\circ T^g = S^g\circ\pi$ for all $g \in G$.
In this case, $\pi^{-1}(\mathscr{A})=\{\pi^{-1}A:A\in\mathscr{A}\}$ is a $T$-invariant sub-$\sigma$-algebra of $\mathscr{B}$.
In fact, every $T$-invariant sub-$\sigma$-algebra of $\mathscr{B}$ corresponds to a factor of $\mathcal{X}$ in this manner, so we also refer to $T$-invariant sub-$\sigma$-algebras of $\mathscr{B}$ as factors. 
If $\mathcal{X}'$ is a factor of $\mathcal{X}$ corresponding to the sub-$\sigma$-algebra $\mathscr{B}'$, then we will write $L^2(X,\mu)$ for $L^2(X,\mathscr{B},\mu)$, and we will write $L^2(X,\mathscr{B}',\mu)$ when we want to discuss $L^2$ of the factor $\mathcal{X}'$.

Given a $G$-system $\XX$, the \emph{invariant factor} is the $\sigma$-algebra consisting of those sets $A\in \mathscr{B}$ satisfying $T^{-g}A=A$ for every $g\in G$.
When every set in the invariant factor has measure either $0$ or $1$, we say that the system is \emph{ergodic}.

If $\mathcal{Y}$ is a factor of $\mathcal{X}$, then $\mathcal{X}$ is an \textbf{extension} of $\mathcal{Y}$. 
It is a classical theorem of Rohlin (see \cite[Theorem 3.18]{ErgodicTheoryViaJoinings}) that if $\mathcal{X}$ is ergodic and an extension of $\mathcal{Y}$, then $\mathcal{X}$ is measurably isomorphic to the skew product system $(Y\times Z,\mathscr{A}\otimes\mathscr{C},\nu\otimes\rho,((S\rtimes \phi)^g)_{g \in G})$, where $(S\rtimes \phi)^g(y,z) = (S^gy,\phi(g,y)z)$, where $\phi:G\times Y\rightarrow\text{Aut}(Z,\rho)$ satisfies the cocycle equation $\phi(g_1+g_2,y) = \phi(g_1,S^{g_2}y)\phi(g_2,y)$.
If $Z = K/H$ with $K$ a second countable compact group, $H \le K$ a closed subgroup, $\rho$ the pushforward of the normalized Haar measure of $K$, and $\phi$ takes values in the set of left translations on $K/H$ by elements of $K$, then $\mathcal{X}$ is a \textbf{homogeneous skew-product extension} of $\mathcal{Y}$. 
A homogeneous skew-product extension in which $H = \{e_K\}$ is called a \textbf{compact group extension}.
If $\mathcal{X}$ is an extension of $\mathcal{Y}$, then $f \in L^2(X,\mu)$ is a \textbf{$\mathcal{Y}$-eigenfunction} if $f$ is contained in some $T$-invariant finite rank $L^\infty(Y,\nu)$-module $M$.
The system $\mathcal{X}$ is an \textbf{isometric extension} of $\mathcal{Y}$ if the closed linear hull of the $\mathcal{Y}$-eigenfunctions is $L^2(X,\mu)$.
If $\mathcal{X}$ is ergodic, then it is an isometric extension\footnote{Isometric extensions are equivalent to the compact extensions of Furstenberg \cite[Theorem 6.13]{FBook}.
In some sources, $\mathcal{X}$ is said to have relatively discrete spectrum over $\mathcal{Y}$ if $\mathcal{X}$ is an isometric extension of $\mathcal{Y}$.} of $\mathcal{Y}$ if and only if it is a homogeneous skew-product extension of $\mathcal{Y}$ (cf. \cite[Theorem 9.14]{ErgodicTheoryViaJoinings}).

If $\pi:\XX\to\YY$ is a factor map, the \emph{disintegration of $\mu$ with respect to $\pi$} is a measurable assignment $y\mapsto\mu_y$ to each $y\in Y$ of a probability measure $\mu_y$ on $(X,\mathscr{B})$ such that $\mu=\int_Y\mu_y\d\nu(y)$ and $\mu_y(\pi^{-1}(\{y\}))=1$ for $\nu$-a.e. $y\in Y$. 
We see that if $\pi:\mathcal{X}\rightarrow\mathcal{Y}$ is a factor map, then the map $\pi':X\rightarrow Y\times X$ given by $\pi'(x) = (\pi(x),x)$ gives a measurable isomorphism between $\mathcal{X}$ and the system $(Y\times X,\mathscr{A}\otimes\mathscr{B},\nu\rtimes\mu_y,((S\times T)^g)_{g \in G})$, where $(\mu_y)_{y \in Y}$ is the disintegration of $\mu$ with respect to $\pi$ and $(\nu\rtimes\mu_y)(A\times B) = \int_A\mu_y(B)d\nu(y)$.
If there is some $n \in \mathbb{N}$ for which $\mu_y$ is an atomic measure with at most $n$ atoms for a.e. $y$, then $\mathcal{X}$ is a \textbf{finite extension} of $\mathcal{Y}$.
We see that if $\mathcal{X}$ is a finite extension of $\mathcal{Y}$, then $L^2(X,\mu)$ is a finite rank $L^\infty(Y,\nu)$-module, so $\mathcal{X}$ is also an isometric extension of $\mathcal{Y}$.

If $\YY$ is the invariant factor of $\XX$, we say that the disintegration $(\mu_y)_{y\in Y}$ is the \emph{ergodic decomposition} of $\XX$.
In this case, $(X,\mathscr{B},\mu_y,T)$ is an ergodic measure-preserving system for $\nu$-almost every $y\in Y$.
We call the measures $\mu_y$ the \emph{ergodic components} of $\mu$, and we let $\mathcal{X}_y := (X,\mathscr{B},\mu_y,T)$.
When working with the ergodic decomposition, we will usually denote the invariant factor of the system at hand by $(\Omega,\mathscr{F},\mathbb{P},\text{Id})$, as we will mostly be concerned with the structure of the invariant factor as a probability space rather than as a measure-preserving system.

%%%%%%%%%%%%%%%%%%%%%%%%%%%%%%%%%%%%%%%%%%%%%%%%%%%%%%%%%%%%%%%%%%%%%%%%%%%%%%%%%%%%%%%%%%%%%%%%%%%%%%%%%
\subsection{Joinings and Disjointness}\label{ssec:JoiningsAndDisjointness}
A \textbf{joining} of the systems $\mathcal{X}$ and $\mathcal{Y}$ is a $T\times S$ invariant probability measure $\lambda$ on $X\times Y$ whose marginals onto $X$ and $Y$ are $\mu$ and $\nu$ respectively. 
By abuse of terminology, we also sometimes call the system $\mathcal{X}\vee\mathcal{Y} := (X\times Y,\mathscr{B}\otimes\mathscr{A},\lambda,(T^g\times S^g)_{g \in G})$ a joining.
Note that a joining of two systems is a common extension, with factor maps given by the coordinate projection maps $\pi_X:X\times Y\to X$ and $\pi_Y:X\times Y\to Y$. 
We denote by $J(\mathcal{X},\mathcal{Y}) \subseteq \mathcal{M}(X\times Y)$ the set of all joinings of $\mathcal{X}$ and $\mathcal{Y}$, and we let $J^e(\mathcal{X},\mathcal{Y})$ denote the set of extreme points of $J(\mathcal{X},\mathcal{Y})$.
If $\mathcal{X}$ and $\mathcal{Y}$ are both ergodic, then $J^e(\mathcal{X},\mathcal{Y})$ coincides with the set of ergodic joinings of $\mathcal{X}$ and $\mathcal{Y}$ (see \cite[Theorem 6.2]{ErgodicTheoryViaJoinings}).
The systems $\mathcal{X}$ and $\mathcal{Y}$ are \textbf{disjoint} if their only joining is the product measure, and we denote this by $\mathcal{X}\perp\mathcal{Y}$.
It is well known that if neither $\mathcal{X}$ nor $\mathcal{Y}$ are ergodic, then they are not disjoint,\footnote{\label{fn:NoErgodicityNoDisjointness} We see that if $A$ is $T$-invariant and $B$ is $S$-invariant, then for $\text{max}(0,\mu(A)+\nu(B)-1) \le t \le \text{min}(\mu(A),\nu(B))$ the measure $\lambda_t := t\mu_A\otimes\nu_B+(\mu(A)-t)\mu_A\otimes\nu_{B^c}+(\nu(B)-t)\mu_{A^c}\otimes\nu_B+(1-\mu(A)-\nu(B)+t)\mu_{A^c}\otimes\nu_{B^c}$ is a joining of $\mu$ and $\nu$.} so disjointness is only meaningfully discussed when at least one of the systems is ergodic.

Given (not necessarily distinct) systems $\mathcal{X}$ and $\mathcal{Y}$, a \textbf{Markov operator} is a bounded positive linear operator $M:L^2(X,\mu)\rightarrow L^2(Y,\nu)$ for which $M\mathbbm{1}_X = \mathbbm{1}_Y$ and $M^*\mathbbm{1}_Y = \mathbbm{1}_X$.
There is a one-to-one correspondence between joinings $\lambda \in J(\mathcal{X},\mathcal{Y})$ and Markov operators $M_\lambda$ that intertwine $T$ and $S$, i.e., Markov operators $M_\lambda$ for which $M_\lambda T^g = S^gM_\lambda$ for all $g \in G$.
If $M_\lambda$ is the Markov operator associated to a joining $\lambda \in J^e(\mathcal{X},\mathcal{Y})$, then it is \textbf{indecomposable}, and this happens if and only if $M_\lambda$ is an extreme point in the space of Markov operators that intertwine $T$ and $S$.

The following lemma is well known in the folklore, and used implicitly in \cite{ErgodicPerp,UniquelyErgodicPerp}. 
For completeness we present a proof.

\begin{lemma}\label{lem:DisjointnessAndErgodicDecomposition}
    Let $\mathcal{X}$ and $\mathcal{Y}$ be systems with $\mathcal{X}$ ergodic.
    Let $\ZZ$ denote the invariant factor of $\YY$ with factor map $\pi:Y\to Z$ and let $z\mapsto\nu_z$ be the ergodic decomposition of $\nu$.
    \begin{enumerate}[(i)]
        \item If $\lambda \in J(\mathcal{X},\mathcal{Y})$ and $z\mapsto \lambda_z$ is the disintegration of $\lambda$ with respect to $\pi\circ\pi_Y$, then $\lambda_z \in J(\mathcal{X},\mathcal{Y}_z)$ for a.e. $z$.

        \item If $\lambda \in J^e(\mathcal{X},\mathcal{Y})$, then $\lambda_z \in J^e(\mathcal{X},\mathcal{Y}_z)$ for a.e. $z$.

        \item The systems $\mathcal{X}$ and $\mathcal{Y}$ are disjoint if and only if almost every ergodic component of $\mathcal{Y}$ is disjoint from $\mathcal{X}$.
    \end{enumerate}
\end{lemma}

\begin{proof}
    We first prove (i).
    Since the projections $(\pi_X)_*\lambda_z$ are $T$-invariant measures on $X$ and the integral $\int_Z(\pi_X)_*\lambda_z\d\rho(z)=(\pi_X)_*\lambda=\mu$ is an ergodic measure, almost every $(\pi_X)_*\lambda_z$ must equal $\mu$.
    Since $(\pi_Y)_*\lambda = \nu$, we see that $z\mapsto(\pi_Y)_*\lambda_z$ gives a disintegration of $\nu$ with respect to $\pi$.
    Since disintegration with respect to a factor map is unique, we see that $(\pi_Y)_*\lambda_z = \nu_z$ for a.e. $z$, hence $\lambda_z$ is a joining of $\mu$ and $\nu_z$ for a.e. $z$.

    We prove (ii) by way of contradiction.
    Suppose that there exists $A \in \mathscr{C}$ with $\rho(A) > 0$ and for each $z \in A$ we have $\lambda_z \notin J^e(\mathcal{X},\mathcal{Y}_z)$.
    Consider $A_\lambda := \{(z,\eta^{(1)},\eta^{(2)}) \in Z\times \mathcal{M}(X\times Y)^2\ |\ z \in A, (\eta^{(1)},\eta^{(2)}) \in J(\mathcal{X},\mathcal{Y}_z)^2,\text{ and }\lambda_z = \frac{1}{2}\eta^{(1)}+\frac{1}{2}\eta^{(2)}\}$, and observe that $A_\lambda$ is measurable.
    For each $z \in A$, the set $(A_\lambda)_z := \{(\eta^{(1)},\eta^{(2)})\ |\ (z,\eta^{(1)},\eta^{(2)}) \in A_\lambda\}$ is closed in the weak$^*$ topology of $\mathcal{M}(X\times Y)^2$.
    It follows that $(A_\lambda)_z\setminus\{(\lambda_z,\lambda_z)\}$ is $\sigma$-compact, so \cite[Proposition 3.1]{BorelSelectionLemma} tells us that there exists a Borel measurable function $f:A\rightarrow \mathcal{M}(X\times Y)^2$ for which $f(z) \in (A_\lambda)_z\setminus\{(\lambda_z,\lambda_z)\}$ for all $z \in A$.
    Let us write $f(z) = (\lambda_z^{(1)},\lambda_z^{(2)})$.
    We see that $\lambda = \frac{1}{2}\lambda^{(1)}+\frac{1}{2}\lambda^{(2)}$, where $\lambda^{(i)} = \int_{A^c}\lambda_zd\rho(z)+\int_A\lambda_z^{(i)}d\rho(z)$, which contradicts the fact that $\lambda \in J^e(\mathcal{X},\mathcal{Y})$.

    For the first direction of (iii), let us assume that $\mathcal{X}$ is disjoint from almost every ergodic component of $\mathcal{Y}$.
    Using part (i), we see that for every $\lambda \in J(\mathcal{X},\mathcal{Y})$ and a.e. $z$ we have $\lambda_z=\mu\times\nu_z$, whence it follows that $\lambda=\mu\times\nu$.

    For the next direction, let us assume for the sake of contradiction that there exists $A \in \mathscr{C}$ with $\rho(A) > 0$ and for each $z \in A$ there is a nontrivial joining $\lambda_z$ of $\mu$ and $\nu_z$.
     Consider $A_\mu := \{(z,\lambda_z)\ |\ z \in A\text{ and }\lambda_z \in J(\mu,\nu_z)\setminus\{\mu\otimes\nu_z\}\}$, and observe that $A_\mu$ is measurable.
     For each $z \in A$, the set $J(\mathcal{X},\mathcal{Y}_z)$ of joinings between $\mathcal{X}$ and $\mathcal{Y}_z$ is a closed subset of $\mathcal{M}(X\times Y)$ with the weak$^*$ topology.
     It follows that $J(\mathcal{X},\mathcal{Y}_z)\setminus\{\mu\otimes\nu_z\}$ is $\sigma$-compact, so \cite[Proposition 3.1]{BorelSelectionLemma} tells us that there exists a Borel measurable function $f:A\rightarrow \mathcal{M}(X\times Y)$ for which $f(z) \in J(\mathcal{X},\mathcal{Y}_z)\setminus\{\mu\otimes\nu_z\}$ for all $z \in A$.
     The measure $\lambda := \int_{A^c} \mu\otimes\nu_z d\rho(z)+\int_Af(z)d\rho(z)$ is a nontrivial joining of $\mu$ and $\nu$, which contradicts the disjointness of $\mathcal{X}$ and $\mathcal{Y}$.
\end{proof}

%Furstenberg's Theorem about lifting disjointness via ergodic compact group extensions \cite{DisjointnessInErgodicTheory} was used to show that $\mathcal{D}\cap\mathcal{E} \subseteq \mathcal{W}^\perp$ in the case of $\mathbb{Z}$-systems.

In \cite{DisjointnessInErgodicTheory} Furstenberg showed that $\mathcal{D}\cap\mathcal{E} \subseteq \mathcal{W}^\perp$ in the case of $\mathbb{Z}$-systems by ``lifting disjointness'' through compact group extensions. 
We require the following closely related result.

\begin{theorem}[{cf. \cite[Theorem 10.19.2]{ErgodicTheoryViaJoinings}}]\label{thm:FurstenbergLiftingDisjointness}
    Let $\mathcal{X}$ and $\mathcal{Y}$ be disjoint ergodic systems.
    If $\mathcal{X}$ is weakly mixing and $\mathcal{Y}^c$ is an isometric extension of $\mathcal{Y}$, then $\mathcal{Y}^c\perp\mathcal{X}$.
\end{theorem}

\begin{lemma}\label{lem:FiniteExtensionsPreserveDisjointness}
    Let $\mathcal{X}$ be a weakly mixing system and let $\mathcal{Y}$ be a system that is disjoint from $\mathcal{X}$.
    If $\mathcal{Y}^c$ is an isometric extension of $\mathcal{Y}$, then $\mathcal{Y}^c\perp\mathcal{X}$.
\end{lemma}

\begin{proof}
    Let $\mathcal{Y}'$ be the intermediate factor between $\mathcal{Y}^c$ and $\mathcal{Y}$ obtained by adjoining the invariant $\sigma$-algebra of $\mathcal{Y}^c$ to the $\sigma$-algebra of $\mathcal{Y}$.
    Since every identity system is a multiplier for the class of ergodic systems, and $\mathcal{Y}'$ is a joining of $\mathcal{Y}$ with an identity system, we have $\mathcal{Y}'\perp\mathcal{X}$.
    Since $\mathcal{Y}^c$ and $\mathcal{Y}'$ have the same $\sigma$-algebra of invariant sets, $\mathcal{Y}^c$ is a relatively ergodic extension of $\mathcal{Y}'$ in addition to being an isometric extension of $\mathcal{Y}'$.
    We now invoke \cite[Proposition 5.7]{AustinNonErgodic} to see that a.e. ergodic component of $\mathcal{Y}^c$ is a homogeneous skew-product extension of an ergodic component of $\mathcal{Y}'$.
    Lemma \ref{lem:DisjointnessAndErgodicDecomposition}(iii) tells us that a.e. ergodic component of $\mathcal{Y}'$ is disjoint from $\mathcal{X}$, so Theorem \ref{thm:FurstenbergLiftingDisjointness} tells us that a.e. ergodic component of $\mathcal{Y}^c$ is disjoint from $\mathcal{X}$, hence another application of Lemma \ref{lem:DisjointnessAndErgodicDecomposition}(iii) yields the desired result.
\end{proof}
    \begin{corollary}\label{cor:ClosedUnderCompactExtensions}
        Let $\mathcal{C} \subseteq \mathcal{W}$ be a class of weakly mixing systems.
        \begin{enumerate}[(i)]
            \item The class $\mathcal{C}^\perp$ is closed under isometric extensions.

            \item The class $\mathscr{M}(\mathcal{C}^\perp)$ is closed under isometric extensions.

            \item We have $\mathcal{D} \subseteq \mathscr{M}(\mathcal{C}^\perp)$.
        \end{enumerate}
    \end{corollary}

    \begin{proof}
        Part (i) follows from Lemma \ref{lem:FiniteExtensionsPreserveDisjointness}.

        We now proceed to prove (ii).
        Let $\mathcal{X} \in \mathscr{M}(\mathcal{C}^\perp)$ and $\mathcal{Y} \in \mathcal{C}^\perp$ both be arbitrary, and let $\mathcal{X}^c$ be an isometric extension of $\mathcal{X}$.
        Let $\mathcal{X}^c\vee\mathcal{Y}$ be a joining of $\mathcal{X}^c$ and $\mathcal{Y}$, and observe that it is an isometric extension of some joining $\mathcal{X}\vee\mathcal{Y}$ of $\mathcal{X}$ and $\mathcal{Y}$.
        We have $\mathcal{X}\vee\mathcal{Y} \in \mathcal{C}^\perp$ by assumption, so part (i) tells us that $\mathcal{X}^c\vee\mathcal{Y} \in \mathcal{C}^\perp$.

        We recall that $\mathscr{M}(\mathcal{C}^\perp)$ is a characteristic class, so it is closed under inverse limits.
        Since any distal system is an inverse limit of a tower of isometric extensions of the 1-point system, part (iii) follows from part (ii).
    \end{proof}

    \begin{remark}\label{rem:DistalIsNotRigidOrParreau}
        Corollary \ref{cor:ClosedUnderCompactExtensions} is implicit in \cite{WPerpAndItsMultipliers} in the case of $\mathbb{Z}$-systems, and it implies that $\mathcal{D} \subseteq \mathscr{M}(\mathcal{M}^\perp)$ as well as $\mathcal{D} \subseteq \mathscr{M}(\mathcal{S}^\perp)$.
        The system $\mathcal{X} = (\mathbb{T}^2,\mathscr{L}^2,m^2,T)$, where $m$ denotes the Lebesgue measure, $\mathscr{L}:=\text{Borel}(\mathbb T)$, and $T(x,y) = (x,y+x)$, is a distal system, in fact, a $2$-step nilsystem, that is not rigid.
        We will also see later on that $\mathcal{X}$ is not a $\mathcal U$-generated factor, so $\mathscr{M}(\mathcal{M}^\perp) \not\subseteq \mathcal{R}$ and $\mathscr{M}(\mathcal{S}^\perp) \not\subseteq \mathcal{G}$, as claimed in the introduction.
        In fact, $\mathcal{X}$ has a Lebesgue component to its spectrum, so it also shows that $\mathscr{M}(\mathcal{A}^\perp)$ strictly contains the class of singular spectrum systems.
    \end{remark}

    If either of $\mathcal{X}$ and $\mathcal{Y}$ is not ergodic, then there is no ergodic joining of the two systems.
    Nonetheless, we have the following result.\\

    \begin{lemma}\label{lem:ProductIsExtreme}
        If $\mathcal{X}$ and $\mathcal{Y}$ are systems with $\mathcal{X}$ weakly mixing, then $\mu\otimes\nu \in J^e(\mathcal{X},\mathcal{Y})$.
    \end{lemma}

    \begin{proof}
        Let $\{\nu_\omega\}_{\omega \in \Omega}$ be the ergodic decomposition of $\mathcal{Y}$.
        Let $\lambda \in J(\mathcal{X},\mathcal{Y})$ be arbitrary, and let $\lambda = \int_\Omega\lambda_\omega d\mathbb{P}(\omega)$ be the disintegration of $\lambda$ with respect to the invariant factor of $\mathcal{Y}$.
        Since $\lambda$ is $T\times S$ invariant, so is each $\lambda_\omega$, hence each $\mu_\omega := (\pi_{X})_*\lambda_\omega$ is $T$-invariant.
        Since $\mu$ is ergodic and $\mu = \int_\Omega\mu_\omega d\mathbb{P}(\omega)$, we see that $\mu = \mu_\omega$ for a.e. $\omega$.
        Now let $\lambda^{(1)},\lambda^{(2)} \in J(\mathcal{X},\mathcal{Y})$ and $t \in (0,1)$ be such that $\mu\otimes\nu = t\lambda^{(1)}+(1-t)\lambda^{(2)}$.
        For $i = 1,2$, consider the disintegrations $\lambda^{(i)} = \int_{\Omega}\lambda^{(i)}_\omega d\mathbb{P}(\omega)$ of $\lambda^{(i)}$ with respect to the invariant factor of $\mathcal{Y}$, and observe that $\mu\otimes\nu_\omega = t\lambda^{(1)}_\omega+(1-t)\lambda^{(2)}_\omega$ for a.e. $\omega$. 
        Since $\mu$ is weakly mixing and $\nu_\omega$ is ergodic, we see that $\mu\otimes\nu_\omega$ is ergodic, and hence an extreme point in the space of joinings of $\mu$ and $\nu_\omega$.
        It follows that $\lambda^{(1)}_\omega = \lambda^{(2)}_\omega = \mu\otimes\nu_\omega$ for a.e. $\omega$, hence $\lambda^{(1)} = \lambda^{(2)} = \mu\otimes\nu$.
    \end{proof}

    We mention that if neither of $\mathcal{X}$ and $\mathcal{Y}$ are ergodic, then $\mu\otimes\nu \notin J^e(\mathcal{X},\mathcal{Y})$.\footnotemark[\getrefnumber{fn:NoErgodicityNoDisjointness}]
    If $\mathcal{X}$ is a nontrivial Kronecker system, then $\mu\otimes\mu \notin J^e(\mathcal{X},\mathcal{X})$ regardless of whether or not $\mathcal{X}$ is ergodic.
    It is also worth noting that $J(\mathcal{X},\mathcal{Y})$ is a Choquet simplex if $\mathcal{X}$ and $\mathcal{Y}$ are both ergodic, but it might not be a Choquet simplex if $\mathcal{X}$ is nonergodic, regardless of the mixing properties of $\mathcal{Y}$.
    
    \begin{remark}\label{rem:StillHoldsForGeneralGroups}
        We mention that all of the results of this subsection are still true with the same proofs if the acting group $G$ is a locally compact second countable group. 
    \end{remark}
%%%%%%%%%%%%%%%%%%%%%%%%%%%%%%%%%%%%%%%%%%%%%%%%%%%%%%%%%%%%%%%%%%%%%%%%%%%%%%%%%%%%%%%%%%%%%%%%%%%%%%%%%
\subsection{Spectrum and Spectral Disjointness}\label{ssec:Spectrum}
Let $U$ and $V$ be unitary representations of a countable abelian group $G$ on separable Hilbert spaces $\mathcal{H}_1$ and $\mathcal{H}_2$ respectively.
An \textbf{intertwining operator} is a bounded linear operator $M:\mathcal{H}_1\rightarrow\mathcal{H}_2$ satisfying $MU^g = V^gM$ for all $g \in G$.
The representations $U$ and $V$ are \textbf{disjoint} if the only intertwining operator is the $0$ operator.
The systems $\mathcal{X}$ and $\mathcal{Y}$ are \textbf{spectrally disjoint} if the Koopman representations $T$ and $S$ of $G$ on $L_0^2(X,\mu)$ and $L_0^2(Y,\nu)$ respectively are disjoint as unitary representations, where $L_0^2(X,\mu) := \{f \in L^2(X,\mu)\ |\ \int_Xfd\mu = 0\}$.
It is a classical result (see \cite[Lemma 1]{ThouvenotSpectrallyDisjointImpliesDisjoint} or \cite[Theorem 6.28]{ErgodicTheoryViaJoinings}) that spectral disjointness implies Furstenberg disjointness.

Now let us fix a unitary representation $U$ of $G$ on $\mathcal{H}$.
The spectral theorem tells us that for each $\xi \in \mathcal{H}$, there exists a measure $\mu_\xi$ for which the restriction of $U$ to $\mathcal{H}_\xi := \overline{\text{Span}_\mathbb{C}\{U^g\xi\ |\ g \in G\}}$ is isomorphic to a multiplication operator on $L^2(\widehat{G},\mu_\xi)$, where $\widehat{G}$ is the Pontryagin dual of $G$.
More concretely, we let $V$ be the unitary representation of $G$ on $L^2(\widehat{G},\mu_\xi)$ given by $(V^gf)(\chi) = \chi(g)f(\chi)$.
There exists a unitary intertwining operator $M:\mathcal{H}_\xi\rightarrow L^2(\widehat{G},\mu_\xi)$ for which we also have $M(\xi) = 1$.
Since $\mathcal{H}$ is separable, there exists a countable set $\{\xi_i\}_{i \in I} \subseteq \mathcal{H}$ for which $\mathcal{H} = \oplus_{i \in I}\mathcal{H}_{\xi_i}$, hence $\mathcal{H}$ is isomorphic to $\oplus_{i \in I}L^2(\widehat{G},\mu_{\xi_i})$.
Since $I$ is countable, one can find a $\xi\in\mathcal H$ for which $\mu_{\xi_i} \ll \mu_\xi$ for all $i \in I$.
While the vector $\xi$ with this property is not unique, the measure class of $\mu_\xi$, denoted by $[\mu_\xi]$, is unique and is called the \textbf{maximal spectral type of $U$}.
It is well known that unitary representations $U$ and $V$ are disjoint if and only if their maximal spectral types are mutually singular measure classes.

Given a nontrivial system $\mathcal{X}$, the maximal spectral type of $\mathcal{X}$ is the maximal spectral type of the Koopman operator $T$ acting on $L^2_0(X,\mu)$, and we denote it by $[\mathcal{X}]$.
It is classical that $\mathcal{X}$ is an invariant system if and only if $[\mathcal{X}] = [\delta_{e_{\widehat{G}}}]$, $\mathcal{X}$ is ergodic if and only if $\nu(\{e_{\widehat{G}}\}) = 0$ for every $\nu \in [\mathcal{X}]$, $\mathcal{X}$ is a Kronecker system if and only if $[\mathcal{X}]$ consists of discrete measures, and $\mathcal{X}$ is weakly mixing if and only if $[\mathcal{X}]$ consists of continuous measures. 
The system $\mathcal{X}$ is defined to have \textbf{Lebesgue spectrum} if $[\mathcal{X}]$ consists of measures that are equivalent to the Haar measure $\lambda$ of $\widehat{G}$, $\mathcal{X}$ has \textbf{absolutely continuous spectrum} if $[\mathcal{X}]$ consists of measures that are absolutely continuous with respect to $\lambda$, and $\mathcal{X}$ has \textbf{singular spectrum} if $[\mathcal{X}]$ consists of measures that are mutually singular with $\lambda$.
%$\mathcal{X}$ is rigid if and only if $\mu$ is a rigid probability measure, meaning that there exists a sequence $(g_k)_{k = 1}^\infty \subseteq G$ for which $\lim_{k\rightarrow\infty}\hat{\mu}(g_k) = 1$, $\mathcal{X}$ is mildly mixing if and only if $\text{IP}^*-\lim_g\widehat{\mu}(g) = 0$, and $\mathcal{X}$ is strongly mixing if and only if $\lim_{g\rightarrow\infty}\widehat{\mu}(g) = 0$.
%%%%%%%%%%%%%%%%%%%%%%%%%%%%%%%%%%%%%%%%%%%%%%%%%%%%%%%%%%%%%%%%%%%%%%%%%%%%%%%%%%%%%%%%%%%%%%%%%%%%%%%%%%%%%%%%%%%%%%
%%%%%%%%%%%%%%%%%%%%%%%%%%%%%%%%%%%%%%%%%%%%%%%%%%%%%%%%%%%%%%%%%%%%%%%%%%%%%%%%%%%%%%%%%%%%%%%%%%%%%%%%%%%%%%%%%%%%%%%%%%%%%%%%%%%%%%%%%%%%%%%%%%
\subsection{Normal operators and the Image-Kernel decomposition}\label{ssec:ImageKernelDecomposition}
In this section we state and prove two classical results dealing with the decomposition of a separable Hilbert space $\mathcal H$. 
Denote the class of all bounded linear operators defined on $\mathcal H$ by $\mathcal B(\mathcal H)$. 
Recall that a map $V\in\mathcal B(\mathcal H)$ is called normal if $V^*V=VV^*$. (Here $V^*$ denotes the adjoint of $V$.) 
We remark that Lemma \ref{3.LemmaDecompositionOfSingleNormal} was used in the proof of  \cite[Theorem 3.1]{AlmostMixingOfAllOrders}.
\begin{lemma}\label{3.LemmaDecompositionOfSingleNormal}
    Let $\mathcal H$ be a separable Hilbert space and let $V\in \mathcal B(\mathcal H)$ be a normal operator. Then, 
    \begin{equation}\label{3.eq:OrthocomplementOfImage}
\text{\rm{Im}}(V)^\perp=\text{\rm{Ker}}(V)
    \end{equation}
and so, 
\begin{equation}\label{3.eq:DirectSumIm+Ker}
\mathcal H=\overline{\text{\rm{Im}}(V)}\oplus\text{\rm{Ker}}(V),
\end{equation}
where $\overline{\text{\rm{Im}}(V)}$ denotes the closure of $\text{\rm{Im}}(V)$ with respect to the norm-topology of $\mathcal H$. 
\end{lemma}
\begin{proof}
Note that because $V$ is normal, we have that for every $f\in\mathcal H$, 
$$
\|Vf\|^2=\langle Vf,Vf \rangle=\langle V^*Vf,f\rangle=\langle VV^*f,f\rangle=\langle V^*f,V^*f\rangle=\|V^*f\|^2  
$$
and so, 
$$
\text{\rm{Ker}}(V)=\text{\rm{Ker}}(V^*).
    $$
Noting that  \eqref{3.eq:DirectSumIm+Ker} is an immediate consequence of \eqref{3.eq:OrthocomplementOfImage}, it follows that all we need to show is that $\text{\rm{Im}}(V)^\perp=\text{\rm{Ker}}(V^*)$. To do this, consider 
     $\xi,\eta\in\mathcal H$ with $\eta\in \text{\rm{Ker}}(V^*)$. Noting that $\langle V\xi,\eta\rangle=\langle\xi,V^*\eta\rangle=0$ we see that $\eta\in \text{\rm{Im}}(V)^\perp$. On the other hand, if $\eta\in \text{\rm{Im}}(V)^\perp$, then for every $\xi\in\mathcal H$, $\langle \xi,V^*\eta\rangle=\langle V\xi,\eta\rangle=0$. So, $\eta\in \text{\rm{Ker}}(V^*)$.
\end{proof}
\begin{lemma}\label{3.thm:DecompositionOfHilbertSpace2}
    If $\mathcal H$ is a separable Hilbert space and $\mathcal V\subseteq\mathcal B(\mathcal H)$ is a non-empty family of bounded normal operators, then
    $$\left(\bigcap_{V\in\mathcal V}\text{{\rm Ker}}(V)\right)^\perp = \overline{\text{{\rm Span}}_\mathbb C\left(\bigcup_{V\in\mathcal V}\text{{\rm Im}}(V)\right)}$$
\end{lemma}
\begin{proof}
  By Equation \eqref{3.eq:OrthocomplementOfImage}, we have
  $$
\bigcap_{V\in\mathcal V} \text{{\rm Ker}}(V) = \bigcap_{V\in\mathcal V} \text{{\rm Im}}(V)^\perp = \left(\bigcup_{V\in\mathcal V}\text{{\rm Im}}(V)\right)^\perp.
$$
So, 
$$
\left(\bigcap_{V\in\mathcal V} \text{{\rm Ker}}(V)\right)^\perp=\left(\left(\bigcup_{V\in\mathcal V}\text{{\rm Im}}(V)\right)^\perp\right)^\perp=\overline{\text{{\rm Span}}_\mathbb C\left(\bigcup_{V\in\mathcal V}\text{{\rm Im}}(V)\right)},
$$
as claimed. 
\end{proof}
%%%%%%%%%%%%%%%%%%%%%%%%%%%%%%%%%%%%%%%%%%%%%%%%%%%%%%%%%%%%%%%%%%%%%%%%%%%%%%%
%%%%%%%%%%%%%%%%%%%%%%%%%%%%%%%%%%%%%%%%%%%%%%%%%%%%%%%%%%%%%%%%%%%%%%%%%%%%%%%%%%%%%%%%%%%%%%%%%%%%%%%%%%%%%%%%%%%%%%%%%%%%%%%%%%%%%%%%%%%
\subsection{Ultrafilters and the algebra of \texorpdfstring{$\beta G$}{beta G}}\label{ssec:Ultrafilters}
As we mentioned in the Introduction, the work of Parreau \cite{ParreauFactor} can be combined with some of the results in \cite{AlmostMixingOfAllOrders} and the language of limits along ultrafilters to obtain interesting results related not only with the notion of  strong mixing  but also with  a wide variety of modes of mixing (which include the classical notions of weak and mild mixing). Our goal in this subsection is to review the necessary background on $\beta G$, which we view as the set of all ultrafilters on the group $G$. 
For a complete treatment of this topic we refer the reader to \cite{HBook}.%We remark in passing that our usage of ultrafilters (or, more precisely, limits along ultrafilters) in this paper is not done  out of simple notational convenience but rather due to mathematical necessity (see  Remark\rz{} for further discussion). 
%%%%%%%%%%%%%%%%%%%%%%%%%%%%%%%%%%%%%%%%%%%%%%%%%%%%%%%%%%%%%%%%%%%%%%%%%%%%%%%%%%%%%%
\subsubsection{The Stone-{\v C}ech compactification of a set and limits along ultrafilters}\label{sssec:UltrafilterBasics}
Let $S$ be a non-empty, countable set.
An ultrafilter $p$ on $S$ is a non-empty family of subsets of $S$ with the following properties:
\begin{enumerate}[(i)]
    \item $\emptyset\not\in p$.
    \item If $A\in p$ and $A\subseteq B$, then $B\in p$.
    \item If $A,B\in p$, then $A\cap B\in p$.
    \item If $A,B\subseteq S$ satisfy $A\cup B\in p$, then at least one of $A,B$ belongs to $p$.
\end{enumerate}
Denote the Stone-{\v C}ech compactification of $S$ by $\beta S$.
Identifying every $s\in S$ with the  ultrafilter 
$$
p_s:=\{A\subseteq S\,|\,s\in A\},
$$
one can view $\beta S$ as the set of ultrafilters on $S$ endowed with $\mathcal T$,  the topology generated by the sets of the form $\overline A:=\{p\in\beta S\,|\,A\in p\}$, $A\subseteq S$.
Notice that $\{\overline A\,|\,A\subseteq S\}$ forms a basis of $\mathcal T$ and that for every $A\subseteq S$, one has that $\overline A$ is both closed and open (with respect to $\mathcal T$). 
The set $\{p_s\,|\,s\in S\}$ is called the set of principal ultrafilters on $S$. 
We denote the set of non-principal ultrafilters by $S^*$ (so, $S^*:=\beta S\setminus\{p_s\,|\,s\in S\}$).\\
Consider now a topological space $(X,\tau)$, an $x\in X$,  and a sequence $(x_s)_{s\in S}$ in $X$. For any $p\in\beta S$, we write 
$$
\plimgG{p}{s}{S}x_s=x
$$
if for every $U\in\tau$ with $x\in U$, $\{s\in S\,|\,x_s\in U\}\in p$.
We remark that when $(X, \tau)$ is a compact Hausdorff topological space, $\plimgG{p}{s}{S}x_s$ always exists and is unique.
Furthermore, the compactness of $\beta S$ implies that for any $x\in \overline{\{x_s\,|\,s\in S\}^{\tau}}$, there is a $p\in \beta S$ with $\plimgG{p}{s}{S}x_s=x$.
Indeed, 
$$
\bigcap_{U\in\tau\text{ with }x\in U} \{p\in\beta S\,|\,\{s \in S\,|\,x_s\in U\}\in p\}\neq \emptyset.
$$

\subsubsection{The algebra of \texorpdfstring{$\beta G$}{beta G}}\label{ssec:AlgebraOfBetaG}

Let $(G,+)$ be a countably infinite abelian group. 
The group operation $+$ on $G$ extends naturally to $\beta G$ by the rule 
\begin{equation}\label{2.SumFormula}
A\in p+q\iff \{g\in G\,|\,-g+A\in q\}\in p,
\end{equation}
where $-g+A=\{h\in G\,|g+h\in A\}$. We remark that $(\beta G,+)$ becomes a compact Hausdorff right topological semigroup (the last condition means that for any $q\in \beta G$, the map $\rho_q(p):=p+q$ is continuous with respect to $\mathcal T$).\\
The following result connects the arithmetic properties of $\beta G$ with the concept of $p$-limit. We omit the proof.
For any $p\in\beta G$, we denote the ultrafilter 
$$
\{A\subseteq G\,|\,-A\in p\}
$$
by $-p$ (here $-A=\{g\in G\,|\,-g\in A\}$).
\begin{proposition}
    Let $p,q,r\in\beta G$ be ultrafilters. The following statements hold:
    \begin{enumerate}
        \item (Cf. Theorem 4.5 in \cite{HBook} and Theorem 3.8 in \cite{ERTAnUpdate}) $r=p+q$ if and only if for every compact metric space $X$ and every sequence $(x_g)_{g\in G}$ in $X$,
        $$
\plimgG{r}{g}{G}x_g=\plimgG{p}{h}{G}\plimgG{q}{g}{G}x_{h+g}.
        $$
        \item $r=-p$ if and only if  for every compact metric space $X$ and every sequence $(x_g)_{g\in G}$ in $X$,
        $$
\plimgG{r}{g}{G}x_g=\plimgG{p}{g}{G}x_{-g}.
        $$
    \end{enumerate}
\end{proposition}

\begin{remark}
    Let $(g_k)_{k = 1}^\infty \subseteq G$ be an injective sequence, and let $\mathcal{U} \subseteq G^*$ be the collection of ultrafilters $p$ containing $\{g_k\}_{k = 1}^\infty$.
    If $X$ is a compact metric space and $(x_g)_{g \in G} \subseteq X$, then for $x \in X$ we have $\displaystyle\lim_{k\rightarrow\infty}x_{g_k} = x$ if and only if $\plimgG{p}{g}{G}x_g = x$ for every $p \in \mathcal{U}$.
\end{remark}
%%%%%%%%%%%%%%%%%%%%%%%%%%%%%%%%%%%%%%%%%%%%%%%%%%%%%%%%%%%%%%%%%%%%%%
%%%%%%%%%%%%%%%%%%%%%%%
\subsubsection{Unitary representations and \texorpdfstring{$p$}{p}-limits}\label{sssec:UltrafiltersAndUnitaryReps}
Let $(G,+)$ be a countably infinite abelian group, let $\mathcal H$ denote a separable Hilbert space, let $U=(U^g)_{g\in G}$ be a unitary action of $G$ on $\mathcal H$ by unitary maps, and let $\mathcal B(\mathcal H)$ be the set of bounded, linear operators defined on $\mathcal H$. Observe that when endowed with the weak operator topology, which we denote by $\sigma$, any bounded subset of  $\mathcal B(\mathcal H)$ is pre-compact. So, in particular, the closure of $\{U^g\,|\,g\in G\}$ with respect to the weak operator topology, $\overline{\{U^g\,|\,g\in G\}}^\sigma$, is compact. Furthermore, since for any $V,W\in \overline{\{U^g\,|\,g\in G\}}^\sigma$ one can find $p,q\in\beta G$ with $\plimgG{p}{g}{G}U^g=V$ and $\plimgG{q}{g}{G}U^g=W$ (where both limits are taken with respect to the weak operator topology), we see that 
\begin{equation}\label{eq:CommutingPlimits}
VW=\plimgG{p}{h}{G}U^hW=\plimgG{p}{h}{G}\plimgG{q}{g}{G}U^hU^g=\plimgG{p}{h}{G}\plimgG{q}{g}{G}U^gU^h=\plimgG{p}{h}{G}WU^h=WV
\end{equation}
and
\begin{equation}\label{2.eq:SumIsComposition}
    VW=\plimgG{p}{h}{G}\plimgG{q}{g}{G}U^{h+g}=\plimgG{(p+q)}{g}{G}U^g.
\end{equation}
So, $\overline{\{U^g\,|\,g\in G\}}^\sigma$ is a commutative subsemigroup of $\mathcal B(\mathcal H)$. For any $p\in\beta G$, we will set 
$$
U^p:=\plimgG{p}{g}{G}U^g,
$$
where, as before, the limit is taken with respect to the weak operator topology.\\
The next result (or, rather, one of its variants) was originally proved in the course of the proof of (iv)$\implies$(i) of Theorem 3.1 in \cite{AlmostMixingOfAllOrders}. We omit the proof.
\begin{proposition}\label{2.prop:NormalOperator}
    Let $(G,+)$ be a countably infinite abelian group, let $\mathcal H$ be a separable Hilbert space, let $(U^g)_{g\in G}$ be a unitary action of $G$ on $\mathcal H$, and let $V\in \mathcal B(\mathcal H)$ be a normal operator (i.e. $V^*V=VV^*$). The following statements hold:
    \begin{enumerate}[(i)]
        \item For any $p\in\beta G$, $U^{-p}=(U^p)^*$.
        \item For any $n\in\N$, 
        $$\text{{\rm Ker}}(V^*)=\text{{\rm Ker}}(V)=\text{{\rm Ker}}(V^n).$$
        So, in particular, for any non-negative integers $n,m$ with $n+m\geq 1$ and any $p\in\beta G$, 
        $$
\text{{\rm Ker}}((U^p)^n(U^{-p})^m)=\text{{\rm Ker}}(U^p).
        $$
    \end{enumerate}
\end{proposition}
\begin{remark}
    Note that Proposition \ref{2.prop:NormalOperator} item (i) (in combination with \eqref{eq:CommutingPlimits}) implies that for every $p\in\beta G$ the operator $U^p$ is normal. 
    Thus, by invoking Lemma \ref{3.thm:DecompositionOfHilbertSpace2} and the fact that $\{U^p\,|\,p\in\beta G\}$ is a commutative semigroup, we see that for any non-empty family of ultrafilters $\mathcal U\subseteq \beta G$, one has 
    $$f\in \overline{\text{Span}_{\mathbb C}\left(\bigcup_{p\in\mathcal U}\text{Im}(U^p)\right)}$$
    if and only if there are sequences $(\xi_k)_{k\in\N}$ in $\mathcal H$ and $(p_k)_{k\in\N}$ in $\mathcal U$ such that $f=\sum_{k=1}^\infty\xi_k$ where for  each $k\in\N$, 
    $$\xi_k\in \bigcap_{j=0}^{k-1}\text{Ker}(U^{p_j})\cap\overline{\text{Im}(U^{p_k}) },$$
     where we take $U^{p_0}$ to be the zero operator.
     We remark that the sequences $(\xi_k)_{k\in\N}$ and $(p_k)_{k\in\N}$ may not be unique. 
\end{remark}
%%%%%%%%%%%%%%%%%%%%%%%%%%%%%%%%%%%%%%%%%%%%%%%%%%%%%%%%%%%%%%%%%%%%%%%%%%%%%%%%%%%%%%%%%%%%%
%%%%%%%%%%%%%%%%%%%%%%%%%%%%%%%%%%%%%%%%%%%%%%%%%%%%%%%%%%%%%%%%%%%%%%%%%%%%%%%%%%%%%%%%%%%%%%%%%%%%%%%%%%%%%%%%
\section{Modes of mixing  associated with a family of ultrafilters}\label{sec:SpectralMixing}
Let $(G,+)$ be a countably infinite abelian group, let $\mathcal H$ denote a separable Hilbert space, let $U=(U^g)_{g\in G}$ be a unitary representation of $G$ on $\mathcal H$, and let $\mathcal U\subseteq G^*=\beta G\setminus G$ be non-empty. In this section we discuss an ergodic-theoretical phenomenon associated with the pair $(U,\mathcal U)$. Namely, we introduce and describe the basic properties of the notion of $\mathcal U$-mixing. As we mentioned in the introduction, this concept is behind many classical dichotomies such as compact-weak mixing and rigid-mild mixing. We remark that the concept of $\mathcal U$-mixing is closely related with many of the results presented in 
\cite{MixingViaFatIntersections}, \cite{KuangYeMixingViaFamilies}, \cite{AlmostMixingOfAllOrders}. It is also worth mentioning that the concept of $\mathcal U$-mixing was introduced for subsemigroups of $\Z^*$ in \cite{StronglyMixingPET}.
%%%%%%%%%%%%%%%%%%%%%%%%%%%%%%%%%%%%%%%%%%%%%%%%%%%%%%%%%%%%%%%%%
\subsection{The definition of \texorpdfstring{$\mathcal U$}{\mathcal U}-mixing and an elementary observation}
Let $\mathcal{U} \subseteq G^*$ be non-empty.
We say that the pair $(\mathcal H,(U^g)_{g\in G})$ is $\mathcal U$-mixing (or simply that one of the members of the pair is $\mathcal U$-mixing when the other member of the pair is clear from the context) if for every $f\in \mathcal H$ and every $p\in\mathcal U$, $U^pf=0$, meaning that for every $f'\in\mathcal H$, 
$$
\plimgG{p}{g}{G}\langle f',U^gf\rangle=0.
$$
The following proposition, which follows from item (ii) in Proposition \ref{2.prop:NormalOperator}, 
shows that for any given $\mathcal U\subseteq G^*$, one can always find various families of ultrafilters $\mathcal V\subseteq G^*$ for which the notions of $\mathcal U$-mixing and $\mathcal V$-mixing coincide (see Remark \ref{rem:NonUniqueFamilyForUmixing} for the explanation). 
\begin{proposition}\label{prop:DifferentFamiliesForSameUmixing}
    Let $(G,+)$ be a countably infinite abelian group and let $\mathcal U\subseteq G^*$ be non-empty.
    Consider the following families of ultrafilters associated with $\mathcal U$: 
    \begin{itemize}
         \item For each $\ell\in\N$, $\ell\mathcal U:=\{p_1+\cdots+p_\ell\,|\,p_1,...,p_\ell\in\mathcal U\}$.
         \item $\Delta(\mathcal U):=\{-p+p\,|\,p\in\mathcal U\}$.
        \item $\langle \mathcal U\rangle$ denotes the semigroup generated by $\mathcal U$. 
        In other words, 
        $\langle \mathcal{U}\rangle:=\bigcup_{\ell\in\N}\ell\mathcal U$.
        \item $\langle \mathcal U\rangle_\Delta:=\{\xi_1p_1+\cdots+\xi_\ell p_\ell\,|\,\ell\in\N,\,\xi_1,...,\xi_\ell\in\{-1,1\},\text{ and }p_1,...,p_\ell\in\mathcal U\}$, where we adopt the convention that for any $p\in\beta G$, $1p=p$ and $(-1)p=-p$.
    \end{itemize}
    The following are equivalent for a pair $(\mathcal H,U)$: 
    \begin{enumerate}[(i)]
        \item $(\mathcal H,U)$  is $\mathcal U$-mixing,

        \item $(\mathcal H,U)$ is $\overline{(\beta G+\langle\mathcal U\rangle_\Delta+\beta G)}$-mixing,

        \item $(\mathcal H,U)$ is $\langle \mathcal U\rangle_\Delta$-mixing, 
        
        \item $(\mathcal H,U)$ is $\langle \mathcal U\rangle$-mixing,

        \item For each $\ell>1$, $(\mathcal H,U)$ is $\ell\mathcal U$-mixing, 

        \item For some $\ell>1$, $(\mathcal H,U)$ is $\ell\mathcal U$-mixing, 
        
        \item $(\mathcal H,U)$ is $\Delta(\mathcal U)$-mixing. 
    \end{enumerate}
\end{proposition}
\begin{proof}
    Note that for every $\ell\in\N$, 
    $$\ell\mathcal U\subseteq \langle \mathcal U \rangle\subseteq \langle \mathcal U\rangle_\Delta\subseteq \overline{\beta G+\langle \mathcal{U}\rangle_\Delta+\beta G}$$
    and so, the implications (ii)$\implies$(iii)$\implies$(iv)$\implies$(v)$\implies$(vi) are all trivial. 
    It remains to show that (i)$\iff$(vi), (i)$\iff$(vii), and (i)$\implies$(ii).\\ 
    (i)$\iff$(vi): By formula \eqref{2.eq:SumIsComposition}, for any $\ell \in \mathbb{N}$ and any $p_1,...,p_\ell\in\mathcal U$, $U^{p_1+\cdots+p_\ell}=\prod_{j=1}^\ell U^{p_j}$ and so, 
    $\bigcap_{j=1}^\ell \text{Ker}(U^{p_j})\subseteq \text{Ker}(U^{p_1+\cdots+p_\ell})$. It follows that 
    $$\bigcap_{p\in\mathcal U}\text{Ker}(U^p)\subseteq \bigcap_{p\in\ell\mathcal U}\text{Ker}(U^p).$$
    On the other hand, let $q\in\mathcal U$ and set $r:=\underbrace{q+\cdots+q}_{\ell\text{ times}}$. By item (ii) in Proposition \ref{2.prop:NormalOperator}, 
    $$\text{Ker}(U^q)=\text{Ker}((U^q)^\ell)=\text{Ker}(U^r).$$ 
    So, 
    $$\bigcap_{p\in\ell\mathcal U}\text{Ker}(U^p)\subseteq \bigcap_{p\in\mathcal U}\text{Ker}(U^p).$$
    (i)$\iff$(vii): The equivalence of (i) and (vii) follows from item (ii) in Proposition \ref{2.prop:NormalOperator} which establishes that for every $p\in\mathcal U$,
    $$
\text{Ker}(U^p)=\text{Ker}((U^p)^*U^p)=\text{Ker}(U^{-p+p}).
    $$
    (i)$\implies$(ii): Suppose that $(\mathcal H,(U^g)_{g\in G})$ is  $\mathcal U$-mixing. To prove that $(\mathcal H,(U^g)_{g\in G})$ is $\overline{(\beta G+\langle\mathcal U\rangle_\Delta+\beta G)}$-mixing, it is enough to show that for any $q,r\in\beta G$, any $p\in\mathcal U$, any $\xi\in \{-1,1\}$, and any $f\in\mathcal H$, $U^qU^{\xi p}U^rf=0$. Invoking \eqref{eq:CommutingPlimits} and the fact that $\text{Ker}(U^p)=\text{Ker}(U^{-p})$, we see that 
    $$
U^q U^{\xi p}U^rf=U^q(U^{\xi  p}(U^rf))=U^q0=0,
    $$
 completing the proof.
\end{proof}
\begin{remark}\label{rem:NonUniqueFamilyForUmixing}
    Let $\widehat G$ denote the Pontryagin dual of $G$ and let $\chi\in\widehat{G}$ be a nontrivial character (i.e. there is $g\in G$ with $\chi(g)\neq 1$). Note that for any non-empty $\mathcal U\subseteq G^*$, one has that for every $p\in \Delta(\mathcal U)$, $\plimgG{p}{g}{G}\chi(g)=1$ and that for every $z\in \overline{\chi(G)}$, there is a $p\in \overline{\beta G+\langle \mathcal U\rangle_\Delta +\beta G}$ with $\plimgG{p}{g}{G}\chi(g)=z$. Thus, one always has $\Delta(\mathcal{U}) \subsetneq \overline{\beta G+\langle \mathcal U\rangle_\Delta +\beta G}$. 
\end{remark}
%%%%%%%%%%%%%%%%%%%%%%%%%%%%%%%%%%%%%%%%%%%%%%%%%%%%%%%%%%%%%%%%%%%%%%%%%%%%%%%%%%%%
\subsection{The notion dual to \texorpdfstring{$\mathcal U$}{\mathcal U}-mixing}
Note that the pair $(\mathcal H,(U^g)_{g\in G})$ is $\mathcal U$-mixing if and only if
\begin{equation}\label{HIsKernel}
\mathcal H=\bigcap_{p\in\mathcal U}\text{Ker}(U^p).
\end{equation}
Our goal in this subsection is to provide a ``spectral characterization'' of all those pairs $(\mathcal H,(U^g)_{g\in G})$ for which the situation complementing \eqref{HIsKernel} holds, that is those pairs  $(\mathcal H,(U^g)_{g\in G})$ for which
\begin{equation}\label{eq:HGeneratedByU}
\mathcal H=\overline{\text{Span}_\mathbb C\left( \bigcup_{p\in\mathcal U}\text{Im}(U^p)\right)}=\left(\bigcap_{p\in\mathcal U}\text{Ker}(U^p)\right)^\perp.
\end{equation}\\ 
We say that a pair $(\mathcal H,(U^g)_{g\in G})$ is \textbf{generated by  $\mathcal U$} if it satisfies \eqref{eq:HGeneratedByU}.

\begin{theorem}\label{0.thm:SpectralDisjointnessTheorem}
    Let $(G,+)$ be a countably infinite abelian group, let $\mathcal H$ be a separable Hilbert space, let $\mathcal U\subseteq G^*$ be non-empty, and let $U=(U^g)_{g\in G}$ be a unitary action of $G$ on $\mathcal H$.  
    The following statements are equivalent:
    \begin{enumerate}[(i)]
     \item $(\mathcal H,U)$ is spectrally disjoint from every $(\mathcal U,U)$-annihilated pair $(\mathcal H',(V^g)_{g\in G})$. (The pair $(\mathcal H',(V^g)_{g\in G})$ is \textbf{$(\mathcal U,U)$-annihilated} if 
       $$\mathcal H'= \overline{\text{{\rm Span}}_{\mathbb C}(\bigcup_{q\in\{r\in\beta G\,|\,U^r=U^p\}}\text{Ker}(V^q))}$$
       for every $p\in\mathcal U$.)
      \item $(\mathcal H,U)$ is spectrally disjoint from every $\mathcal U$-mixing pair $(\mathcal H',(V^g)_{g\in G})$. 
    \item $(\mathcal H, U)$ is generated by $\mathcal U$.
    \end{enumerate}
\end{theorem}

\begin{proof}[Proof of \cref{0.thm:SpectralDisjointnessTheorem}]
(i)$\implies$(ii): This implication is trivial.\\
(ii)$\implies$(iii): Suppose that (iii) does not hold. Then, by \cref{3.thm:DecompositionOfHilbertSpace2}, there exists a non-zero $\xi\in\mathcal H$ with $\xi\in \bigcap_{p\in\mathcal U}\text{{\rm Ker}}(U^p)$. Letting 
$$\mathcal H'=\overline{\text{\rm{Span}}_{\mathbb C}(\{U^g\xi\,|\,g\in G\})},$$
we see that $(\mathcal H',U)$ is a $\mathcal U$-mixing Hilbert space.
Since $(\mathcal H',U)$ and $(\mathcal H,U)$ are not spectrally disjoint, we see that (ii) does not hold.\\
(iii)$\implies$(i): Let $(\mathcal H',(V^g)_{g\in G})$ be a $(\mathcal U,U)$-annihilated pair and let $\mathcal J:\mathcal H'\rightarrow \mathcal H$ be an intertwining operator. 
Let $\xi\in\mathcal H$ and $\eta\in \mathcal H'$.
We will show that $\langle \xi,\mathcal J\eta\rangle=0$.\\
To do this, first note that because $\mathcal H$ is $\mathcal U$-generated, we can assume without loss of generality that there is a $p\in\mathcal U$ such that $\xi\in \text{\rm{Im}}(U^p)$. 
Thus, since $(\mathcal H',(V^g)_{g\in G})$ is $(\mathcal U,U)$-annihilated,  we can further assume without loss of generality that there is a $q\in \beta G$ such that $U^q=U^p$ and $\eta\in \text{{\rm Ker}}(V^q)$. 
Note that \cref{2.prop:NormalOperator} item (ii) implies that $\eta\in\text{{\rm Ker}}(V^{-q})$. Thus, noting that  there is a $\xi' \in \mathcal{H}$ with $U^p\xi'=\xi$, we obtain
$$
\langle \xi,\mathcal J\eta\rangle=\langle U^p\xi',\mathcal J \eta\rangle=\langle U^q\xi',\mathcal J\eta\rangle=\langle \xi',\mathcal J(V^{-q}\eta)\rangle=0,
$$
which proves the claim.
\end{proof}

The following is an immediate corollary to Lemma \ref{3.thm:DecompositionOfHilbertSpace2}, \eqref{HIsKernel}, and Theorem \ref{0.thm:SpectralDisjointnessTheorem}. It is needed for the sequel and the proof of Theorem \ref{thm:Disjointness} in the Introduction.
\begin{corollary}\label{cor:HilbertianUmixingChar}
    Let $(G,+)$ be a countably infinite abelian group, let $\mathcal H$ be a separable Hilbert space, let $\mathcal U\subseteq G^*$ be non-empty, and let $U=(U^g)_{g\in G}$ be a unitary action of $G$ on $\mathcal H$. The pair $(\mathcal H,(U^g)_{g\in G})$ is $\mathcal U$-mixing if and only if it is spectrally disjoint from every pair $(\mathcal H',(V^g)_{g\in G})$ generated by $\mathcal U$.
\end{corollary}
%%%%%%%%%%%%%%%%%%%%%%%%%%%%%%%%%%%%%%%%%%%%%%%%%%%%%%%%%%%%%%%%%%%
\subsection{Connection with the classical notions of mixing}
Our goal in this subsection is to  describe some of the families of ultrafilters which can be used to define (via Corollary \ref{cor:HilbertianUmixingChar})  the following classical notions of mixing: (a) weak mixing,  (b) mild mixing, and (c) strong mixing. 
In other words, we want to find those families of ultrafilters $\mathcal U\subseteq G^*$ with the property that a pair $(\mathcal H,(U^g)_{g\in G})$ possesses a given mode of mixing if and only if it is spectrally disjoint from every pair $(\mathcal H',(V^g)_{g\in G})$ generated by $\mathcal U
$.
%%%%%%%%%%%%%%%%%%%%%%%%%%%%%%%%%%%%%%%%%%%%%%%%%%%%%%%%%%%%%%%%%%%%%%%
\subsubsection{Weak mixing}\label{sssec:WeakMixing}
A unitary action $(U^g)_{g\in G}$ of a countably infinite abelian group $G$ on a Hilbert space $\mathcal H$ is called weakly mixing if for every $f\in \mathcal H$ and every F{\o}lner sequence $(F_N)_{N\in\N}$ in $G$, one has that 
\begin{equation}\label{eq:WeakMixingFormula}
\lim_{N\rightarrow\infty}\frac{1}{|F_N|}\sum_{g\in F_N}|\langle U^gf,f\rangle|=0.
\end{equation}
Let $\Phi(G)$ denote the set of all F{\o}lner sequences in $G$. We let $\Lambda(G)$ be the set of all ultrafilters $p\in\beta G$ with the property that for each $A\in p$,  there is a F{\o}lner sequence $(F_N)_{N\in\N}\in \Phi(G)$ such that 
    $$
    \overline d_{(F_N)}(A):=\limsup_{N\rightarrow\infty}\frac{|A\cap F_N|}{|F_N|} > 0.
    $$
The compact-weak mixing decomposition, a form of which appeared for the first time in  \cite{KoopmanVonNeumannContinuousSpectra} (see also \cite[Theorem 3.4, page 96]{Krengel1985ErgodicTheorems} or
\cite[Theorem 2.3]{ERTAnUpdate}), says that we have the decomposition $\mathcal{H} = \mathcal{H}_c\oplus\mathcal{H}_w$ where
\begin{equation*}
    \mathcal{H}_c = \{f \in \mathcal{H}\ |\ \{U^gf\}_{g \in G}\text{ is norm pre-compact}\}
\end{equation*}
and 
\begin{equation*}
    %\mathcal{H}_c = \{f \in \mathcal{H}\ |\ \{U^gf\}_{g \in G}\text{ is norm pre-compact}\}
    \mathcal{H}_w = \{f \in \mathcal{H}\ |\ \forall (F_N)_{N\in\N}\in\Phi(G)\forall f'\in \mathcal H,\,  \lim_{N\rightarrow\infty}\frac{1}{|F_N|}\sum_{g\in F_N}|\langle U^gf,f'\rangle| = 0\}.
\end{equation*}
%and 
%\begin{equation*}
%\mathcal{H}_w = \{f \in \mathcal{H}\ |\ U^pf = 0\text{ for some }p \in G^*\}.
%\end{equation*}
It turns out that $\mathcal{H}_w$ is the largest subspace of $\mathcal{H}$ on which $U$ is weakly mixing, and $\mathcal{H}_c$ is the closed linear hull of the eigenfunctions of $U$.
We mention that the compact-weak mixing decomposition is a special case of the Jacobs-de Leeuw-Glicksberg decomposition (see, for example, \cite[Chapter 16]{OTAoET} and \cite{JdLGDecompositionSurvey}) in which $\mathcal{H}_c$ is known as the reversible part, and $\mathcal{H}_w$ is known as the almost weakly stable part.
It is not hard to deduce from the compact-weak mixing decomposition that the pair $(\mathcal H, (U^g)_{g\in G})$ is weakly mixing if and only if it is spectrally disjoint from every pair $(\mathcal H',(V^g)_{g\in G})$ with the property that for every $f\in\mathcal H'$, the set $\{V^g f\,|\,g\in G\}$ is norm pre-compact in $\mathcal H'$ (\cite[Theorem 1.9]{BerRosWeakMixingActions1988}).
Let $P_c:\mathcal{H}\rightarrow\mathcal{H}_c$ denote the orthogonal projection that induces the compact-weak mixing decomposition.
Bergelson \cite[Corollary 4.6]{MinimalIdempotentsAndERT} showed that if $p \in G^*$ is a minimal idempotent, then $U^p = P_c$.
In  \cite[Theorem 2.25]{ProjectionLemma} it was shown that $U^q = P_c$ whenever  $q$ is an idempotent in $\Lambda(G)$. Moreover, 
\cite[Theorem 5.2]{JdLGDecompositionSurvey} establishes that for every $q\in\Delta(\Lambda(G))$, $U^q$ also equals $P_c$.\footnote{In fact, \cite[Corollary 4.6]{MinimalIdempotentsAndERT} holds for any countably infinite (not necessarily abelian) group, and \cite[Theorem 2.25]{ProjectionLemma} and \cite[Theorem 5.2]{JdLGDecompositionSurvey} are shown in  \cite{JdLGDecompositionSurvey} to hold for any countably infinite amenable group.}
In particular, we see that the compact-weak mixing decomposition is a special case of the Image-Kernel decomposition discussed in Section \ref{ssec:ImageKernelDecomposition}.
The following proposition demonstrates that the notion of $\mathcal U$-mixing coincides with that of weak mixing for various families of ultrafilters. 
\begin{proposition}\label{3.Prop:WeakMixingCharacterizations}
    Let $\mathcal H$ be a Hilbert space and let $(U^g)_{g\in G}$ be a unitary representation of $G$ acting on $\mathcal H$. The pair $(\mathcal H,(U^g)_{g\in G})$ is weakly mixing if and only if it is $\mathcal U$-mixing for any of the following choices of $\mathcal U$:
    \begin{enumerate}
        \item $\mathcal U=\Lambda(G)$.
     \item $\mathcal U=\{p\}$ for any given $p\in\Lambda(G)$.
    \item $\mathcal U=\{p\}$, where $p$ belongs to the minimal two-sided ideal of $\beta G$.
    \end{enumerate}
\end{proposition}
\begin{proof}
First assume that $(U^g)_{g\in G}$ is weakly mixing and note that by \eqref{eq:WeakMixingFormula}, one has that for every $f\in\mathcal H$ and every $p\in \Lambda(G)$, $U^p f=0$ and so, $(U^g)_{g\in G}$ is $\Lambda(G)$-mixing (that $U^pf=0$ follows from the definition of $p$-limit and the fact that for any $A\in p$, $\overline d_{(F_N)}(A)>0$).\\
If $(U^g)_{g\in G}$ is $\Lambda(G)$-mixing, then it is $\{p\}$-mixing for any $p\in\Lambda(G)$. Since $\Lambda(G)$ is a closed two-sided ideal \cite[Theorem 2.8]{glasscock2025folner}, we see that every $p$ belonging to the minimal two-sided ideal of $\beta G$ also belongs to $\Lambda(G)$. \\
All that remains to be shown is  that if $(U^g)_{g\in G}$ is $\{p\}$-mixing for some $p\in \Lambda(G)$, then it is weakly mixing. To do this, we will show that a pair $(\mathcal H',(V^g)_{g\in G})$ is spectrally disjoint from every weakly mixing system if and only if it is generated by $\{p\}$.\\
If the pair $(\mathcal H',(V^g)_{g\in G})$ is generated by $\{p\}$, it is not hard to deduce from \eqref{eq:WeakMixingFormula} that  $(\mathcal H',(V^g)_{g\in G})$  is spectrally disjoint from every weakly mixing pair.
On the other hand, the compact-weak mixing decomposition tells us that $(\mathcal{H}',(V^g)_{g \in G})$ is spectrally disjoint from every weakly mixing pair if and only if every $f\in\mathcal H'$ is such that $\{V^gf\,|\,g\in G\}$ is pre-compact.
Now fix such an $f$ and note that for every $\epsilon>0$ there is a non-empty finite subset $F_\epsilon\subseteq G$ such that 
$$
\plimgG{p}{g}{G}\inf_{h\in F_\epsilon}\|V^gf-V^hf\|\leq\epsilon.
$$
By property (iv) in the definition of an ultrafilter, we have that there is an $h_\epsilon \in F_\epsilon$ and an $A_\epsilon\in p$ with the property that for every $g\in A_\epsilon$, $\|V^gf-V^{h_\epsilon}f\|\leq\epsilon$. It follows that  
\begin{multline*}
    \|V^pf-V^{h_\epsilon}f\|^2=\|f\|^2+\|V^pf\|^2-2\Re(\langle V^pf,V^{h_\epsilon}f \rangle )\\
    \leq 2(\|f\|^2-\Re(\langle V^pf,V^{h_\epsilon}f \rangle ))=\plimgG{p}{g}{G}\|V^gf-V^{h_\epsilon}f\|^2\leq \epsilon^2,
\end{multline*}
and so, $\|V^pf\|=\lim_{\epsilon\rightarrow 0^+}\|V^{h_\epsilon}f\|=\|f\|$. Thus, $\text{Ker}(V^p)=\{0\}$ and so, $\mathcal H'$ is generated by $\{p\}$.
\end{proof}
From Proposition \ref{3.Prop:WeakMixingCharacterizations}, we see that for any $p \in \Lambda(G)$, the Image-Kernel decomposition given by $U^p$ is the compact-weak mixing decomposition, even though we do not necessarily have that $U^p = P_c$. The following result  demonstrates that for any $\mathcal U\subseteq G^*$, the notion of $\mathcal U$-mixing implies that of weak mixing. 
\begin{corollary}\label{3.cor:UmixingIsWeakMixing}
For any non-empty $\mathcal U\subseteq G^*$ and any $\mathcal U$-mixing pair $(\mathcal H,(U^g)_{g\in G})$, one  has that $(\mathcal H,(U^g)_{g\in G})$ is weakly mixing.
\end{corollary}
\begin{proof}
Set $\mathcal V=\beta G+\langle \mathcal U\rangle_{\Delta}+\beta G$. 
The equivalence of (i) and (ii) in the statement of Proposition \ref{prop:DifferentFamiliesForSameUmixing} implies that $(\mathcal H,(U^g)_{g\in G})$ is $\mathcal V$-mixing. 
Since $\mathcal V$ is a two-sided ideal of $\beta G$, we have that $(\mathcal H,(U^g)_{g\in G})$ is $\{p\}$-mixing for every $p$ belonging to the minimal two-sided ideal of $\beta G$.
Thus, by item 3 in Proposition \ref{3.Prop:WeakMixingCharacterizations}, we obtain that $(\mathcal H,(U^g)_{g\in G})$ is weakly mixing. 
\end{proof}
%%%%%%%%%%%%%%%%%%%%%%%%%%%%%%%%%%%%%%%%%%%%%%%%%%%%%%%%%%%%%%%%%
\subsubsection{Mildly mixing and rigid actions}\label{sssec:MildMixingRigid}
The notions of  mild mixing and rigidity were first introduced for measure-preserving systems in \cite{Walters1972,TheBirthOfMildMixing}, and they naturally generalize to the setting of unitary representations on Hilbert spaces as follows.
Let $(U^g)_{g\in G}$ be a unitary representation of $G$ acting on a Hilbert space $\mathcal H$. 
For any given sequence $(g_k)_{k\in\N}$ in $G$, we write $\lim_{k\rightarrow\infty}g_k=\infty$ if for every finite $F\subseteq G$, the set $\{k:g_k\in F\}$ is finite.
A vector $f\in\mathcal H$ is called \textbf{rigid} if there exists a sequence $(g_k)_{k\in\N}$ in $G$ with $\lim_{k\rightarrow\infty}g_k=\infty$ and the property that $\lim_{k\rightarrow\infty}U^{g_k}f=f$ weakly. 
Whenever the equality  $\lim_{k\rightarrow\infty}U^{g_k}f=f$ holds, we will say that $f$ is rigid along $(g_k)_{k\in\N}$.
The pair $(\mathcal H, (U^g)_{g\in G})$ is called \textbf{rigid}, if there exists a sequence $(g_k)_{k\in\N}$ in $G$ with $\lim_{k\rightarrow\infty} g_k=\infty$ along which every $f\in\mathcal H$ is rigid.
The pair $(\mathcal H,(U^g)_{g\in G})$ is \textbf{mildly mixing} if the only rigid $f\in \mathcal H$ is $f=0$.
It is known that the pair $(\mathcal H,(U^g)_{g\in G})$ is mildly mixing if and only if it is spectrally disjoint from every rigid pair $(\mathcal H',(V^g)_{g\in G})$ \cite[Theorem 0.4]{PolynomialActionsOfUnitaryOperators}.
We will let $E(G^*):=\{p\in G^*\,|\,p+p=p\}$ denote the set of idempotent non-principal ultrafilters. 
\begin{proposition}\label{3.prop:CharMildMixing}
    Let $(G,+)$ be a countably infinite abelian group and let $(U^g)_{g\in G}$ be a unitary representation of $G$ acting on the Hilbert space $\mathcal H$. The pair $(\mathcal H,(U^g)_{g\in G})$ is mildly mixing if and only if it is $E(G^*)$-mixing.  
\end{proposition}
Our proof of Proposition \ref{3.prop:CharMildMixing} makes use of the following two lemmas.
\begin{lemma}[Cf. Proposition 1.7 in \cite{PolynomialActionsOfUnitaryOperators}]\label{3.lem:SortOfEllisLemma}
 Let $(G,+)$ be a countably infinite abelian group, let $(U^g)_{g\in G}$ be a unitary representation of $G$ acting on the Hilbert space $\mathcal H$, and let $(g_k)_{k\in\N}$ be a sequence in $G$ satisfying $\lim_{k\rightarrow\infty}g_k=\infty$. 
    Suppose that there is an $f\in\mathcal H$ which is rigid along $(g_k)_{k\in\N}$.
    Then, there is a $p\in E(G^*)$ such that $U^pf=f$.
\end{lemma}
\begin{proof}[Proof of Lemma \ref{3.lem:SortOfEllisLemma}]
Let $\mathcal V=\{q\in \beta G\,|\,U^qf=f\}$. Note that for any $q,r\in \mathcal V$, $q+r\in\mathcal V$, and so, $\mathcal V$ is a subsemigroup of $\beta G$.
Furthermore, note that $\mathcal V$ is closed and so, $\mathcal V\cap G^*$ is a closed subsemigroup of $\beta G$. By our assumption, $f$ is rigid along $(g_k)_{k\in\N}$ and so, $\mathcal V\cap G^*\neq \emptyset$. By the Ellis-Numakura Lemma \cite[Theorem 2.5]{HBook}, $E(G^*)\cap \mathcal V\cap G^*$ is non-empty. 
\end{proof}
\begin{lemma}[Cf. Theorem 3.12 in \cite{ERTAnUpdate}, Theorem 1.7 in \cite{FKIPSzemerediLong}]\label{3.lem:IdempotentLimit}
    Let $(G,+)$ be a countably infinite abelian group, let $(U^g)_{g\in G}$ be a unitary representation of $G$ acting on the Hilbert space $\mathcal H$, and let $p\in E(G^*)$. Then, $U^p$ is an orthogonal projection (i.e. $U^pU^p=U^p$ and $(U^p)^*=U^p$).
\end{lemma}
\begin{proof}[Proof of Lemma \ref{3.lem:IdempotentLimit}]
    That $U^p$ is an orthogonal projection is a consequence of the classical fact that a bounded operator $V\in \mathcal B(\mathcal H)$ is an orthogonal projection if $V^2=V$ and $\|V\|_{\mathcal B(\mathcal H)}\leq 1$. For completeness, we now give the usual proof of this classical fact: Note that $\|V\|_{\mathcal B(\mathcal H)}=\|V^*\|_{\mathcal B(\mathcal H)}$. Thus, for any $f\in\mathcal H$,
    \begin{multline*}
    \|V^*Vf-Vf\|^2=\|V^*Vf\|^2+\|Vf\|^2-2\Re(\langle V^*Vf,Vf\rangle)\\
     =\|V^*Vf\|^2+\|Vf\|^2-2\Re(\langle Vf,V^2f\rangle)\leq2\|Vf\|^2-2\Re(\langle Vf,Vf\rangle)=0. 
    \end{multline*}
    It follows that $V=V^*V$ and so, $V^*=(V^*V)^*=V^*V=V$.
\end{proof}

\begin{remark}\label{3.rem:prigidpmixingdecomposition}
    Lemma \ref{3.lem:IdempotentLimit} shows us that if $p \in E(G^*)$, then the Image-Kernel decomposition associated to $U^p$ is the $p$-rigid-$p$-mixing decomposition.
    More concretely, we have $\mathcal{H} = \mathcal{H}_{p,r}\oplus\mathcal{H}_{p,m}$ where $U^pf = f$ for all $f \in \mathcal{H}_{p,r}$, $U^pf' = 0$ for all $f' \in \mathcal{H}_{p,m}$, and $U^p:\mathcal{H}\rightarrow\mathcal{H}_{p,r}$ is the orthogonal projection.
    In the case of $G = \mathbb{Z}$, this was observed in \cite[Proposition 1.7]{PolynomialActionsOfUnitaryOperators}.
\end{remark}
\begin{proof}[Proof of Proposition \ref{3.prop:CharMildMixing}]
If the pair $(\mathcal H,(U^g)_{g\in G})$ is $E(G^*)$-mixing then Corollary \ref{cor:HilbertianUmixingChar} tells us that it is spectrally disjoint from every pair $(\mathcal H',(V^g)_{g\in G})$ which is $\{p\}$-generated for some $p\in E(G^*)$.
By Lemmas \ref{3.lem:SortOfEllisLemma} and \ref{3.lem:IdempotentLimit}, it follows that $(\mathcal H,(U^g)_{g\in G})$ is spectrally disjoint from every rigid pair $(\mathcal H',(V^g)_{g\in G})$ and so, it is mildly mixing.\\
If the pair $(\mathcal H,(U^g)_{g\in G})$ is mildly mixing, Lemmas \ref{3.lem:SortOfEllisLemma} and \ref{3.lem:IdempotentLimit} imply that it is spectrally disjoint from every pair $(\mathcal H',(V^g)_{g\in G})$ which is $\{p\}$-generated for some $p\in E(G^*)$. 
Thus, $(\mathcal H,(U^g)_{g\in G})$ is spectrally disjoint from every pair  $(\mathcal H',(V^g)_{g\in G})$ which is $E(G^*)$-generated. 
The result now follows by invoking Corollary \ref{cor:HilbertianUmixingChar}.
\end{proof}
\begin{remark}
    Let $(X,d)$ be a metric space, let $(x_g)_{g\in G}$ be a sequence in $X$, and let $x\in X$.
    Following \cite[Definition 9.2]{FBook}, we write
    \begin{equation}\label{eq:IP*Formula}
\mathop{\text{IP$^*$-lim}}_{g\in G}\,x_g=x
    \end{equation}
    if for every $\epsilon>0$, the set $\{g\in G\,|\,d(x_g,x)<\epsilon\}$ is IP$^*$, meaning that for any sequence $(g_k)_{k\in\N}$ with $\lim_{k\rightarrow\infty}g_k=\infty$, there is a $t\in\N$ and $k_1<\cdots<k_t$ such that $d(x_{\sum_{j=1}^t g_{k_j}},x)<\epsilon$.
    One has that \eqref{eq:IP*Formula} holds if and only if for every $p\in E(G^*)$, $\plimgG{p}{g}{G}x_g=x$.
\end{remark}
%%%%%%%%%%%%%%%%%%%%%%%%%%%%%%%%%%%%%%%%%%%%%%%%%%%%%
\subsubsection{Strong mixing}
The pair $(\mathcal H,(U^g)_{g\in G})$ is called strongly mixing if for any $(g_k)_{k\in\N}$ in $G$ with $\lim_{k\rightarrow\infty}g_k=\infty$ and any $f\in\mathcal H$, $\lim_{k\rightarrow\infty}U^{g_k}f=0$ weakly.
Equivalently,  $(\mathcal H,(U^g)_{g\in G})$ is strongly mixing if it is $G^*$-mixing.%for any $p\in G^*$, $U^p=0$ (i.e. for every $f\in\mathcal H$, $U^pf=0$). 
The following result is an immediate consequence of Proposition \ref{prop:DifferentFamiliesForSameUmixing} (we omit the proof). 
For any non-empty $\mathcal U\subseteq G^*$ and any $\ell\in\N\cup\{0\}$, we define the family of ultrafilters $\Delta^\ell(\mathcal U)$ inductively as follows: (1) $\Delta^{0}(\mathcal U)=\mathcal U$ and (2) for $\ell>0$, $\Delta^\ell(\mathcal U)=\Delta(\Delta^{\ell-1}(\mathcal U))$. 
\begin{proposition}\label{3.prop:StrongMixingChar}
    The pair $(\mathcal H,(U^g)_{g\in G})$ is strongly mixing if and only if it is $\mathcal U$-mixing for any of the following choices of $\mathcal U$:
    \begin{enumerate}
        \item $\mathcal U=\Delta^\ell(G^*)$ for some $\ell\in\N$.
        \item $\mathcal U=\ell G^*$ for some $\ell\in\N$.
    \end{enumerate}
\end{proposition}
\begin{remark}
    It is worth mentioning that the distinct characterizations of strong mixing provided by Proposition \ref{3.prop:StrongMixingChar} can all be traced back to previous work (although these equivalences were established in the context of measure-preserving systems and in a somewhat distinct but equivalent way). 
    That the notion of strong mixing coincides with: (1) $\Delta(G^*)$-mixing was first noticed in \cite[Theorem 4.4]{KuangYeMixingViaFamilies} in the special case that $G=\Z$, (2) $\Delta^\ell(G^*)$-mixing, $\ell\in\N$, was first noticed in \cite[Theorem 1.11]{IteratedDifferenceSets} in the context of $G=\Z$, and (3)  $\ell G^*$-mixing was first noticed in \cite[Theorem 1.6]{AlmostMixingOfAllOrders} for any choice of countable abelian group $G$. 
\end{remark}
\begin{remark}
    A consequence of item (a) of Remark 1.12 in \cite{AlmostMixingOfAllOrders} is that for every countably infinite abelian group $G$ and any $\ell\in\N$, one has
    $$
c\ell_{\beta G}((\ell+1) G^*)\subsetneq c\ell_{\beta G}(\ell G^*)
    $$
    and so, Proposition \ref{3.prop:StrongMixingChar} provides infinitely many different characterizations of strong mixing. (Here $c\ell_{\beta G}(\cdot)$ denotes closure with respect to $\mathcal T$, the topology on $\beta G$ introduced in Subsection  \ref{sssec:UltrafilterBasics}.) 
\end{remark}
\begin{remark}
    Let $(G,+)$ be a countably infinite abelian group, let $(X,d)$ be a metric space, and let $(x_g)_{g\in G}$ be a sequence in $X$. 
    For each $\ell\in\N$, we write
    $$
\mathop{\Sigma_\ell^*\text{-lim}}_{g\in G}\,x_g=x
    $$
    if for every $p\in \ell G^*$, one has $\plimgG{p}{g}{G}x_g=x$. 
    It follows that the pair $(\mathcal H,(U^g)_{g\in G})$ is strongly mixing if and only if there exists an $\ell\in\N$ with the property that for every $f\in\mathcal H$, $
\mathop{\Sigma_\ell^*\text{-lim}}_{g\in G}\,U^gf=0$.
\end{remark}
%%%%%%%%%%%%%%%%%%%%%%%%%%%%%%%%%%%%%%%%%%
The following result gives usable conditions to check for disjointness from strong mixing.
Part (ii) was previously observed by Ryzhikov (cf. \cite[Remark 1]{SpectrumOfAlphaRigidMaps}), and part (iii) by El Abdalaoui (cf. \cite[Lemma 3]{SpectrumOfAlphaRigidMaps}) when $G=\Z$ and $F$ is analytic.

\begin{corollary}\label{cor:SpectralDisjointnessFromSM}
    Let $(G,+)$ be a countably infinite abelian group, let $\mathcal H$ be a separable Hilbert space, and let $U = (U^g)_{g \in G}$ be a unitary action of $G$ on $\mathcal{H}$.
    \begin{enumerate}[(i)]
        \item $(\mathcal{H},U)$ is spectrally disjoint from every strongly mixing pair $(\mathcal{H}',V)$ if and only if $(\mathcal{H},U)$ is generated by $G^*$.
        \item If $U^p$ is injective for some $p \in G^*$, then $(\mathcal{H},U)$ is spectrally disjoint from every $p$-mixing pair $(\mathcal{H}',V)$.
        \item %If $U^p = F(U)$ for some $p \in G^*$ and some $F \in C(\widehat{G})$ having at most countably many zeroes,\footnote{If $G = \mathbb{Z}$, then this condition is automatically fulfilled by nonzero analytic functions $F$.} then $(\mathcal{H},U)$ is spectrally disjoint from every $p$-mixing pair $(\mathcal{H}',V)$.
        Suppose there exist a $p \in G^*$, a measurable function $F:\hat G\rightarrow \mathbb C$, and a sequence of trigonometric polynomials $(F_n)_{n\in\N}$ such that $F_n\to F$ pointwise, $F_n(U)\to U^p$ in the weak operator topology, and $F$ has at most countably many zeros. Then,  $(\mathcal{H},U)$ is spectrally disjoint from every $p$-mixing pair $(\mathcal{H}',V)$.
    \end{enumerate}
\end{corollary}

\begin{proof}
    Part (i) follows immediately from Theorem \ref{0.thm:SpectralDisjointnessTheorem}, and part (ii) follows from Lemma \ref{3.LemmaDecompositionOfSingleNormal}.
    To prove part (iii),  let $\mathcal{H} = \mathcal{H}_c\oplus\mathcal{H}_w$ be the compact-weak mixing decomposition of $\mathcal{H}$ and observe that $(\mathcal{H}_c,U)$ is spectrally disjoint from all weakly mixing pairs, and hence all $p$-mixing pairs by Corollary \ref{3.cor:UmixingIsWeakMixing}, so it only remains to show that the same holds for $(\mathcal{H}_w,U)$.
    Using part (ii), it suffices to show that $U^p$ is injective on $\mathcal{H}_w$. 
    We will do this by showing that if $U^pf=0$ for some $f\in \mathcal{H}_w$, then $f=0$.
    Thus, pick $f\in\mathcal{H}_w$ with $U^pf=0$  and let $\sigma_f$ be the spectral measure of $f$ with respect to $U$.
To prove that $f=0$, it suffices to show that  $\int_{\widehat G}|F(\chi)|\text{d}\sigma_f(\chi)=0$. Indeed, if  $\int_{\widehat G}|F(\chi)|\text{d}\sigma_f(\chi)=0$ one has that $F=0$ for $\sigma_f$-almost every $\chi\in\widehat G$.
So, since $f\in \mathcal{H}_w$, $\sigma_f(\{\chi\})=0$ for each $\chi\in\widehat G$ and since $F$ has at most countably many zeros, we must have that $\sigma_f$ is the zero measure, proving that $\|f\|^2=\sigma_f(\widehat G)=0$.\\ 
To complete the proof, suppose now for contradiction that $\int_{\widehat G}|F(\chi)|\text{d}\sigma_f(\chi)>0$.
This means that one of the four open half-planes $U\subseteq \C$ satisfies  $\sigma_f(F^{-1}U)>0$. 
Let $B$ denote the  intersection of $U$ with an open ball centered at zero with large enough radius to ensure that $\sigma_f(F^{-1}B)>0$.
Note that $\lim_{N\rightarrow\infty}\sigma_f(F^{-1} B\setminus\bigcap_{n\geq N}F_n^{-1}B)=0$.
Thus, for some $N\in\N$, the set $A:=F^{-1}B\cap\bigcap_{n\geq N}F_n^{-1}B$ satisfies $\sigma_f(A)>0$ and, so, because $F(A)\subseteq B$ is a bounded subset of a half-plane, $\int_{A}F(\chi)\text{d}\sigma_f(\chi)\neq 0$. 
Let $h_A \in \mathcal{H}_w$ be the vector corresponding to $\mathbbm{1}_A \in L^2(\widehat{G},\sigma_f)$.
Noting that since $B$ is bounded, the dominated convergence theorem, followed by the weak operator convergence $F_n(U)\to U^p$, yields
 $$
\int_{\widehat G}1_{A}(\chi)F(\chi)\text{d}\sigma_f(\chi)=\lim_{n\rightarrow\infty}\int_{\widehat G}\mathbbm 1_{A}(\chi)F_n(\chi)\text{d}\sigma_f(\chi) = \langle h_A,U^pf\rangle = 0, 
 $$
a contradiction.
\end{proof}

\begin{remark}
    A simple case in which one can apply Corollary \ref{cor:SpectralDisjointnessFromSM} (iii) is when $U^p = F(U)$ for some $F \in C(\widehat{G})$ with countably many zeros.
    We recall that for any countable abelian group $G$, trigonometric polynomials, i.e., functions of the form $P(\chi) = \sum_{g \in G}c_g\chi(g)$ with $(c_g)_{g \in G}$ finitely supported, are dense in $C(\widehat{G})$.
    Consequently, for any $F \in C(\widehat{G})$ and any unitary representation $U$ of $G$ on a Hilbert space $\mathcal{H}$, we may define $F(U) = \displaystyle\lim_{n\rightarrow\infty}P_n(U) = \lim_{n\rightarrow\infty}\sum_{g \in G}c_{n,g}U^g$, where $P_n(\chi) = \sum_{g \in G}c_{n,g}\chi(g)$ is any sequence of trigonometric polynomials converging to $F$, and convergence to $F(U)$ takes place in the operator norm.
\end{remark}

%%%%%%%%%%%%%%%%%%%%%%%%%%%%%%%%%%%%%%%%%%%%%%%%%%
\subsection{Notions of mixing for measure-preserving systems}
Given a system $\mathcal{X}$ and a non-empty family of ultrafilters $\mathcal U\subseteq G^*$, we say that $\mathcal{X}$ is \textbf{$\mathcal U$-mixing} if the pair $(L^2_0(X,\mu), (T^g)_{g\in G})$ is $\mathcal U$-mixing as a Hilbert space.
We denote the collection of $\mathcal{U}$-mixing systems by $\mathcal{S}_\mathcal{U}$.
It follows that the system $\mathcal{X}$ is $\mathcal U$-mixing if and only if one has that for every $A,B\in\mathscr{B}$ and any $p\in\mathcal U$,
$$
\plimgG{p}{g}{G}\mu(A\cap T^{-g}B)=\mu(A)\mu(B).
$$
As a consequence of our previous discussion, we recover the following characterizations:
\begin{enumerate}
    \item (Cf. \cite[Theorem 4.1]{BerRosWeakMixingActions1988}) $\mathcal{X}$ is weakly mixing if and only if for every F{\o}lner sequence $(F_N)_{N\in\N}$ in $G$ and any $A,B\in\mathscr{B}$,
    $$
\lim_{N\rightarrow\infty}\frac{1}{|F_N|}\sum_{g\in F_N}|\mu(A\cap T^{-g}B)-\mu(A)\mu(B)|=0. 
    $$
    \item (Cf. \cite[Proposition 9.22]{FBook}) $\mathcal{X}$ is mildly mixing  if and only if for any $A,B\in\mathscr{B}$,
    $$
    \mathop{\text{IP$^*$-lim}}_{g\in G}\,\mu(A\cap T^{-g}B)=\mu(A)\mu(B).
    $$
    \item (Cf. \cite[Theorem 1.6]{AlmostMixingOfAllOrders}) Let $\ell \in \mathbb{N}$ be arbitrary.
    $\mathcal{X}$ is strongly mixing  if and only if for any $A,B\in\mathscr{B}$,
    $$
\mathop{\Sigma_\ell^*\text{-lim}}_{g\in G}\,\mu(A\cap T^{-g}B)=\mu(A)\mu(B).
    $$
\end{enumerate}
\begin{remark}
Let $(X,d)$ be a compact metric space and let $(x_g)_{g\in G}$ be a sequence in $X$. Suppose that there exists an $\ell\in\N$ and an $x\in X$ for which 
$
\mathop{\Sigma_\ell^*\text{-lim}}_{g\in G}\,x_g=x$.
The results in \cite[Section 5]{AlmostMixingOfAllOrders} imply that for such an $(x_g)_{g\in G}$ and $x\in X$, one also has
$$
\mathop{\text{IP$^*$-lim}}_{g\in G}\,x_g=x\text{ and }\lim_{N\rightarrow\infty}\frac{1}{|F_N|}\sum_{g\in F_N}d(x_g,x)=0,
$$
for every F{\o}lner sequence $(F_N)_{N\in\N}$. 
In other words, $\Sigma_\ell^*$-convergence implies IP$^*$-convergence and  uniform strong Ces{\`a}ro convergence. 
\end{remark}
\section{$\mathcal U$-generated factors}\label{sec:ParreauFactors}
\begin{definition}\label{def:ParreauFactor}
    Let $G$ be a discrete abelian group, let $\mathcal{X} = (X,\mathscr{B},\mu,(T^g)_{g \in G})$ be a measure-preserving $G$-system, and let $\mathcal{U} \subseteq G^*$ be non-empty. 
    The \textbf{$\mathcal{U}$-generated factor} of $\mathcal{X}$, denoted by $\mathcal{X}^{\mathcal{U}}$, is the smallest $G$-invariant factor containing all the functions $\{T^pf\ |\ f \in L^\infty(X,\mathscr{B},\mu),\ p \in \mathcal{U}\}$.
    We denote by $\mathscr{B}_{\mathcal{U}}$  the $\sigma$-subalgebra of $\mathscr{B}$ corresponding to $\mathcal{X}^{\mathcal{U}}$.
    If $\mathcal{X}^{\mathcal{U}} = \mathcal{X}$, then $\mathcal{X}$ is a \textbf{$\mathcal{U}$-generated system}.
    We let $\mathcal{G}_\mathcal{U}$ denote the class of $\mathcal{U}$-generated systems, and we use $\mathcal{G}$ to denote $\mathcal{G}_{G^*}$.
\end{definition}

\begin{remark}\label{rem:PFactorBasics}
     Given $p \in G^*$, and an eigenfunction $f \in L^\infty(X,\mu)$, we have $T^pf = c_pf$ for some $c_p \in \mathbb{S}^1$, so $f \in \text{Im}(T^p)$.
     It follows that for any $\emptyset \neq \mathcal{U} \subseteq G^*$, $\mathcal{X}^\mathcal{U}$ contains the Kronecker factor $\mathcal{X}_\mathcal{K}$, which is the factor of $\mathcal{X}$ generated by the eigenfunctions of $T$.
     As discussed in Section \ref{sssec:WeakMixing}, if $\mathcal{U} \subseteq \Lambda(G)$, then $\mathcal{X}^\mathcal{U} = \mathcal{X}_\mathcal{K}$.
     As discussed in Section \ref{sssec:MildMixingRigid}, if $p \in E(G^*)$, then $\mathcal{X}^p := \mathcal{X}^{\{p\}}$ is a $p$-rigid factor.
     We note that the collection of bounded eigenfunctions of $T$ is an algebra, and that the collection of bounded $p$-rigid functions of $T$ is also an algebra, which is why in the two previously mentioned cases we have that $T^p$ is a surjection from $L^2(X,\mu)$ to $L^2(X,\mathscr{B}_p,\mu)$.
     However, we stress that the collection $\{T^pf\ |\ f \in L^\infty(X,\mu),\ p \in \mathcal{U}\}$ is in general not closed under products, in which case $L^\infty(X,\mathscr{B}_{\mathcal{U}},\mu)$ is a strictly larger collection.
     This phenomenon is demonstrated explicitly by Examples \ref{ex:FirstExample}-\ref{ex:FourthExample}.
\end{remark}

\begin{remark}\label{rem:UMixingIFFNoUParreauFactor}
    If $\mathcal{X}$ is not a $\mathcal{U}$-mixing system, then there exists $p \in \mathcal{U}$ for which $T^p$ is not the projection onto the constants, and hence $\mathcal{X}^p$ is a nontrivial $\mathcal{U}$-generated factor of $\mathcal{X}$ that is contained in $\mathcal{X}^\mathcal{U}$. 
    It follows that $\mathcal{X}$ is $\mathcal{U}$-mixing if and only if its $\mathcal{U}$-generated factor $\mathcal{X}^\mathcal{U}$ is the trivial system.
    In particular, $\mathcal{X}$ is strongly mixing if and only if all of its $\mathcal{U}$-generated factors are trivial, or equivalently, if $\mathcal{X}^{G^*}$ is the trivial system. With the help of the language of limits along ultrafilters, the results of 
    Parreau \cite{ParreauFactor} (see also \cite[Theorem 11]{Lemanczyk2009Spectral}) can easily be modified to show that for any non-empty $\mathcal U\subseteq G^*$, the factor $\mathcal{X}^\mathcal U$ is disjoint from every $\mathcal U$-mixing system $\mathcal{Y}$.\footnote{Parreau dealt with the case of $G = \mathbb{Z}$, and worked with sequences instead of ultrafilters.}
    In Theorem \ref{thm:ParreauFactorsAreMultipliers} we generalize this result to show that $\mathcal{U}$-generated systems are multipliers for $\mathcal{U}$-mixing systems.
\end{remark}

\begin{lemma}\label{lem:StabilityOfParreauFactors}
    Given $\mathcal{U} \subseteq G^*$ and an m.p.s. $\mathcal{X}$, the factor $\mathcal{X}^{\mathcal{U}}$ is a $\mathcal{U}$-generated system. In other words, $\mathcal X^{\mathcal U}=(\mathcal X^{\mathcal U})^\mathcal U$.
\end{lemma}

\begin{proof}
By item (ii) in Proposition \ref{2.prop:NormalOperator} and Lemma \ref{3.LemmaDecompositionOfSingleNormal}, we have that for each $p\in \mathcal U$,
$$
\overline{\text{Im}((T^p)^2)}=(\text{Ker}((T^p)^2))^\perp=(\text{Ker}(T^p))^\perp=\overline{\text{Im}(T^p)}.
$$
It now follows from the continuity of $T^p$ and the fact that $L^\infty(X,\mu)$ is dense in $L^2(X,\mu)$ that 
$$\overline{T^p(L^\infty(X,\mu))}=\overline{\text{Im}(T^p)}=\overline{T^p(T^p(L^\infty(X,\mu)))},$$
so the $L^2$-closure of the algebra generated by $\{T^pf\ |\ f \in L^\infty(X,\mu), p \in \mathcal{U}\}$ is the same as the $L^2$-closure of the algebra generated by $\{T^pf\ |\ f \in T^p(L^\infty(X,\mu)), p \in \mathcal{U}\}$, which yields the desired result.
\end{proof}

\begin{lemma}\label{lem:ParreauFactorsAreClosedUnderJoinings}
    Let $I$ be a countable set, and for each $i \in I$ let $\mathcal{U}_i \subseteq G^*$ and let $\mathcal{X}_i$ be a $\mathcal{U}_i$-generated system.
    If $\mathcal{X}$ is a joining of $\{\mathcal{X}_i\}_{i \in I}$, then $\mathcal{X}$ is a $\left(\bigcup_{i \in I}\mathcal{U}_i\right)$-generated system.
\end{lemma}

\begin{proof}
    Let $X = \prod_{i \in I}X_i$, let $\mu = \vee_{i \in I}\mu_i$ be a joining, and let $\pi_i:X\rightarrow X_i$ denote the natural projection map.
    Observe that the algebra generated by $\{f\circ\pi_i\ |\ i \in I\ \&\ f \in L^\infty(X_i,\mu_i)\}$ is dense in $L^2(X,\mu)$.
    Since $\mathcal{X}_i$ is a $\mathcal{U}_i$-generated system, we see that the algebra generated by $\{T^pf\circ\pi_i\ |\ p \in \mathcal{U}_i\ \&\ f \in L^\infty(X_i,\mu_i)\}$ is dense in $(\pi_i)^*L^2(X_i,\mu_i)$.
    It follows that the algebra generated by $\{T^pf\circ\pi_i\ |\ i \in I, p \in \mathcal{U}_i,\ \&\ f \in L^\infty(X_i,\mu_i)\}$ is dense in $L^2(X,\mu)$ as desired.
\end{proof}

\begin{remark}\label{rem:QuasiCharacteristicClass}
    Examples \ref{ex:SecondExample} and \ref{ex:ThirdExample} will show us that $\mathcal{G}_{\mathcal{U}}$ is not closed under factors for some choices of $\mathcal{U}$, hence $\mathcal{G}_{\mathcal{U}}$ is not a characteristic class.
    Let $\overline{\mathcal G_{\mathcal{U}}}$ denote the class of systems that are factors of some system in $\mathcal{G}_{\mathcal{U}}$. 
    We observe that if $\{\mathcal{X}_i\}_{i \in I}$ is a countable collection of systems in $\mathcal{G}_{\mathcal{U}}$, $\mathcal{Y}_i$ is a factor of $\mathcal{X}_i$ for each $i \in I$, and $\mathcal{Y}$ is a joining of $\{\mathcal{Y}_i\}_{i \in I}$, then $\mathcal{Y}$ is a factor of some joining $\mathcal{X}$ of the systems $\{\mathcal{X}_i\}_{i \in I}$.
    Lemma \ref{lem:ParreauFactorsAreClosedUnderJoinings} tells us that $\mathcal{G}_{\mathcal{U}}$ is closed under countable joinings, so $\overline{\mathcal{G}_{\mathcal{U}}}$ is also closed under countable joinings.
    Since $\overline{\mathcal{G}_{\mathcal{U}}}$ is closed under factors by definition, we see that $\overline{\mathcal{G}_{\mathcal{U}}}$ is the smallest characteristic class containing $\mathcal{G}_\mathcal{U}$.
\end{remark}

We now collect a list of examples of various ${\mathcal{U}}$-generated systems and ${\mathcal{U}}$-generated factors.

\begin{example}[Nilsystems]\label{ex:Nilsystems}
    It is now known \cite{SpectrumOfNilsystems} that if $\mathcal{X} = (X,\mathscr{B},\mu,T)$ is an ergodic nilrotation of step $k \ge 2$, then $L^2(X,\mu) = \mathcal{H}_c\oplus\mathcal{H}_L$, where $\mathcal{H}_c$ is the closure of the span of the eigenfunctions, and $\mathcal{H}_L$ has countable Lebesgue spectrum with respect to $T$.
    In this case, we see that for any $p \in \mathbb{Z}^*$, we have that $T^p:\mathcal{H}_c\rightarrow\mathcal{H}_c$ is unitary and that $T^p$ annihilates $\mathcal{H}_L$.
    Consequently, we see that for any $\emptyset \neq \mathcal{U} \subseteq \mathbb{Z}^*$, we have $\mathcal{X}^\mathcal{U} = \mathcal{X}_\mathcal{K}$.
    It is worth recalling that nilrotations are distal systems, hence they are Furstenberg disjoint from all weakly mixing systems.
\end{example}

\begin{example}[Chacon's Transformation]\label{ex:Chacon}
    In \cite[Theorem 5.1]{PowersOfChacon}, it is shown that if $\mathcal{X} = (X,\mathscr{B},\mu,T)$ is the Chacon transformation, then for any $p \in \mathbb{Z}^*$, we either have that $T^p = \Theta$, the orthogonal projection onto the space of constants, or $T^p = \sum_{j = 1}^ka_jT^{b_j}$ for some $k \in \mathbb{N}$, some $b_1,\cdots,b_k \in \mathbb{Z}$, and some $a_1,\cdots,a_k \ge 0$ satisfying $\sum_{j = 1}^ka_j = 1$.
    It follows that for any $p \in \mathbb{Z}^*$, $\mathcal{X}$ is either $p$-mixing, or it is its own $p$-generated factor.
    To see this, pick $p \in \mathbb{Z}^*$ for which $T^p = \sum_{j = 1}^ka_jT^{b_j}$ as above.
    We will show the stronger statement that $\mathcal{X}$ is spectrally disjoint from all $p$-mixing systems. 
    Indeed, the spectral theorem tells us that on each cyclic subspace orthogonal to the space of constants $T^p$ is unitarily isomorphic to the multiplication operator $(Mf)(x) = (\sum_{j = 1}^ka_je^{2\pi i b_jx})f(x)$ on $L^2(\mathbb{T},\nu)$, where $\nu$ is a representative of the maximal spectral type of the subspace.
    We see that $Mf = 0$ if and only if $f$ is supported on the $x \in \mathbb{T}$ for which $\sum_{j = 1}^ka_je^{2\pi i b_jx} = 0$, which is a finite set.
    It is well known that $\mathcal{X}$ is weakly mixing \cite{OriginalChacon}, hence the spectral measure $\nu$ has no atoms, thus $M$ and $T^p$ are injective.
    The claim now follows from Corollary \ref{cor:SpectralDisjointnessFromSM}.
\end{example}

\begin{example}[The Rudin-Shapiro Subshift]\label{ex:RudinShapiro}
    Let $\mathcal{X}$ denote the subshift associated to the Rudin-Shapiro sequence.
    %The subshift on 4 symbols is topologically isomorphic to the subshift on 2 symbols, so it doesn't matter which one we use (see \cite[Proposition 1]{NewPropertyOfRS}).
    $\mathcal{X}$ is $\frac{1}{2}$-rigid, and it is a two-point extension of $\mathcal{X}_\mathcal{K}$, its Kronecker factor.
    Furthermore, the orthocomplement of the Kronecker factor has Lebesgue spectrum of multiplicity 2 (see \cite{NewPropertyOfRS}), so for any $\emptyset \neq \mathcal{U} \subseteq \mathbb{Z}^*$, we have $\mathcal{X}^\mathcal{U} = \mathcal{X}_\mathcal{K}$.   
\end{example}

\begin{example}\label{ex:FirstExample}
    There exists a $\{p,q\}$-generated system that has a strongly mixing component to its spectrum and is not generated by a singleton ${\mathcal{U}}=\{r\}$.
    We construct the example when $G = \mathbb{Z}$, and we mention that a related construction is given in \cite[Example 5.5]{MildMixingForLCGroups}\footnote{In \cite[Example 5.5]{MildMixingForLCGroups} two rigid systems $\mathcal{X}$ and $\mathcal{Y}$ are constructed for which $L^2(X\times Y,\mu\times\nu)\ominus\left(L^2(X,\mu)\oplus L^2(Y,\nu)\right)$ has Lebesgue spectrum.}.
   
   It is known that the generic system $\mathcal{X}$ is simultaneously rigid and weakly mixing (see \cite[Proposition 2.16]{RigidityAndNonRecurrence}, for example), so let us fix such a system $\mathcal{X}$.
   Since $\mathcal{X}$ is weakly mixing, there exists a set $D_X \subseteq \mathbb{N}$ of natural density $1$, such that for all $f \in L^2_0(X,\mu)$ we have

   \begin{equation}
       \lim_{\underset{n \in D_X}{n\rightarrow\infty}}T^nf = 0
   \end{equation}
   in the weak topology of $L^2(X,\mu)$.
   Ryzhikov \cite{MixingSetsForRigidTransformations} showed that for any infinite set $M \subseteq \mathbb{N}$ of natural density $0$, there exists a weakly mixing rigid system that is mixing along $M$, so let $\mathcal{Y}$ be such a system for $M = D_X^c$.
   Let $(n_k)_{k = 1}^\infty,(m_k)_{k = 1}^\infty$ be rigidity sequences for $\mathcal{X}$ and $\mathcal{Y}$ respectively.
   Let $p,q \in \mathbb{Z}^*$ be supported on $\{n_k\}_{k = 1}^\infty$ and $\{m_k\}_{k = 1}^\infty$ respectively.
   We see that $T^p$ and $S^q$ are both identity operators, so $\mathcal{X}\times\mathcal{Y}$ is a $\{p,q\}$-generated system.

   To see that $\mathcal{X}\times\mathcal{Y}$ is not an $r$-generated system, let $r \in \mathbb{Z}^*$ be arbitrary.
   We have that either $T^r$ or $S^r$ is the orthogonal projection onto the space of constants, so let us assume without loss of generality that it is $T^r$.
   We see that for any function $f(x,y) \in L^2(X\times Y,\mu\otimes\nu)$ of the form $f(x,y) = f_1(x)f_2(y)$ with $f_1 \in L^2_0(X,\mu)$ and $f_2 \in L^2(Y,\nu)$, we have $(T\times S)^rf = 0$.
   Consequently, we have that $(T\times S)^r(L^\infty(X\times Y,\mu\otimes\nu)) \subseteq \{1\}\otimes L^2(Y,\nu)$, hence the $r$-generated factor of $\mathcal{X}\times\mathcal{Y}$ is a factor of $\mathcal{Y}$.
   In particular, we see that the unitary action of $T\times S$ on $L^2(X\times Y,\mu\otimes\nu)\ominus\left(L^2(X,\mu)\oplus L^2(Y,\nu)\right)$ is strongly mixing.
\end{example}

\begin{remark}\label{rem:CharacteristicClosureOfRigid}
    The system $\mathcal X\times\mathcal Y$ constructed in Example \ref{ex:FirstExample} is not a rigid system since its spectrum has a strongly mixing component.
    Nonetheless, $\mathcal X\times\mathcal Y$ is the direct product of two rigid systems, so we could even take $p$ and $q$ to be idempotents.

    Recalling that $\mathcal{R}$ denotes the class of rigid systems, we have $\mathcal{R} = \bigcup_{p \in E(G^*)}\mathcal{G}_p$.
    While $\mathcal{R}$ is closed under factors, Example \ref{ex:FirstExample} shows us that $\mathcal{R}$ is not closed under joinings, so it is not a characteristic class.
    Using Lemma \ref{lem:ParreauFactorsAreClosedUnderJoinings} we see that the smallest characteristic class generated by $\mathcal{R}$ is $\overline{\mathcal{G}_{E(G^*)}}$.
    While the systems in $\overline{\mathcal{G}_{E(G^*)}}$ are not always spectrally disjoint from all mildly mixing systems, we will see in Corollary \ref{cor:MainMultiplierResult} that $\overline{\mathcal{G}_{E(G^*)}} \subseteq \mathscr{M}(\mathcal{M}^\perp)$. 
\end{remark}

\begin{example}\label{ex:SecondExample}
    There exists an ultrafilter $p\in\Z^*$ and a $p$-generated system $\mathcal{X}_C$ that is not spectrally disjoint from all $p$-mixing systems.
    Furthermore, the system $\mathcal{X}_C$ will possess a factor that is not a $p$-generated system.

    Let $C \subseteq [0,1]$ denote the set of real numbers whose base $4$ expansion only has the digits 0 and 1, and let $\mu$ be the uniform Cantor measure on $C$.
    Viewing the digits of $x \in C$ as an i.i.d. sequence of $\{0,1\}$-valued random variables, we see that for any $n \in \mathbb{Z}\setminus\{0\}$, we have

    \begin{equation}
        \widehat{\mu}(n) = \prod_{k = 1}^\infty\left(\frac{1+e^{-2\pi i n\cdot4^{-k}}}{2}\right).
    \end{equation}
    In particular, we see that $\widehat{\mu}(4n) = \widehat{\mu}(n)$ for all $n \in \mathbb{Z}$.
    We also see that $\widehat{\mu}(n) = 0$ if and only if the first nonzero digit in the base $4$ expansion of $|n|$ is $2$.
    Lastly, we see that for all $m,n \in \mathbb{Z}$ we have

    \begin{alignat*}{2}
        \lim_{j\rightarrow\infty}\widehat{\mu}(m\cdot4^j+n) & = \lim_{j\rightarrow\infty}\prod_{k = 1}^\infty\left(\frac{1+e^{-2\pi i(m\cdot4^j+n)4^{-k}}}{2}\right)\\
        &= \lim_{j\rightarrow\infty}\prod_{k = 1}^j\left(\frac{1+e^{-2\pi in\cdot4^{-k}}}{2}\right)\prod_{k = j+1}^\infty\left(\frac{1+e^{-2\pi i(m\cdot4^j+n)4^{-k}}}{2}\right) = \widehat{\mu}(n)\widehat{\mu}(m).
    \end{alignat*}
    
    Now consider the measure-preserving system $\mathcal{X}_C := ([0,1]^2,\mathscr{L}^2,\mu\otimes m,T)$, where $T(x,y) = (x,y+x)$, and let $p \in \mathbb{Z}^*$ be supported on $\{4^j\}_{j = 1}^\infty$.
    We see that for any $a,b,c,d\in\Z$,
\begin{multline*}
        \int_{[0,1)^2}e^{2\pi i(ax+by)}T^pe^{2\pi i(cy+dx)}\text{d}\mu(x)\text{d}m(y) = \plimgG{p}{n}{\Z}
        \int_{[0,1)^2}e^{2\pi i(ax+by)}e^{2\pi i(cy+cnx+dx)}
        \text{d}\mu(x)\text{d}m(y)\\
        =\left(\int_{[0,1)}e^{2\pi i(b+c)y}\text{d}m(y)\right)\left(\plimgG{p}{n}{\Z}
        \int_{[0,1)}e^{2\pi i (cn+a+d)x}\text{d}\mu(x)\right)\\=\mathbbm{1}_{\{-c\}}(b)
       \left(\plimgG{p}{n}{\Z}
       \hat\mu(-cn-a-d)\right) =\mathbbm{1}_{\{-c\}}(b)
       \hat\mu(-c)\hat\mu(-a-d)\\
       = \hat\mu(-c)\int_{[0,1)^2}e^{2\pi i(ax+by)}e^{2\pi i(cy+dx)}\text{d}\mu(x)\text{d}m(y),
\end{multline*}
so for any $d,c\in \Z$, $T^p(e^{2\pi i(dx+cy)})=\hat\mu(-c)e^{2\pi i(dx+cy)}$. 

    We now observe that $L^2([0,1]^2,\mu\otimes m)$ contains the $p$-mixing vector $e^{4\pi iy}$, so $\mathcal{X}_C$ is not spectrally disjoint from all $p$-mixing systems.
    We also see that $T^pe^{2\pi ix} = e^{2\pi ix}$ and $T^pe^{2\pi iy} = \widehat{\mu}(-1)e^{2\pi iy}$, so $\mathcal{X}_C$  is its own $p$-generated factor.
    
    Now let $\mathcal{X}_C'$ be the factor of $\mathcal{X}_C$ corresponding to the algebra generated by $1,e^{2\pi ix},$ and $e^{4\pi iy}$.
    We see that $T^pe^{2\pi i(2+4n)y} = 0$ for all $n \in \mathbb{Z}$, so the $p$-generated factor of $\mathcal{X}_C'$ is the factor of $\mathcal{X}_C$ generated by $1,e^{2\pi ix},$ and $e^{8\pi iy}$, which is a proper factor of $\mathcal{X}_C'$.
\end{example}

\begin{example}\label{ex:ThirdExample}
    We will now give an example of a $\{p,q\}$-generated system $\mathcal{X}_C\times\mathcal{X}_{C,2}$ that has a Lebesgue component to its spectrum, and also has a factor that is not a $\mathbb{Z}^*$-generated system.
    This is a continuation of Example \ref{ex:SecondExample}.

    Let $\mathcal{X}_{C,2} := ([0,1]^2,\mathscr{L}^2,\mu\times m,T^2)$, so $\mathcal{X}_C\times\mathcal{X}_{C,2} = ([0,1]^4,\mathscr{L}^4,\mu^2\otimes m^2,S)$, where $S(w,x,y,z) = (w,x,y+w,z+2x)$.
    Let $p \in \mathbb{Z}^*$ be any ultrafilter supported on $(4^k)_{k = 1}^\infty$.
    For $z \in \mathbb{Z}$ and $A \subseteq \mathbb{Z}$, let $zA = \{za\ |\ a \in A\}$, let $zp := \{A \subseteq \mathbb{Z}\ |\ \{n \in \mathbb{Z}\ |\ zn \in A\} \in p\}$, and observe that $4p \in \mathbb{Z}^*$ is also supported on $(4^k)_{k = 1}^\infty$.
    It is clear that $S^pe^{2\pi iw} = e^{2\pi iw}$ and $S^pe^{2\pi ix} = e^{2\pi ix}$.
    We have also already seen that $S^pe^{2\pi iy} = \widehat{\mu}(-1)e^{2\pi iy}$.
    We observe that

    \begin{equation}
        S^{2p}e^{2\pi iz} = \plimgG{2p}{n}{\Z}e^{2\pi i(z+2nx)} = \plimgG{4p}{n}{\Z}e^{2\pi i(z+nx)} = \widehat{\mu}(-1)e^{2\pi iz},
    \end{equation}
    so $\mathcal{X}_C\times\mathcal{X}_{C,2}$ is a $\{p,2p\}$-generated system.

    Next, we observe that for $t \in \mathbb{Z}\setminus\{0\}$, $n \neq 0$, and $f_t(w,x,y,z) := e^{2\pi it(y+z)}$, we have

    \begin{equation}
        \langle S^nf_t,f_t\rangle = \int_{[0,1]^2}e^{2\pi i(tnw+2tnx)}d\mu^2 = \widehat{\mu}(-tn)\widehat{\mu}(-2tn) = 0,
    \end{equation}
    where the last equality follows from the fact that either $|tn|$ or $2|tn|$ has $2$ as the first nonzero digit in its base $4$ expansion.
    Consequently, we see that each of the functions $f_t$ has Lebesgue spectrum.

    Lastly, we let $(\mathcal{X}_C\times\mathcal{X}_{C,2})'$ be the smallest $S$-invariant factor of $\mathcal{X}_C\times\mathcal{X}_{C,2}$ generated by $1, e^{2\pi ix}, e^{2\pi iw},$ and $f_1$.
    Letting $f_0 = 1$, we see that $f_tf_s = f_{t+s}$ for all $t,s \in \mathbb{Z}$.
    Since for $t\neq 0$, $f_t$ has Lebesgue spectrum, we see that for any $g \in L^\infty([0,1]^4,\mu^2\otimes m^2)$ of the form $g(w,x,y,z) = g'(w,x)$, any $t \in \mathbb{Z}\setminus\{0\}$, and any $p \in \mathbb{Z}^*$, we have $S^p(f_tg) = gS^pf_t = 0$.
    It follows that for any $p \in \mathbb{Z}^*$, $S^p$ is a projection onto a proper, nontrivial $S$-invariant factor of $(\mathcal{X}_C\times\mathcal{X}_{C,2})'$, hence $(\mathcal{X}_C\times\mathcal{X}_{C,2})'$ is not a $\mathbb{Z}^*$-generated system.
\end{example}

\begin{example}\label{ex:FourthExample}
    We will now give an example of a $p$-generated system $\mathcal{X}_C^2$ that has a factor that is not a $\mathbb{Z}^*$-generated system.
    This is a continuation of Examples \ref{ex:SecondExample} and \ref{ex:ThirdExample}.
    
    Consider $\mathcal{X}_C^2 = ([0,1]^4,\mathscr{L}^4,\mu^2\times m^2,R)$, where $R(w,x,y,z) = (w,x,y+w,z+x)$.
    Since $\mathcal{X}_C$ is a $p$-generated system, Lemma \ref{lem:ParreauFactorsAreClosedUnderJoinings} tells us that $\mathcal{X}_C^2$ is also a $p$-generated system.
    We see that $\mathcal{X}_C\times\mathcal{X}_{C,2}$ is a factor of $\mathcal{X}_C^2$ via the factor map 
    $\pi(w,x,y,z) = (w,x,y,2z)$.
    It follows that $(\mathcal{X}_C\times\mathcal{X}_{C,2})'$ is also a factor of $\mathcal{X}_C^2$, from which the desired result follows.
\end{example}
%%%%%%%%%%%%%%%%%%%%%%%%%%%%%%%%%%%%%%%%%%%%%%%%%%%%%%%%%%%%%%%%%%%%%%%%%%%%%%%%%%%%%%%%%%%%%%%%%%%%%%%%%%%%%%%%%%%%%%%%%%%%%%%%%
\section{Partially rigid systems}\label{sec:PartialRigidity}

\begin{definition}\label{PartialMixingAndRigidityDefinition}
Let $\alpha \in (0,1]$, $G$ be an abelian group, $\mathcal{X}$ a system, and $p \in G^*$ an ultrafilter.
$\mathcal{X}$ is \textbf{$\alpha$-rigid along $p$} if $\plimgG{p}{g}{G}\mu(A\cap T^gA) \ge \alpha\mu(A)$ for all $A \in \mathscr{B}$.
If $\alpha \in [0,1]$ and $\plimgG{p}{g}{G}\mu(A\cap T^gB) = \alpha\mu(A)\mu(B)+(1-\alpha)\mu(A\cap B)$ for all $A,B \in \mathscr{B}$, then $\mathcal{X}$ is \textbf{$\alpha$-weakly mixing along $p$}.\footnote{St{\"e}pin \cite{PartiallyRigidAndMixingIsGeneric} called our $\alpha$-weakly mixing systems $\alpha$-mixing. 
We chose to instead follow the terminology of \cite{ErgodicTheoryViaJoinings}. 
If $\alpha \in (0,1)$, then $\alpha$-weak mixing implies weak mixing but not strong mixing.
We also choose to use the term $\alpha$-partially mixing instead of $\alpha$-mixing to avoid confusion with various notions of mixing that are discussed in probability theory \cite{MixingInProbabilitySurvey}.}
If $\mathcal{X}$ is $\alpha$-rigid ($\alpha$-weakly mixing) along some $p$, then $\mathcal{X}$ is \textbf{$\alpha$-rigid ($\alpha$-weakly mixing)}.
If $\mathcal{X}$ is $\alpha$-rigid for some $\alpha \in (0,1]$, then $\mathcal{X}$ is \textbf{partially rigid}.
$\mathcal{X}$ is \textbf{$\alpha$-partially mixing along $p$} if for every $A,B \in \mathscr{B}$ we have $\plimgG{p}{g}{G}\mu(A\cap T^gB) \ge \alpha\mu(A)\mu(B)$.
$\mathcal{X}$ is \textbf{$\alpha$-partially mixing} if for every $A,B \in \mathscr{B}$ we have $\liminf_{g\rightarrow\infty}\mu(A\cap T^gB) \ge \alpha\mu(A)\mu(B)$.
\end{definition}

The notions in Definition \ref{PartialMixingAndRigidityDefinition} are well-studied in the case of $\mathbb{Z}$-systems.
St{\"e}pin \cite{PartiallyRigidAndMixingIsGeneric} showed that for any $\alpha \in (0,1)$, a generic system $\mathcal{X}$ is $\alpha$-weakly mixing (see also \cite{GenericSpectralProperties,MixingAndRigidRigo,zelada2025coexistence}). It follows by a standard approximation argument that a generic system $\mathcal{X}$ is $\alpha$-weakly mixing for \textit{every} $\alpha\in [0,1]$.
Friedman \cite{PartialMixingAndPartialRigidity} constructed a concrete example of such a system $\mathcal{X}$, and we remark that such a system is also rigid.
Dekking and Keane \cite{SubstitutionsAreNotMixing} showed that the minimal subshifts corresponding to a substitution system are not strongly mixing, and in fact, they even showed that such systems are $\alpha$-rigid (see also \cite{RateOfRigidtyForSubstitutions}).
Katok \cite{IETsAreNotMixing} showed that interval exchange transformations are not strongly mixing, then Danilenko \cite{DanilenkoPartiallyRigid} and Ryzhikov \cite{RyzhikovAbsenceOfMixing} independently observed that a small modification of Katok's proof shows that ergodic IETs are $\alpha$-rigid.
Recently, Creutz \cite{ComplexityAndStrongMixing} showed that ergodic subshifts of linear complexity are partially rigid.
Danilenko \cite{DanilenkoPartiallyRigid} as well as Bruin, Karpel, and Oprocha \cite{RigidtyAndToeplitzSystems} also showed that many other systems are $\alpha$-rigid.
Ryzhikov \cite{NotPartiallyRigidCharacterization} has a characterization of systems that are not $\alpha$-rigid.

Lema\'nczyk and Fr\k{a}czek \cite{IETsAreDisjointFromSM} showed that ergodic $\alpha$-rigid systems and many other systems are always disjoint from strongly mixing systems.
Their techniques were also used by Ku{\l}aga \cite[Proposition 6(ii)]{IsomorphismOfCartesianProducts} to show that ergodic systems that are $\alpha$-rigid along some $p$ are disjoint from systems that are $\beta$-partially mixing along $p$ if $\beta > 1-\alpha$.\footnote{The results of Lema\'nczyk, Fr\k{a}czek, and Ku{\l}aga are stated and proven for measurable flows, i.e., $\mathbb{R}$-systems. Nonetheless, it is easy to deduce from them the analogous result for $\mathbb{Z}$-actions.}
We give a proof of this result for a general $G$-action without the assumption of ergodicity.

\begin{theorem}\label{thm:PartiallyRigidMixingDisjointness}
    If $\mathcal{X}$ is $\alpha$-rigid along $p$ and $\mathcal{Y}$ is $\beta$-partially mixing along $p$ with $\alpha+\beta > 1$, then $\mathcal{X}$ and $\mathcal{Y}$ are disjoint.
\end{theorem}

\begin{proof}
    Since $\mathcal{X}$ is $\alpha$-rigid along $p$ and $\mathcal{Y}$ is $\beta$-partially mixing along $p$, we have $T^p = \alpha I+(1-\alpha)M_1$ and $S^p = \beta\Theta+(1-\beta)M_2$ for some Markov operators $M_1,M_2$.
    Let $\lambda \in J^e(\mathcal{X},\mathcal{Y})$ and let $J:L^2(X,\mu)\rightarrow L^2(Y,\nu)$ be the corresponding Markov operator.
    We see that 

    \begin{equation}\label{eq:PartialRigidPerpPartialMixing}
        \beta\Theta J+(1-\beta)M_2J = S^pJ = JT^p = \alpha J+(1-\alpha)JM_1.
    \end{equation}
    We observe that for $f \in L^2(X,\mu)$, we have $\Theta Jf = \int_Xfd\mu \in L^2(Y,\nu)$.
    Since $\mathcal{Y}$ is $\beta$-partially mixing, it is weakly mixing, so Lemma \ref{lem:ProductIsExtreme} tells us that $\mu\otimes\nu \in J^e(\mathcal{X},\mathcal{Y})$.
    Since $\Theta J$ is the Markov operator corresponding to $\mu\otimes\nu$, we see that $\Theta J$ is indecomposable. 
    Now let $\eta,\rho \in J(\mathcal{X},\mathcal{Y})$ correspond to $M_2J$ and $JM_1$ respectively.
    Let $\pi$ be the factor map from $\mathcal{X}$ to its invariant factor, and let $(\mu_\omega)_{\omega \in \Omega}$ be the ergodic decomposition of $\mathcal{X}$.
    Since the disintegration of a measure with respect to a factor map is unique, we can first rewrite Equation \eqref{eq:PartialRigidPerpPartialMixing} by replacing each Markov operator with its corresponding joining measure, then we can take the disintegration with respect to $\pi\circ\pi_X$ to see that for a.e. $\omega$ we have

    \begin{equation}\label{eq:Disintegrated}
        \beta\mu_\omega\otimes\nu+(1-\beta)\eta_\omega = \alpha\lambda_\omega+(1-\alpha)\rho_\omega.
    \end{equation}
    Lemma \ref{lem:DisjointnessAndErgodicDecomposition}(i) shows us that Equation \eqref{eq:Disintegrated} takes place in $J(\mathcal{X}_\omega,\mathcal{Y})$, which is a Choquet simplex.
    Lemma \ref{lem:DisjointnessAndErgodicDecomposition}(ii) shows us that $\lambda_\omega \in J^e(\mathcal{X}_\omega,\mathcal{Y})$ for a.e. $\omega$, and we similarly see that $\mu_\omega\otimes\nu \in J^e(\mathcal{X}_\omega,\mathcal{Y})$ for a.e. $\omega$ since $\mathcal{X}_\omega$ is ergodic and $\mathcal{Y}$ is weakly mixing.
    Since $\alpha+\beta > 1$, we must have $\mu_\omega\otimes\nu = \lambda_\omega$ for a.e. $\omega$, hence $\lambda = \mu\otimes\nu$.
\end{proof}
\begin{remark}
The use of Lemma \ref{lem:ProductIsExtreme} instead of Lemma \ref{lem:DisjointnessAndErgodicDecomposition}(iii) in the proof of Theorem \ref{thm:PartiallyRigidMixingDisjointness} is motivated by the fact that,  in general, the property of $\alpha$-rigidity of an invertible measure-preserving system $\mathcal X$ along a sequence $(n_k)_{k\in\N}$ in $\N$ is not inherited by each of its ergodic components.
For an example of such a system, consider the set  $X=[0,1)^2$, the $\sigma$-algebra $\mathcal B=\text{Borel}([0,1)^2)$, and let $T:X\rightarrow X$ be defined by $T(x,y)=(x,y+x)\mod 1$. 
By \cite[Theorem 2]{fayad2015rigidityNonTorusRecurrent} there exists an increasing sequence $(n_k)_{k\in\N}$ in $\N$ and a continuous probability measure $\sigma$ on $[0,1)$ with the properties that (1) $\lim_{k\rightarrow\infty}\int_{[0,1)}e^{2\pi in_kx}\text{d}\sigma(x)=1$ and (2) for every irrational $\theta\in \mathbb R$, the sequence $(n_k\theta\mod 1)_{k\in\N}$ is dense in $[0,1]$.
It follows that the measure-preserving system $\mathcal X=(X, \mathcal B,\sigma\otimes m, T)$ is $1$-rigid along the sequence $(n_k)_{k\in\N}$ but for $\sigma$-almost every $x\in [0,1)$, the ergodic component $(\{x\}\times [0,1), \text{Borel}(\{x\}\times[0,1)),\delta_x\otimes m, T)$ is not $1$-rigid along $(n_k)_{k\in\N}$.
\end{remark}

\begin{theorem}\label{thm:PartiallyRigidFiniteFibers}
    If $\alpha \in (0,1],\ p \in G^*,$ and $\mathcal{X}$ is a system that is $\alpha$-rigid along $p$, then $\mathcal{X}$ is a finite extension of $\mathcal{X}^p$.
    Furthermore, if $\mathcal{X}$ is ergodic, then $\mathcal{X}$ is an $n$-point extension of $\mathcal{X}^p$ for some $n \le \lfloor\alpha^{-1}\rfloor$.
\end{theorem}

\begin{proof}
    Let $\mathcal{Y} = \mathcal{X}^p$, and let $\mu = \int_Y\mu_yd\nu(y)$ be the disintegration of $\mu$ with respect to $\mathcal{Y}$.
    Since $\mathcal{X}$ is $\alpha$-rigid along $p$, we have $T^p = \alpha I+(1-\alpha)M$ for some Markov operator $M:L^2(X,\mu)\rightarrow L^2(X,\mu)$.
    For $A \in \mathscr{B}^+$ and $y \in Y$, let $c_A(y) := \mu_y(A)$, and let $E_A := \{y\ |\ 0 < c_A(y) < \alpha\}$.
    Suppose for the sake of contradiction that $\nu(E_A) > 0$ for some $A \in \mathscr{B}^+$.
    Let $E := E_A$ and let $A' = A\cap\pi^{-1}(E)$, where $\pi:X\rightarrow Y$ is the factor map.
    We see that $0 < \mu_y(A') < \alpha$ for every $y \in E$, and that $\mu_y(A') = 0$ for every $y \notin E$, so

    \begin{equation}
        0 < \mu(A') = \int_Y\mu_y(A')d\nu(y) = \int_E\mu_y(A')d\nu(y) < \alpha\nu(E).
    \end{equation}
    Let $h := T^p\mathbbm{1}_{A'} = \alpha\mathbbm{1}_{A'}+(1-\alpha)M\mathbbm{1}_{A'}$ and observe that $h$ is $\mathscr{B}_p = \mathscr{A}$-measurable, so $h$ is constant on $\pi^{-1}(\{y\})$ for $\nu$-a.e. $y \in Y$, hence we may define $h' \in L^\infty(Y,\nu)$ by $h'(y) = h(x)$ for $x \in \pi^{-1}(\{y\})$.
    We note that $h'(y) \ge \alpha$ for $\nu$-a.e. $y \in E$.
    Recalling that $T^p$ is a Markov operator, we see that

    \begin{equation}
        \mu(A') = \int_XT^p\mathbbm{1}_{A'}d\mu = \int_Yh'd\nu \ge \alpha\nu(E), 
    \end{equation}
    which yields the desired contradiction.

    Since $\mathscr{B}$ is countably generated modulo $\mu$, we have shown that for $\nu$-a.e. $y$ the probability measure $\mu_y$ takes values in $\{0\}\cup[\alpha,1]$.
    This means that $\mu_y$ is a purely atomic measure with at most $\alpha^{-1}$ atoms.

    Now let us further assume that $\mathcal{X}$ is ergodic. 
    Since $(T^g)_*\mu_y=\mu_{T^gy}$ for every $g \in G$, it follows that the number $n(y)$ of atoms of $\mu_y$ and the vector of weights $y\mapsto(\mu_y(\{1\}),\cdots,\mu_y(\{n(y)\}))$ are both $T$-invariant functions.
    By ergodicity, both functions are constant $\nu$-a.e., from which it is also immediate that $\mu_y(\{i\}) = \mu_y(\{j\})$ for $\nu$-a.e. $y$ and every $1 \le i,j \le n := n(y)$.
\end{proof}

\begin{proposition}
    Fix $\alpha \in (0,1), p \in G^*,$ and a system $\mathcal{X}$.
    \begin{enumerate}[(i)]
        \item If $\alpha > \frac{1}{2}$ and $\mathcal{X}$ is $\alpha$-rigid along $p$, then $T^p$ is injective.\footnote{Example \ref{ex:RudinShapiro} shows that the lower bound for $\alpha$ is sharp. Part (i) was also proven in the case of $G = \mathbb{Z}$ as \cite[Theorem 5]{SpectrumOfAlphaRigidMaps}.}

        \item If $\mathcal{X}$ is $\alpha$-weakly mixing along $p$, then $T^p$ is injective on $L^2_0(X,\mu)$.
    \end{enumerate}
    In both cases, $\mathcal{X}$ is spectrally disjoint from all $p$-mixing systems, and $\mathcal{X}^p = \mathcal{X}$.
\end{proposition}

\begin{proof}
    Part (i) follows from the observation that $T^p = \alpha I+(1-\alpha)M$ for some Markov operator $M$, so $||T^pf|| \ge \alpha||f||-(1-\alpha)||f|| > 0$ for all $0 \neq f \in L^2(X,\mu)$.
    Part (ii) follows from the fact that $T^p = (1-\alpha)I$ on $L^2_0(X,\mu)$.
    The last statement follows from Corollary \ref{cor:SpectralDisjointnessFromSM}.
\end{proof}

\begin{remark}
    Terry Adams witnessed Nathaniel Friedman describing for each $\alpha \in (0,1)$ a $\mathbb{Z}$-system $\mathcal{X}_\alpha$ that is $\alpha$-partially mixing and $(1-\alpha)$-rigid.
    Such a system $\mathcal{X}_\alpha$ is also $\alpha$-weakly mixing, hence it is spectrally disjoint from all strongly mixing systems.
    This construction can also be accomplished using a modified staircase transformation where infinitely often a partial staircase is inserted after including subcolumns with no spacers.
    Friedman and Ornstein \cite{AlphaButNotAlphaPlusEpsilonPM} constructed for each $\alpha \in (0,1)$ an example of a $\mathbb{Z}$-system that is $\alpha$-partially mixing, but is not $(\alpha+\epsilon)$-partially mixing for any $\epsilon > 0$.
    We see that $\mathcal{X}_\alpha$ gives a new example of a system that is $\alpha$-partially mixing, but is not $(\alpha+\epsilon)$-partially mixing for any $\epsilon > 0$.
\end{remark}

\begin{remark}\label{rem:PartialMixingIsNotUMixing}
    For $\alpha \in (0,1)$, let $\mathcal{S}_\alpha$ denote the collection of $\alpha$-partially mixing systems, and let $\mathcal{X}_\alpha$ be as above.
    We see that for any sequence $(\alpha_n)_{n = 1}^\infty \subseteq (0,1)$ and any $\epsilon > 0$, the system $\mathcal{X} := \prod_{n = 1}^\infty\mathcal{X}_{\alpha_n}$ is not $(\alpha+\epsilon)$-partially mixing, where $\alpha = \prod_{n = 1}^\infty\alpha_n$. 
    It follows that the classes $\mathcal{S}_\alpha$ for $\alpha \in (0,1)$ as well as the class $\bigcup_{\alpha > 0}\mathcal{S}_\alpha$ are not closed under countable direct products, so they cannot be represented as $\mathcal{S}_\mathcal{U}$ for any $\mathcal{U} \subseteq \mathbb{Z}^*$.
\end{remark}

\begin{remark}
    While Theorem \ref{thm:PartiallyRigidFiniteFibers} shows us that partially rigid systems are closely related to their ${\mathcal{U}}$-generated factors, Thierry de la Rue and Emmanuel Roy have constructed (private communication) an example of a $\mathbb{Z}$-system that is disjoint from every strongly mixing system, but has no partially rigid factor.
    Consequently, there exists a ${\mathcal{U}}$-generated system that has no partially rigid factor.
\end{remark}
%%%%%%%%%%%%%%%%%%%%%%%%%%%%%%%%%%%%%%%%%%%%%%%%%%%%%%%%%%%%%%%%%%%%%%%%%%%%%%%%%%%%%%%%%%%%%%%%%%%%%%%%%%%%%%%%%%%%%%%%%%
\section{The Multiplier Property}\label{sec:MultipliersMainResults}

Let $\widehat{G}$ denote the Pontryagin dual of $G$, and let $\mathcal{M}(\widehat{G})$ denote the space of finite complex Borel measures on $\widehat{G}$.
$\mathcal{M}(\widehat{G})$ is naturally a commutative ring with identity $\delta_{1_{\widehat{G}}}$ when endowed with the operation of convolution for multiplication.
An \textbf{$L$-ideal} $I$ of $\mathcal{M}(\widehat{G})$ is an algebraic ideal that is closed with respect to the total-variation norm, and for which we have $\nu \in I$ whenever $\mu \in I$ and $\nu \ll \mu$.
The following $L$-ideals of $\mathcal{M}(\widehat{G})$ are of interest.
\begin{enumerate}
    \item $\mathcal{M}_c(\widehat{G})$, the $L$-ideal of continuous measures.

    \item Given $\emptyset \neq \mathcal{U} \subseteq G^*$, we define $\mathcal{M}_\mathcal{U}(\widehat{G})$ to be the $L$-ideal of $\mathcal{U}$-mixing measures, which are those measures $\mu \in \mathcal{M}(\widehat{G})$ for which $\plimgG{q}{g}{G}\widehat{\mu}(g) = 0$ for every $q \in \beta G+\mathcal{U}$.

    \item $\mathcal{M}_r(\widehat{G})$, the $L$-ideal of Rajchman measures. 
    We recall that $\mu \in \mathcal{M}(\widehat{G})$ is Rajchman if for every $p \in G^*$ we have $\plimgG{p}{g}{G}\widehat{\mu}(g) = 0$ (i.e. $\lim_{g\rightarrow\infty}\widehat{\mu}(g) = 0$).

    \item $\mathcal{M}_\mathcal{A}(\widehat{G})$, the $L$-ideal of measures that are absolutely continuous with respect to the Haar measure of $\widehat{G}$.
    We recall that when $G = \mathbb{Z}$, this is the ideal of measures that are absolutely continuous with respect to the Lebesgue measure.
\end{enumerate}

\begin{theorem}
   Let $\mathcal{X}$ be a system and $[\mathcal{X}]$ its maximal spectral type.
    \begin{enumerate}[(i)]
        \item $\mathcal{X}$ is weakly mixing if and only if  $[\mathcal{X}] \subseteq \mathcal{M}_c(\widehat{G})$.

        \item $\mathcal{X}$ is $\mathcal{U}$-mixing if and only if  $[\mathcal{X}] \subseteq \mathcal{M}_\mathcal{U}(\widehat{G})$.

        \item $\mathcal{X}$ is strongly mixing if and only if  $[\mathcal{X}] \subseteq \mathcal{M}_r(\widehat{G})$.

        \item $\mathcal{X}$ has absolutely continuous spectrum if and only if  $[\mathcal{X}] \subseteq \mathcal{M}_\mathcal{A}(\widehat{G})$.
    \end{enumerate}
\end{theorem}
\begin{proof}
    As mentioned in Section \ref{ssec:Spectrum}, part (i) is a classical consequence of the spectral theorem, and part (iv) follows by definition.
    Since $\mathcal{M}_r(\widehat{G}) = \mathcal{M}_{G^*}(\widehat{G})$, we see that part (iii) follows from part (ii).

    We now prove part (ii). For every $f\in L_0^2(X,\mu)$, let $\mu_f$ denote the only member of $\mathcal M(\widehat G)$ defined by the property that for every $g\in G$, 
      $$
\langle T^g f,f\rangle =\int_{\widehat G}\chi(g)\text{d}\mu_f(\chi).
    $$
    
    First assume that $\mathcal X$ is $\mathcal U$-mixing. 
    \cref{prop:DifferentFamiliesForSameUmixing} tells us that $\mathcal{X}$ is also $(\beta G+\mathcal{U})$-mixing.
    Since $\mathcal X$ is separable, there exists an $f\in L_0^2(X,\mu)$ for which $\mu_f\in[\mathcal X]$. 
    Thus, for any $q\in \beta G+\mathcal U$, $\plimgG{q}{g}{G}\widehat\mu_f(g)=\plimgG{q}{g}{G} \langle T^g f,f\rangle=0$. 
    It follows that $\mu_f\in \mathcal M_{\mathcal U}(\widehat G)$ and, so, $[\mathcal X]\subseteq \mathcal M_{\mathcal U}(\widehat G)$.\\
    
    Suppose now that $[\mathcal X]\subseteq \mathcal M_\mathcal U(\widehat G)$. It suffices to show that for every $f\in L_0^2(X,\mu)$, every $p\in \mathcal U$, and every $h\in G$, $\langle T^pf, T^h f\rangle=0$. 
    Since $[\mathcal X]\subseteq \mathcal M_\mathcal U(\widehat G)$, we have that $\mu_f\in \mathcal M_\mathcal U(\widehat G)$ and so, for any $p\in \mathcal U$ and any $h\in G$,
    $$
\langle T^pf,T^{h}f\rangle=\langle T^{-h+p}f,f\rangle=\plimgG{(-h+p)}{g}{G}\langle T^gf,f\rangle=\plimgG{(-h+p)}{g}{G}\widehat{\mu_f}(g)=0.
    $$
    Thus, $\mathcal X$ is $\mathcal U$-mixing.
\end{proof}

The next result is a partial generalization of the classical fact that spectral disjointness implies Furstenberg disjointness.

\begin{theorem}[cf. {\cite[Theorem 2]{Lemanczyk_Parreau}}]\label{SpectralDisjointnessImpliesMultiplier}
    Let $I \subseteq \mathcal{M}(\widehat{G})$ be an $L$-ideal and let $\mathcal{F}_I$ denote the class of systems whose maximal spectral type contains a representative in $I$.
    If $\mathcal{X}$ is a system whose maximal spectral type is mutually singular with every member of $I$, then $\mathcal{X} \in \mathscr{M}(\mathcal{F}_I^\perp)$.
\end{theorem}
\begin{proof}
    Let $\mathcal{Z} \in \mathcal{F}_I$ and $\mathcal{Y} \in \mathcal{F}_I^\perp$ both be arbitrary.
    Let $\eta$ be a joining of $\mu$ and $\nu$, let $\lambda$ be a joining of $\eta$ and $\rho$, and let $f_1\in L^\infty(X,\mu)$, $f_2\in L^\infty(Y,\nu)$, $f_3\in L^\infty(Z,\rho)$.
    Notice that the projections of $\lambda$ to $X\times Z$ and $Y\times Z$ are joinings of $\mathcal X$ and $\mathcal Z$ and $\mathcal Y$ and $\mathcal Z$, respectively, hence each of them is the corresponding product measure. 
    We will show that 
    \begin{equation*}
    \int_{X\times Y\times Z}f_1(x)f_2(y)f_3(z)\text{d}\lambda(x,y,z)=\int_{X\times Y}f_1(x)f_2(y)\text{d}\eta(x,y)\int_Zf_3(z)\text{d}\rho(z).
    \end{equation*} 
    If $f_2=\int_Yf_2\text{d}\nu$, we have that
    \begin{multline*}
    \int_{X\times Y\times Z}f_1(x)f_2(y)f_3(z)\text{d}\lambda(x,y,z)=\int_Y f_2\text{d}\nu\int_{X\times Y\times Z}f_1(x)f_3(z)\text{d}\lambda(x,y,z)\\
    =\int_X f_1\text{d}\mu\int_Yf_2\text{d}\nu\int_Zf_3\text{d}\rho=\int_{X\times Y}f_1(x)f_2(y)\text{d}\eta(x,y)\int_Zf_3(z)\text{d}\rho(z).
    \end{multline*}
Since
\begin{multline*}
\int_{X\times Y\times Z}f_1(x)f_2(y)f_3(z)\text{d}\lambda(x,y,z)=\\
\int_{X\times Y\times Z}f_1(x)\left(f_2(y)-\int_Yf_2\text{d}\nu\right)f_3(z)\text{d}\lambda(x,y,z)+\int_Y f_2\text{d}\nu\int_{X\times Y\times Z}f_1(x)f_3(z)\text{d}\lambda(x,y,z),
\end{multline*}
we can assume without loss of generality that $f_2\neq 0$ and that  $\int_Y f_2\text{d}\nu=0$.
By employing similar arguments we can further assume that $f_1,f_3$ are non-zero and satisfy $\int_X f_1\text{d}\mu = \int_Zf_3\text{d}\rho=0$.\\
 We now observe that 
    \begin{equation}
        \langle (S^g\times R^g)f_2(y)f_3(z),f_2(y)f_3(z)\rangle = \langle S^gf_2,f_2\rangle\langle R^gf_3,f_3\rangle.
    \end{equation}
    Consequently, the spectral measure $\sigma_{(f_2,f_3)}$ of $f_2(y)f_3(z)$ is the convolution of the spectral measures $\sigma_{f_2}$ and $\sigma_{f_3}$ of $f_2$ and $f_3$ respectively. 
    Since $\sigma_{f_3} \in I$, we see that $\sigma_{(f_2,f_3)} \in I$, so the vectors $f_1,f_2\otimes f_3\in L^2(\lambda)$ are orthogonal. Thus,
    \[
\int_{X\times Y\times Z}f_1(x)f_2(y)f_3(z)\text{d}\lambda(x,y,z)=\langle f_1,f_2\otimes f_3 \rangle_{L^2(\lambda)}=0
=\int_{X\times Y}f_1(x)f_2(y)\text{d}\eta(x,y)\int_Zf_3(z)\text{d}\rho(z),
    \]
    completing the proof. 
\end{proof}

\begin{remark}
    Theorem \ref{SpectralDisjointnessImpliesMultiplier} applied to $I = \mathcal{M}_c(\widehat{G})$ recovers the classical fact that $\mathcal{K} \subseteq \mathscr{M}(\mathcal{W}^\perp)$.
    Theorem \ref{SpectralDisjointnessImpliesMultiplier} applied to $I = \mathcal{M}_{\mathcal{U}}(\widehat{G})$ tells us that if $(L^2(X,\mu),T)$ is generated by $\mathcal{U}$, then $\mathcal{X}$ is a multiplier for the class of $\mathcal{U}$-mixing systems.
    A special case of this is that rigid systems are multipliers for the class of mildly mixing systems.
    These results are all also corollaries of Theorem \ref{thm:ParreauFactorsAreMultipliers}, but the next result is only a corollary of Theorem \ref{SpectralDisjointnessImpliesMultiplier} and not of Theorem \ref{thm:ParreauFactorsAreMultipliers}.
\end{remark}
As a corollary, we recover the following result of Lema\'nczyk.
\begin{corollary}[cf. {\cite{Lemanczyk_IET,Lemanczyk_Parreau}}]\label{cor:SingularSpectrumMultipliers}
    The class of systems with singular spectrum is contained in $\mathscr{M}(\mathcal{A}^\perp)$.
\end{corollary}

The following theorem is similar to a result that appears in \cite{Lemanczyk_IET,Lemanczyk_Parreau}.
\begin{theorem}\label{thm:ParreauFactorsAreMultipliers}
    Let $\mathcal{U} \subseteq G^*$ be nonempty.
    If $\mathcal{X}$ is a $\mathcal{U}$-generated system and $\mathcal{Z}$ is a $\mathcal{U}$-mixing system, then $\mathcal{X} \in \mathscr{M}(\{\mathcal{Z}\}^\perp)$.
\end{theorem}

\begin{proof}
    Our proof is a generalization of Parreau's proof \cite{ParreauFactor} that $\mathcal{X}\perp\mathcal{Z}$.
    Let $\mathcal{Y} \in \mathcal{Z}^\perp$ be arbitrary.
    Let $\eta$ be a joining of $\mu$ and $\nu$, and let $\lambda$ be a joining of $\eta$ and $\rho$.
    Since $\mathcal{Y}\perp\mathcal{Z}$, we see that the projection of $\lambda$ onto $Y\times Z$ is a joining of $\nu$ and $\rho$, hence it is $\nu\otimes\rho$.
    Consequently, $\lambda$ is also a joining of $\mu$ and $\nu\otimes\rho$, so let $J:L^2(X,\mu)\rightarrow L^2(Y\times Z,\nu\otimes\rho)$ be the corresponding Markov operator.
    For $p \in \mathcal{U}$, let $\mu_p = \plimgG{p}{g}{G}(\text{Id},T^{-g})_*\mu = \plimgG{p}{g}{G}(\text{Id}\times T^{-g})_*\Delta_\mu$, and observe that $T^p:L^2(X,\mu)\rightarrow L^2(X,\mu)$ is the Markov operator corresponding to the joining $\mu_p$.
    We want to show that $\lambda = \eta\otimes\rho$, so it suffices to show that for all $p_1,\cdots,p_n \in \mathcal{U}, f_1,\cdots,f_n \in L^\infty(X,\mu), h \in L^\infty(Y,\nu),$ and $k \in L^\infty(Z,\rho)$ with $\int_Zkd\rho = 0$, we have

    \begin{equation}\label{eq:MultiplierDisjointnessGoal}
        \int_{X\times Y\times Z}\left(\prod_{j = 1}^nT^{p_j}f_j(x)\right)h(y)k(z)d\lambda(x,y,z) = 0.
    \end{equation}

    Fix $n \ge 1$ and $\vec{p} = (p_1,\cdots,p_n) \in \mathcal{U}^n$.
    We define a probability measure $\gamma_{\vec{p}}$ on $X\times X^n\times Y\times Z$ as follows.
    Firstly, take one copy of $(X,\mathscr{B},\mu,(T^g)_{g \in G})$ called the base coordinate.
    Then take $n$ additional copies of $(X,\mathscr{B},\mu,(T^g)_{g \in G})$ and attach the $j^{\text{th}}$-copy to the base coordinate via the self-joining $\mu_{p_j}$.
    Finally, attach one copy of $Y\times Z$ to the base coordinate via the joining $\lambda$.
    We take the relatively independent joining of these joinings over the base.
    By construction of $\gamma_{\vec{p}}$, for any bounded functions $f,f_1,\cdots,f_n \in L^\infty(X,\mu), h \in L^\infty(Y,\nu),$ and $k \in L^\infty(Z,\rho)$, we have

    \begin{alignat}{2}
        &\int_{X^{n+1}\times Y\times Z}f(x)\left(\prod_{j = 1}^nf_j(x_j)\right)h(y)k(z)d\gamma_{{\vec p}}(x,x_1,\cdots,x_n,y,z)\label{eq:RelIndJoining}\\
        =&\int_{X}f(x)\left(\prod_{j = 1}^nT^{p_j}f_j(x)\right)J^*(h(y)k(z))(x)d\mu(x)\notag\\
        =&\int_{X\times Y\times Z}f(x)\left(\prod_{j = 1}^nT^{p_j}f_j(x)\right)h(y)k(z)d\lambda(x,y,z).\notag      
    \end{alignat}

    We will now prove Equation \ref{eq:MultiplierDisjointnessGoal} by induction on $n$.
    We begin with the base case of $n = 1$.
    Let $p \in \mathcal{U}, f \in L^\infty(X,\mu), h \in L^\infty(Y,\nu),$ and $k \in L^\infty(Z,\rho)$ be such that $\int_Zkd\rho = 0$.
    We see that

    \begin{alignat}{2}
        \int_{X\times Y\times Z}T^pf(x)h(y)k(z)d\lambda(x,y,z) &= p-\lim_{g \in G}\int_{X\times Y\times Z}f(x)S^{-g}h(y)R^{-g}k(z)d\lambda(x,y,z)\\
        &= p-\lim_{g \in G}\int_{Y\times Z}(Jf)(y,z)S^{-g}h(y)R^{-g}k(z)d\nu(y)d\rho(z) = 0.\notag
    \end{alignat}

    Now assume that Equation \eqref{eq:MultiplierDisjointnessGoal} holds for some $n = n_0 \ge 1$, and we will show that it also holds for $n+1$.
    Fix some $\vec{p} = (p_1,\cdots,p_n) \in \mathcal{U}^n$ and let $\widetilde{\gamma}_{\vec{p}}$ denote the projection of $\gamma_{{\vec p}}$ onto $X^n\times Y\times Z$ in which we forget the base coordinate.
    Since $\gamma_{{\vec p}}$ is a joining of $\mu$ and $\widetilde{\gamma}_{{\vec p}}$, let $M:L^2(X,\mu)\rightarrow L^2(X^n\times Y\times Z,\widetilde{\gamma}_{{\vec p}})$ denote the corresponding Markov operator.
    We see that for any $f_1,\cdots,f_n \in L^\infty(X,\mu), h \in L^\infty(Y,\nu)$, and $k \in L^\infty(Z,\rho)$ with $\int_Zkd\rho = 0$, we have

    \begin{alignat}{2}
        \int_{X^n\times Y\times Z}\left(\prod_{j = 1}^nf_j(x_j)\right)h(y)k(z)d\widetilde{\gamma}_{{\vec p}} &= \int_{X^{n+1}\times Y\times Z}1(x)\left(\prod_{j = 1}^{n}f_j(x_j)\right)h(y)k(z)d\gamma_{{\vec p}}\\
        &= \int_{X\times Y\times Z}\left(\prod_{j = 1}^nT^{p_j}f_j(x)\right)h(y)k(z)d\lambda(x,y,z) = 0,\notag
    \end{alignat}
    hence $\widetilde{\gamma}_{\vec{p}} = \gamma'_{\vec{p}}\otimes\rho$ for some probability measure $\gamma'_{\vec{p}}$ on $X^{n}\times Y$.

    We now see that for any $p_1,\cdots,p_n,p_{n+1} \in \mathcal{U}, f_1,\cdots,f_{n+1} \in L^\infty(X,\mu), h \in L^\infty(Y,\nu),$ and $k \in L^\infty(Z,\rho)$ with $\int_Zkd\rho = 0$, we may define $\vec{p} = (p_1,\cdots,p_n)$ and observe that

    \begin{alignat*}{2}
        &\int_{X\times Y\times Z}\left(\prod_{j = 1}^{n+1}T^{p_j}f_j(x)\right)h(y)k(z)d\lambda(x,y,z)\\
        =&\int_{X^{n+1}\times Y\times Z}T^{p_{n+1}}f_{n+1}(x_{n+1})\left(\prod_{j = 1}^nf_j(x_j)\right)h(y)k(z)d\gamma_{{\vec p}}\notag\\
        = &p_{n+1}-\lim_{g \in G}\int_{X^{n+1}\times Y\times Z}f_{n+1}(x_{n+1})\left(\prod_{j = 1}^nT^{-g}f_j(x_j)\right)S^{-g}h(y)R^{-g}k(z)d\gamma_{{\vec p}}\notag\\
        =& p_{n+1}-\lim_{g \in G}\int_{X^n\times Y\times Z}Mf_{n+1}(x_1,\cdots,x_n,y,z)\left(\prod_{j = 1}^nT^{-g}f_j(x_j)\right)S^{-g}h(y)R^{-g}k(z)d\widetilde{\gamma}_{{\vec p}} = 0,
    \end{alignat*}
    where the last equality follows from the fact that $\widetilde{\gamma}_{{\vec p}} = \gamma'_{{\vec p}}\otimes\rho$ and that $\mathcal{Z}$ is $\mathcal{U}$-mixing.
\end{proof}

\begin{corollary}\label{cor:MainMultiplierResult}
    \ 
    \begin{enumerate}[(i)]
        \item $\overline{\mathcal{G}_\mathcal{U}} \subseteq \mathscr{M}(\mathcal{S}_\mathcal{U}^\perp)$, and in particular, $\overline{\mathcal{G}} \subseteq \mathscr{M}(\mathcal{S}^\perp)$.
        
        \item $\mathcal{R} \subseteq \overline{\mathcal{G}_{E(G^*)}} \subseteq \mathscr{M}(\mathcal{M}^\perp)$.

        \item If $\mathcal{X}$ is partially rigid, then $\mathcal{X} \in \mathscr{M}(\mathcal{S}^\perp)$.
    \end{enumerate}
\end{corollary}

\begin{proof}[Proof of (i)]
    It is immediate from Theorem \ref{thm:ParreauFactorsAreMultipliers} that $\mathcal{G}_\mathcal{U} \subseteq \mathscr{M}(\mathcal{S}_{\mathcal{U}}^\perp)$.
    Since $\mathscr{M}(\mathscr{D}^\perp)$ is a characteristic class for any class $\mathscr{D}$, it must contain the smallest characteristic class generated by $\mathcal{G}_\mathcal{U}$, which we saw to be $\overline{\mathcal{G}_{\mathcal{U}}}$ in Remark \ref{rem:QuasiCharacteristicClass}.
    The latter result follows after recalling that $\mathcal{G} = \mathcal{G}_{G^*}$ and $\mathcal{S} = \mathcal{S}_{G^*}$.
\end{proof}

\begin{proof}[Proof of (ii)]
    After recalling that $\mathcal{R} = \bigcup_{p \in E(G^*)}\mathcal{G}_p \subseteq \mathcal{G}_{E(G^*)}$ and $\mathcal{M} = \mathcal{S}_{E(G^*)}$, the desired result follows from Part $(i)$.
\end{proof}

\begin{proof}[Proof of (iii)]
    Theorem \ref{thm:PartiallyRigidFiniteFibers} tells us that if $\mathcal{X}$ is $\alpha$-rigid along some $p \in G^*$, then $\mathcal{X}$ is a finite extension of $\mathcal{X}^p$, so the desired result follows from Part (i) and Corollary \ref{cor:ClosedUnderCompactExtensions}.    
\end{proof}
%%%%%%%%%%%%%%%%%%%%%%%%%%%%%%%%%%%%%%%%%%%%%%%%%%%%%%%%%%%%%%%%%%%%%%%%%%%%%%%%%%%
%%%%%%%%%%%%%%%%%%%%%%%%%%%%%%%%%%%%%%%%%%%%%%%%%%%%%%%%%%%%%%%%%%%%%%%%%%%%%%%%%
\section{The van der Corput approach}\label{sec:vdC}
In this section, we present an alternative proof of Theorem \ref{thm:ParreauFactorsAreMultipliers} based on the following variant of the van der Corput trick proved by Bergelson--Zelada in \cite{StronglyMixingPET} in the case that $S=\Z$. 
%The main motivation for this approach is twofold: First, the results in \cite{AlmostMixingOfAllOrders} suggest that any rich-enough additive structure guarantees mixing. Second, the techniques employed in \cite{kanigowski2024multiple} to obtain such a rich additive structure (roughly speaking, the authors of \cite{kanigowski2024multiple} propose that under convenient conditions one can "approximate" an expression of the form 
%$$
%\int_X f_0T^{-n_1}f_1\cdots T^{-n_\ell}f_\ell\text{d}\mu
%$$
%with a related expression of the form 
%$$
%\int_X \tilde f_0T^{-n_1-h_1}\tilde f_1\cdots T^{-n_\ell-h_\ell}\tilde f_\ell\text{d}\mu.)
%$$

\begin{lemma} [Generalized van der Corput trick] \label{3.GeneralVdC}
 Let $S$ be a countable, non-empty set, let $p_1,...,p_\ell\in\beta S$ be ultrafilters, let  $\mathcal H$ be a Hilbert space, and let $j\in\{1,...,\ell\}$. Let
 $$(x_{s_1,...,s_\ell})_{s_1,...,s_\ell\in S}$$
 be norm\text{-}bounded sequences in $\mathcal H$.
If 
\begin{equation}\label{3.LimitIs0}
\plimgG{p_1}{s_1}{S}\cdots\plimgG{p_{j-1}}{s_{j-1}}{S}\plimgG{p_{j}}{r}{S}\plimgG{p_{j}}{t}{S}\plimgG{p_{j+1}}{s_{j+1}}{S}\cdots\plimgG{p_{\ell}}{s_{\ell}}{S}
\langle x_{s_1,...,s_{j-1},r,s_{j+1},...,s_\ell},x_{s_1,...,s_{j-1},t,s_{j+1},...,s_\ell}\rangle=0,
\end{equation}
then for any bounded sequence $(y_{s_1,...,s_{j-1},s_{j+1},...,s_\ell})_{s_1,...,s_{j-1},s_{j+1},...,s_\ell\in S}$ in $\mathcal H$, 
\begin{equation}\label{3."Weak"LimitIs0}
\plimgG{p_1}{s_1}{S}\cdots\plimgG{p_{j-1}}{s_{j-1}}{S}\plimgG{p_{j}}{s_{j}}{S}\plimgG{p_{j+1}}{s_{j+1}}{S}\cdots\plimgG{p_{\ell}}{s_{\ell}}{S}\\
\langle y_{s_1,...,s_{j-1},s_{j+1},...,s_\ell},x_{s_1,...,s_{j-1},s_{j},s_{j+1},...,s_\ell}\rangle=0.
\end{equation}
So, in particular,  \eqref{3.LimitIs0} implies that for any $y\in \mathcal H$,
$$\plimgG{p_1}{s_1}{S}\cdots\plimgG{p_{j-1}}{s_{j-1}}{S}\plimgG{p_{j}}{s_{j}}{S}\plimgG{p_{j+1}}{s_{j+1}}{S}\cdots\plimgG{p_{\ell}}{s_{\ell}}{S}\\
\langle y, x_{s_1,...,s_{j-1},s_{j},s_{j+1},...,s_\ell}\rangle =0,$$
and, hence, 
$$\plimgG{p_1}{s_1}{S}\cdots\plimgG{p_{j-1}}{s_{j-1}}{S}\plimgG{p_{j}}{s_{j}}{S}\plimgG{p_{j+1}}{s_{j+1}}{S}\cdots\plimgG{p_{\ell}}{s_{\ell}}{S}
 x_{s_1,...,s_{j-1},s_{j},s_{j+1},...,s_\ell}=0$$
 weakly. 
\end{lemma}
\begin{proof}
For any $s_1,...,s_\ell\in S$, let 
$$\xi(s_1,...,s_\ell)=\langle y_{s_1,...,s_{j-1},s_{j+1},...,s_\ell},x_{s_1,...,s_{j-1},s_{j},s_{j+1},...,s_\ell}\rangle$$
and, without loss of generality, suppose that 
$$\|y_{s_1,...,s_{j-1},s_{j+1},...,s_\ell}\|\leq 1\text{ and }\|x_{s_1,...,s_{j-1},s_{j},s_{j+1},...,s_\ell}\|\leq 1.$$
Observe that for any $N\in\N$,
\begingroup
\allowdisplaybreaks
\begin{multline*}
|\plimgG{p_1}{s_1}{S}\cdots\plimgG{p_{j-1}}{s_{j-1}}{S}\plimgG{p_{j}}{s_{j}}{S}\plimgG{p_{j+1}}{s_{j+1}}{S}\cdots\plimgG{p_{\ell}}{s_{\ell}}{S}\xi(s_1,...,s_\ell)|^2\\
=|\frac{1}{N}\sum_{k=1}^N\plimgG{p_1}{s_1}{S}\cdots\plimgG{p_{j-1}}{s_{j-1}}{S}\plimgG{p_{j}}{s^{(k)}_{j}}{S}\plimgG{p_{j+1}}{s_{j+1}}{S}\cdots\plimgG{p_{\ell}}{s_{\ell}}{S}\xi(s_1,...,s_{j-1},s_j^{(k)},s_{j+1},...,s_\ell)|^2\\
=\plimgG{p_1}{s_1}{S}\cdots\plimgG{p_{j-1}}{s_{j-1}}{S}\plimgG{p_{j}}{s^{(1)}_{j}}{S}\cdots\plimgG{p_j}{s_j^{(N)}}{S}\plimgG{p_{j+1}}{s_{j+1}}{S}\cdots\plimgG{p_{\ell}}{s_{\ell}}{S}|\frac{1}{N}\sum_{k=1}^N\xi(s_1,...,s_{j-1},s_j^{(k)},s_{j+1},...,s_\ell)|^2\\
\leq \plimgG{p_1}{s_1}{S}\cdots\plimgG{p_{j}}{s^{(1)}_{j}}{S}\cdots\plimgG{p_j}{s_j^{(N)}}{S}\plimgG{p_{j+1}}{s_{j+1}}{S}\cdots\plimgG{p_{\ell}}{s_{\ell}}{S}
\|y_{s_1,...,s_{j-1},s_{j+1},...,s_\ell}\|^2\|\frac{1}{N}\sum_{k=1}^Nx_{s_1,...,s_{j-1},s_{j}^{(k)},s_{j+1},...,s_\ell}\|^2\\
\leq  \plimgG{p_1}{s_1}{S}\cdots\plimgG{p_{j}}{s^{(1)}_{j}}{S}\cdots\plimgG{p_j}{s_j^{(N)}}{S}\plimgG{p_{j+1}}{s_{j+1}}{S}\cdots\plimgG{p_{\ell}}{s_{\ell}}{S} 
\frac{1}{N^2}\sum_{k,k'=1}^N\langle x_{s_1,...,s_{j-1},s_{j}^{(k)},s_{j+1},...,s_\ell},x_{s_1,...,s_{j-1},s_{j}^{(k')},s_{j+1},...,s_\ell}\rangle\\
\leq \frac{1}{N}+\plimgG{p_1}{s_1}{S}\cdots\plimgG{p_{j-1}}{s_{j-1}}{S}\plimgG{p_{j}}{s^{(1)}_{j}}{S}\cdots\plimgG{p_j}{s_j^{(N)}}{S}\plimgG{p_{j+1}}{s_{j+1}}{S}\cdots\plimgG{p_{\ell}}{s_{\ell}}{S}\\
\frac{2}{N^2}\sum_{k=2}^N\sum_{k'=1}^{k-1}\text{Re}(\langle x_{s_1,...,s_{j-1},s_{j}^{(k)},s_{j+1},...,s_\ell},x_{s_1,...,s_{j-1},s_{j}^{(k')},s_{j+1},...,s_\ell}\rangle)=\frac{1}{N}.
\end{multline*}
\endgroup
Letting $N\rightarrow\infty$, we see that \eqref{3."Weak"LimitIs0} holds.
\end{proof}

Theorem \ref{thm:ParreauFactorsAreMultipliers} is a corollary of Lemma \ref{lem:IndependentJoiningsViaVdC}, which was also implicitly proven in Section \ref{sec:MultipliersMainResults} with different methods.

\begin{lemma}\label{lem:IndependentJoiningsViaVdC}
Let $(G,+)$ be a countable abelian group, let $\mathcal X$, $\mathcal Y$, and $\mathcal Z$ be measure-preserving systems, and let $\mathcal U\subseteq G^*$ be a non-empty family of ultrafilters. Suppose that $\mathcal Z$ is $\mathcal U$-mixing and that $\mathcal Y\perp \mathcal Z$. Let $\sigma$ be a joining of $\mu$ and $\nu$.
Then, for any $p_1,...,p_\ell\in \mathcal U$,  any $f_1,...,f_\ell\in L^\infty(X,\mu)$, any $F\in L^\infty(Y,\nu)$,   any $H\in L^{\infty}(Z,\rho)$, and any joining $\lambda$ of $\sigma$ and $\rho$, one has 
\begin{equation}\label{eq:IndependentJoiningVdC}
\int_{X\times Y\times Z}H(z)F(y)\prod_{j=1}^\ell T^{p_j}f_j(x)\text{d}\lambda(x,y,z)
=\int_{X\times Y} F(y) \prod_{j=1}^\ell T^{p_j}f_j(x)\text{d}\sigma(x,y)\int_Z H(z)\text{d}\rho(z).
\end{equation}
\end{lemma}
\begin{proof}
It suffices to show that \eqref{eq:IndependentJoiningVdC} holds when $\int_Z H\text{d}\rho=0$ and each of  $H$, $F$, and $f_1,\dots,f_\ell$ is real-valued.
To do this, we proceed by induction on $\ell$.\\

$\bullet$ \textit{ Case $\ell=1$.} When $\ell=1$, we have 
\begin{multline}\label{eq:VdCK=1.1}
\int_{X\times Y\times Z} T^{p_1}f_1(x)F(y)H(z)\text{d}\lambda(x,y,z)=
\plimgG{p_1}{g_1}{G}\int_{X\times Y\times Z} T^{g_1}f_1(x)F(y)H(z)\text{d}\lambda(x,y,z) \\
=\plimgG{p_1}{g_1}{G}\int_{X\times Y\times Z} f_1(x)S^{-g_1}F(y)R^{-g_1}H(z)\text{d}\lambda(x,y,z)
\end{multline}
By Lemma \ref{3.GeneralVdC} applied with $\ell=2$ and $p_2=\{A\subseteq G\,|\,0\in A\}$ (so, in particular, $p_1=p_1+p_2$), the expression in \eqref{eq:VdCK=1.1} equals zero provided that 
\begin{multline}\label{eq:VdCk=1.2}
\plimgG{p_1}{r}{G}\plimgG{p_1}{t}{G}\plimgG{p_2}{g}{G}\int_{X\times Y\times Z} S^{-r-g}F(y)R^{-r-g}H(z)S^{-t-g}F(y)R^{-t-g}H(z)\text{d}\lambda(x,y,z)\\
=\plimgG{p_1}{r}{G}\plimgG{p_1}{t}{G}\int_{X\times Y\times Z} S^{-r}F(y)R^{-r}H(z)S^{-t}F(y)R^{-t}H(z)\text{d}\lambda(x,y,z)=0.
\end{multline}
Since $\mathcal{Y}\perp\mathcal{Z}$, we see that $\lambda$ projects to $\nu\otimes\rho$ on $Y\times Z$.
Since $\mathcal Z$ is $\mathcal U$-mixing, we obtain
\begin{multline}
    \plimgG{p_1}{r}{G}\plimgG{p_1}{t}{G}\int_{X\times Y\times Z} S^{-r}F(y)R^{-r}H(z)S^{-t}F(y)R^{-t}H(z)\text{d}\lambda(x,y,z)\\
    = \plimgG{p_1}{r}{G}\plimgG{p_1}{t}{G}\left(\int_{Y}FS^{-t+r}F\text{d}\nu\int_Z HR^{-t+r}H\text{d}\rho\right)
    =\left(\plimgG{p_1}{r}{G}\plimgG{p_1}{t}{G}\int_{Y}FS^{-t+r}F\text{d}\nu\right)\left(\int_Z H\text{d}\rho\right)^2=0.
\end{multline}
Thus, \eqref{eq:IndependentJoiningVdC} holds when $\ell=1$.\\ 

$\bullet$ \textit{ Case $\ell=k+1$.} Let $k\in\N$ and suppose now that \eqref{eq:IndependentJoiningVdC} holds for any $\tilde f_1,...,\tilde f_k\in L^\infty(X,\mu)$, any $\tilde F\in L^\infty(Y,\nu)$, and any $\tilde H\in L^\infty(Z,\rho)$. 
Set $\ell=k+1$ and  take $f_1,...,f_\ell\in L^\infty(X,\mu)$, $F\in L^\infty(Y,\nu)$, and $H\in L^\infty(Z,\rho)$ such that $\int_Z H\text{d}\rho=0$.
Note that, by Lemma \ref{3.GeneralVdC} and the $T\times S\times R$ invariance of $\lambda$, 
\begin{multline*}
0
=  \int_{X\times Y\times Z} \prod_{j=1}^\ell T^{p_j}f_j(x)(F(y)H(z))\,\text{d}\lambda(x,y,z)\\
= \plimgG{p_\ell}{g_\ell}{G}\cdots \plimgG{p_1}{g_1}{G} \int_{X\times Y\times Z}\prod_{j=1}^\ell T^{-\sum_{i\neq j}g_i}f_jS^{-\sum_{i=1}^\ell g_i}FR^{-\sum_{i=1}^\ell g_i}H\,\mathrm{d}\lambda,
\end{multline*}
provided that
\begin{multline*}
0
= \plimgG{p_\ell}{r_\ell}{G}\plimgG{p_\ell}{t_\ell}{G} \plimgG{p_k}{g_k}{G}\cdots\plimgG{p_1}{g_1}{G} \\
\int_{X\times Y\times Z} R^{-\sum_{i=1}^k g_i}\Big(HR^{-t_\ell+r_\ell}H\Big)S^{-\sum_{i=1}^k g_i}\Big( FS^{-t_\ell+r_\ell}F\Big)
\left(\prod_{j=1}^kT^{-\sum_{\substack{i=1\\ i\neq j}}^k g_i}(f_jT^{-t_\ell+r_\ell}f_j)\right)\text{d}\lambda.
\end{multline*} 
For each $g\in G$, set
\[
\tilde H_g:=HR^{-g}H,
\;\tilde F_g := F S^{-g}F\;,
\qquad
\tilde f_{j,g} := f_jT^{-g}f_j,\quad j=1,\dots,k.
\]
By the inductive hypothesis and since $\mathcal Z$ is $\mathcal U$-mixing, we obtain
\begin{multline*}
\plimgG{p_\ell}{r_\ell}{G}\plimgG{p_\ell}{t_\ell}{G}\plimgG{p_k}{g_k}{G}\cdots\plimgG{p_1}{g_1}{G}
\int_{X\times Y\times Z} R^{-\sum_{i=1}^k g_i}\tilde H_{t_\ell-r_\ell}S^{-\sum_{i=1}^kg_i}\tilde F_{t_\ell-r_\ell}\left(\prod_{j=1}^kT^{-\sum_{\substack{i=1\\ i\neq j}}^k g_i}\tilde f_{j,(t_\ell-r_\ell)}\right)\text{d}\lambda\\
= \plimgG{p_\ell}{r_\ell}{G}\plimgG{p_\ell}{t_\ell}{G}
\int_{X\times Y\times Z} \tilde H_{t_\ell-r_\ell}(z)\tilde F_{t_\ell-r_\ell}(y)\prod_{j=1}^kT^{p_j}\tilde f_{j,(t_\ell-r_\ell)}(x)\text{d}\lambda(x,y,z)\\
=\plimgG{p_\ell}{r_\ell}{G}\plimgG{p_\ell}{t_\ell}{G}
\left(\int_{Z} HR^{-t_\ell+r_\ell}H\text{d}\rho(z)\right)\left(\int_{X\times Y}\tilde F_{t_\ell-r_\ell}(y)\prod_{j=1}^kT^{p_j}\tilde f_{j,(t_\ell-r_\ell)}(x)\text{d}\sigma(x,y)\right)=0,
\end{multline*}
which completes the induction.
\end{proof}
%%%%%%%%%%%%%%%%%%%%%%%%%%%%%%%%%%%%%%%%%%%%%%%%%%%%%%%%%%%%%%%%%%%%
\section{Questions and Future Work}\label{sec:Questions}
The study of the spectral properties of Interval Exchange Transformations has long been a topic of interest.
In 1967 Katok and St{\"e}pin \cite{ApproximationsInErgodicTheory} showed that almost every 3-IET has simple spectrum and is weakly mixing, and in 1980 Katok \cite{IETsAreNotMixing} showed that no IET is strongly mixing.
In 1984 Veech \cite{MetricTheoryOfIETs1} showed that almost every IET is totally ergodic, rigid, rank 1, and has simple spectrum.
While Veech had also shown that many IETs are weakly mixing, it was a breakthrough when Avila and Forni \cite{AEIETIsWM} showed that almost every IET is either weakly mixing or an irrational rotation.
As discussed in \cite[Page 15]{Lemanczyk2009Spectral} and \cite{WMBddedCutParameter}, it remains an open question as to whether or not every IET has singular spectrum. 
The following related question is also open.

\begin{question}\label{Q:IETsSpectrallyDisjoint}
    Is every IET spectrally disjoint from every strongly mixing system?
\end{question}

We observe that a positive answer to Question \ref{Q:IETsSpectrallyDisjoint} would imply that every IET has singular spectrum.

As we showed in Example \ref{ex:Chacon}, Chacon's transformation provides a natural example of a $p$-generated system which is not rigid.  
In light of Theorem \ref{thm:PartiallyRigidFiniteFibers}, we see that IETs are another natural candidate for such systems.
\begin{question}\label{Q:AreIETsParreau}
    Suppose that $\mathcal{X}$ is an IET that is partially rigid along some $p \in \mathbb{Z}^*$.
    \begin{enumerate}[(i)]
        \item Is $\mathcal{X}$ a $p$-generated system?

        \item More generally, is $\mathcal{X}$ a ${\mathcal{U}}$-generated system for some ${\mathcal{U}}$?
    \end{enumerate}
\end{question}

We worked with abelian groups because we required $T^p$ to be a normal operator to use the results of Section \ref{ssec:ImageKernelDecomposition}.
Furthermore, if $G$ is abelian, then $T^p$ is the Markov operator associated to a self-joining of $\mathcal{X}$, and this fact was used in the proof of Theorem \ref{thm:ParreauFactorsAreMultipliers}.
Nonetheless, the definitions of $\mathcal{U}$-mixing, $\mathcal{U}$-generated system, and $\mathcal{U}$-generated factor extend to the setting of countably infinite groups without any changes. 
Thus, we are left with the following questions.

\begin{question}\label{Q:NonCommutativeParreau}
    Let $G$ be a countably infinite noncommutative group, $\mathcal{X}$ a measure-preserving $G$-system, and $\emptyset \neq \mathcal{U} \subseteq G^*$.
    \begin{enumerate}[(i)]
        \item Is the factor $\mathcal{X}^\mathcal{U}$ also a $\mathcal{U}$-generated system? In other words, do we still have  $\mathcal X^{\mathcal U}=(\mathcal X^{\mathcal U})^\mathcal U$?

        \item If $\mathcal{X}$ is a $\mathcal{U}$-generated system, is $\mathcal{X}$ disjoint from all $\mathcal{U}$-mixing systems?
    \end{enumerate}
\end{question}

It seems unlikely that our proof of Theorem \ref{thm:ParreauFactorsAreMultipliers} presented in Section \ref{sec:MultipliersMainResults} will help in the noncommutative setting, but the methods of Section \ref{sec:vdC} may help study the situation when $G$ is nilpotent.

In the introduction we saw that $\mathscr{M}(\mathcal{E}^\perp) \subsetneq \mathcal{E}^\perp$ and $\mathscr{M}(\mathcal{W}^\perp) \subsetneq \mathcal{W}^\perp$.
We have not yet determined whether or not $\mathscr{M}(\mathcal{M}^\perp) \subsetneq \mathcal{M}^\perp, \mathscr{M}(\mathcal{S}^\perp) \subsetneq \mathcal{S}^\perp,$ or $\mathscr{M}(\mathcal{A}^\perp) \subsetneq \mathcal{A}^\perp$, which leads us to the following:

\begin{question}\label{Q:Multipliers}
    Does there exist a class of systems $\mathscr{D} \subseteq \mathcal{W}$ that is closed under factors, countable direct products, and inverse limits, for which we also have $\mathscr{M}(\mathscr{D}^\perp) = \mathscr{D}^\perp$?
\end{question}

We have seen that $\mathcal{U}$-mixing encapsulates many but not all notions of mixing that are of interest. 
Another well-studied notion of mixing in the setting of $\mathbb{Z}$-systems is that of light mixing.
The $\mathbb{Z}$-system $\mathcal{X}$ is \textbf{lightly mixing} if for every $A,B \in \mathscr{B}^+$, we have $\liminf_{n\rightarrow\infty}\mu(A\cap T^nB) > 0$.
Light mixing sits strictly in between mild mixing and partial mixing.
Furthermore, Kingman \cite{LightlyMixingIsClosedUnderCountableProducts} showed that the class of lightly mixing systems is closed under direct products and inverse limits, and Adams \cite{DenseLMixingIsNotLMixing} showed that light mixing on a dense algebra of sets does not necessarily imply light mixing.
We are led to the following:

\begin{question}\label{q:LightMixingIsNotUMixing}
    Is there any $\mathcal{U} \subseteq \mathbb{Z}^*$ for which $\mathcal{S}_\mathcal{U}$ coincides with the class of lightly mixing systems?
\end{question}
%%%%%%%%%%%%%%%%%%%%%%%%%%%%%%%%%%%%%%%%%%%%%%%%%%%%%%%%%%%%%%%%%%%%
\bibliographystyle{abbrv}
\begin{center}
	\bibliography{references}
\end{center}
\end{document}